\documentclass[preprint,12pt]{elsarticle}

\usepackage{amssymb}
\usepackage{amsmath}
\usepackage{graphicx} 
\usepackage{amssymb}
\usepackage[utf8]{inputenc}
\usepackage{comment}
\usepackage{multirow}
\usepackage{subfig}
\usepackage{float}
\usepackage{amsmath}
\usepackage{amsthm}
\usepackage{listings}
\usepackage{amsfonts}
\usepackage{pgf,tikz,pgfplots}
\pgfplotsset{compat=1.15}
\usepackage{mathrsfs}
\usepackage{comment}
\usetikzlibrary{arrows}
\usepackage[normalem]{ulem}
\usepackage{appendix}
\usepackage{bm}
\usepackage{color}
\usepackage{hyperref}
\usepackage{placeins}
\usepackage{tabularx}
\usepackage[ruled,vlined]{algorithm2e}

\newtheorem{Theorem}{Theorem}

\newtheorem{Definition}{Definition}
\newtheorem{Remark}{Remark}
\newtheorem{Corollary}{Corollary}

\newcommand{\irene}[1]{{{ \color{black} #1}}}

\newcommand{\refuno}[1]{\textcolor{red}{#1}}

\journal{Journal of Computational and Applied Mathematics }

\begin{document}

\begin{frontmatter}



\title{Well-balanced POD-based reduced-order models for finite volume approximation of hyperbolic balance laws}

\author{I. Gómez-Bueno\corref{cor1}\fnref{label2}}
\ead{igomezbueno@uma.es}
\cortext[cor1]{Corresponding author.}
\address[label2]{Dpto. Matemática Aplicada. Universidad de Málaga, 29010, Málaga, Spain}

\author{E.D. Fernández-Nieto\fnref{label3}}

\author{S. Rubino\fnref{label4}}


\address[label3]{Dpto. Matemática Aplicada I $\&$ IMUS, Universidad de Sevilla, 41012, Sevilla, Spain}

\address[label4]{Dpto. EDAN $\&$ IMUS, Universidad de Sevilla, 41012, Sevilla, Spain}

\begin{abstract}
This paper introduces a reduced-order modeling approach based on finite volume methods for hyperbolic systems, combining Proper Orthogonal Decomposition (POD) with the Discrete Empirical Interpolation Method (DEIM) and Proper Interval Decomposition (PID). Applied to systems such as the transport equation with source term, non-homogeneous Burgers equation, and shallow water equations with non-flat bathymetry and Manning friction, this method achieves significant improvements in computational efficiency and accuracy compared to previous time-averaging techniques. A theoretical result justifying the use of well-balanced Full-Order Models (FOMs) is presented. Numerical experiments validate the approach, demonstrating its accuracy and efficiency. Furthermore, the question of prediction of solutions for systems that depend on some physical parameters is also addressed, and a sensitivity analysis on POD parameters confirms the model's robustness and efficiency in this case.
\end{abstract}


\begin{highlights}
\item Efficient POD-ROM combines DEIM, PID, and FV methods for hyperbolic systems.
\item Proves ROMs inherit well-balanced property from well-balanced FOMs.
\item DEIM strategy enhances efficiency, accuracy, and reduces spurious oscillations.
\item Predictive ROMs validated for parameterized systems via sensitivity analysis.
\item Applied to transport and Burgers equations with source terms and shallow water system with bottom and Manning friction.
\end{highlights}

\begin{keyword}
finite volume method \sep proper orthogonal decomposition \sep reduced order modeling \sep hyperbolic balance laws \sep well-balanced property


\MSC 76M12 \sep 35L60 \sep 65M08

\end{keyword}

\end{frontmatter}



\section{Introduction}

 The numerical resolution of hyperbolic equations for realistic applications could require high computational costs, due to the use of very fine grids to properly capture complex shock phenomena or discontinuities, and also integration over long periods in time. Reduced-Order Models (ROMs) applied to numerical design of complex physics problems are a proper mathematical tool that is wide-spreading in the scientific community in the recent years in order to save computational time without so much loss of accuracy with respect to classical methods such as Finite Difference (FD), Finite Element (FE) or Finite Volume (FV) methods \cite{hesthaven2016certified}.

Among the most popular ROMs approaches, Proper Orthogonal Decomposition (POD) strategy provides optimal (from the energetic point of view) basis or modes to represent the dynamics from a given database (snapshots) obtained by a Full-Order Model (FOM). Onto these reduced bases, a Galerkin projection of the governing equations can be employed to obtain a low-order dynamical system for the basis coefficients. The resulting low-order model is named standard POD-(Galerkin) ROM, which thus consists in the projection of high-fidelity (full-order) representations of physical problems onto low-dimensional spaces of solutions, with a dramatically reduced dimension. These low-dimensional spaces are capable of capturing the dominant characteristics of the solution, their main advantage being that the computations in the low-dimensional space can be done at a reduced computational cost \cite{lumley1967structure,sirovich1987turbulence}. This has led researchers to apply POD-ROMs to a variety of physical and engineering problems, including Computational Fluid Dynamics (CFD) problems in order to model advection-diffusion equations, see e.g. \cite{azaiez2021cure,giere2015supg}, the Navier–Stokes Equations (NSE), see e.g. \cite{baiges2013explicit,ballarin2015,galletti2004low,novo2021error,rebollo2022error,wang2012proper}, and also hyperbolic equations, such as Shallow Water (SW) equations \cite{ahmed2020sampling, dehghan2017use, cstefuanescu2013pod, zeng2022embedded,zokagoa2018pod}. 

Although POD-ROMs can be very computationally efficient and relatively accurate in some flow configurations, they also present several drawbacks, when applied for instance to strongly nonlinear systems. In this direction, several techniques have been presented for hyperbolic systems, such as the dynamic mode decomposition \cite{ahmed2020sampling, schmid2010dynamic}, the Discrete Empirical Interpolation Method (DEIM) \cite{chaturantabut2009discrete,crisovan2019model,fu2018pod}, the tensorial POD \cite{cstefuanescu2014comparison}, and also non-intrusive hyper-reduction techniques \cite{ahmed2020sampling}. Furthermore, to address this issue, a recent paper \cite{solan2023development} uses the time averaging approach in combination with the Proper Interval Decomposition (PID, \cite{ahmed2018stabilized, ijzerman2000signal,zokagoa2018pod}) applied to SW using augmented Riemann solvers. \refuno{Moreover, it should be emphasized that, in the context of hyperbolic problems, reduced order modeling techniques suffer from fundamental limitations when used to predict solutions beyond the training time window. This issue arises due to the highly nonlinear and propagative nature of the solutions to systems of balance laws, which may develop discontinuities or rapidly evolving features whose dynamics are not adequately represented by the dominant modes extracted from limited-time snapshots. Consequently, the reliability of reduced models in extrapolative regimes may be affected, and care should be taken when applying them outside the training interval.
}

In this work, we focus on the use of POD-based ROMs for generic hyperbolic systems with source terms. In particular, we show how the use of DEIM technique to deal with nonlinearities combined with PID outperforms the time averaging approach in combination with PID proposed in \cite{solan2023development}, especially in terms of physical accuracy. 

These PDE systems have been extensively used to model waves caused by small disturbances in steady solutions. A notable example is tsunami waves in the ocean. A numerical method that can solve all steady states of the PDE exactly, or with improved accuracy, or at least a significant subset of them, is referred to as \textit{exactly Well-Balanced} (EWB). Additionally, the method is classified as \textit{fully exactly well-balanced} if it exactly preserves all the stationary solutions of the hyperbolic system (see \cite{gomez2021collocation}). The design of well-balanced schemes is a highly active area of research, and many studies focus on this topic, specially for the case of finite volume methods. For example, see \cite{LR96} for the Burgers equation; \cite{abgrall2023new, audusse2004fast, Kurganov2016, ciallella2022arbitrary, michel2017well, noelle2006well} for the shallow water equations; and \cite{abgrall2022hyperbolic,berthon2022very,Castro2020High,   dumbser2024well,klingenberg2019arbitrary} for other hyperbolic systems, among others.

 In this paper we present an original theoretical result of consistency at ROM level, to justify the choice of well-balanced methods to build the FOM, which is independent of the considered hyperbolic balance law. Three examples of well-balanced POD-based methods are presented, for the transport equation with a source term, non-homogeneous Burgers equation and for the SW equations with topography and Manning friction. Finally, following other works focused on reduced order modelling for parametrized problems (see \cite{crisovan2019model}), the question of prediction of solutions for physical parameter-dependent systems is also addressed in this paper, as well a sensitivity analysis to specific parameters of the proposed POD-ROM, such as the number of POD modes and time windows in terms of the errors between FOM and ROM solutions and the CPU times they require.

The rest of the paper is organized as follows: first, we introduce in Section \ref{sec:FOM} the finite volume full-order models for general one-dimensional hyperbolic PDE systems, detailing the numerical schemes and discussing the construction of exactly well-balanced schemes. In Subsection \ref{subsec:FOMs}, we specify the FOMs for three examples of balance laws: the linear transport equation with source term, the Burgers equation with a nonlinear source term, and the shallow water equations with and without Manning friction. In all cases, we use PVM-0 type numerical fluxes \cite{castro2012class}, and for the shallow water case, we additionally apply the HLL scheme \cite{harten1983upstream}.  

Section \ref{sec:POD-ROMs} outlines the strategy for deriving the POD-based ROMs for the FOMs presented in this article. Special attention is given to the treatment of nonlinearities, where we employ the PID method to adapt the basis functions over time by partitioning the total time interval into subintervals, known as time windows. This method is combined with DEIM, achieving improved results compared to other approaches, such as time-averaging, especially in handling discontinuities. In Section \ref{sec:ROM_examples}, we illustrate the above procedure for the three balance law examples introduced in Subsection \ref{subsec:FOMs}. Section \ref{sec:WB_theorem} introduces a theorem proving that if the initial FOM is exactly well-balanced, the associated ROM will also satisfy this property. In Section \ref{sec:prediction}, we briefly address the problem of solution prediction for systems depending on physical parameters, such as the Manning friction coefficient in the shallow water system. Section \ref{sec:numericalexperiments} presents several numerical tests to support the theoretical findings of the study. Finally, concluding remarks are provided in Section \ref{sec:conclusiones}.

\section{Hyperbolic PDE systems: full-order models}\label{sec:FOM}

Let us consider, for ease of presentation, a general one-dimensional hyperbolic system of balance laws of the form
\begin{equation}\label{PDE}
    W_t+F(W)_x=S(W)H_x+R(W), \quad x\in I, \, t>0,
\end{equation}
where $I \subset \mathbb{R}$; the unknown $W(x; t) =
(w_1(x, t), \cdots, w_N(x, t))^T$ takes values in $\Omega$, being $\Omega$ an open convex set of $\mathbb{R}^N$, called set
of states; $F$ is a regular function from $\Omega$ to $\mathbb{R}^N$ called 
flux function; $S$ and $R$ are functions from $\Omega$
 to $\mathbb{R}^N$ and $H(x)$ is a known bounded
function from $I$ to $\mathbb{R}$.

The technical extension to higher dimension (2D/3D) could be performed by following the guidelines given in this work.

 First-order FV numerical methods will be used to discretize system \eqref{PDE}. The domain \( I \) is divided into \( N_x \) computational cells \( I_i = [x_{i-1/2}, x_{i+1/2}] \). For simplicity, a uniform spatial step size \( \Delta x \) is assumed. The center of each cell \( I_i \) is denoted by \( x_i \), where \( x_i = (i - 1/2) \Delta x \), and the inter-cell boundary is located at \( x_{i+1/2} = i \Delta x \). Similarly, \( \Delta t \) denotes the time step size, where \( t^n = (n - 1) \Delta t \) for \( n = 1, 2, \dots \). The cell-averaged piecewise approximation of the solution \( W(x,t) \) in cell \( I_i \) at time \( t^n \) is denoted by \( W_i^n \):
\begin{equation}\label{discrete_i.c.}
W_i^n \approx \frac{1}{\Delta	x}\int_{x_{i-1/2}}^{x_{i+1/2}}W(x,t^n)\,dx,
\end{equation}
and its $j-$th component is denoted by $w_{j,i}^n, \, j=1,\cdots,N$. The mid-point rule is considered to approximate the integrals. \irene{A first-order FV full-order model can be written in terms of a general operator $\mathcal{L}$ as follows:
\begin{equation}\label{met_num_FOM_operador}
	W_i^{n+1}=W_i^n-\frac{\Delta t}{\Delta x} \mathcal{L}\left(W_{i-1}^n,W_i^n,W_{i+1}^n\right).
\end{equation}
}


Moreover, the systems of balance laws of the form \eqref{PDE} admit non-trivial stationary solutions satisfying
\begin{equation}\label{ODE_ss}
	F(W)_x=S(W)H_x+R(W).
\end{equation}

In this paper we consider numerical methods that are able to preserve such stationary solutions in the following sense:

\begin{Definition}\label{def:ewb}
A numerical method  is said to be Exactly Well-Balanced (EWB) for a stationary solution $W^*$ of \eqref{PDE} if the sequence of its cell-averages $\{ W^*_i \}$ (or the sequence of their approximations if a quadrature formula is used to compute them) is an equilibrium of the method.
\end{Definition}
\irene{
The development of well-balanced methods has been extensively investigated, with various strategies proposed in the literature. For example, within the path-conservative framework, a family of paths can be selected to satisfy the generalized Rankine–Hugoniot condition, thereby ensuring the well-balanced property for all stationary solutions—or at least for a significant subset of them (see \cite{castro2017well}). For instance, in the context of the shallow water equations, when the geometric source term is approximated using a family of straight segments, the resulting numerical method is exactly well-balanced for stationary solutions representing water at rest. This corresponds to cases where the velocity is zero, and the free surface elevation remains constant.

Another strategy for constructing well-balanced schemes involves employing well-balanced reconstruction operators, as described in \cite{Castro2020High, gomez2021collocation}. For example, a first-order, exactly well-balanced reconstruction operator can be defined as:
\begin{equation}
    P_i^n(x) = W_i^{*,n}(x) + W_i^n - W_i^{*,n}(x_i),
\end{equation}
where \( W_i^{*,n}(x) \) represents the stationary solution of \eqref{PDE} such that $W_i^{*,n}(x_i) = W_i^n$, since the mid-point rule is used to approximate the integrals.
}

\irene{
With regard to the numerical fluxes, we consider 
 $ \mathbb{F}(W_L, W_R)$, where $\mathbb{F}$ is a  numerical flux  consistent with $F$, i.e., $\mathbb{F}(W,W)=F(W).$ In this work we consider numerical fluxes written under the following viscous form:
\begin{equation} \label{eq:def:numflux}
    \mathbb{F}(W_L, W_R)=\dfrac{1}{2}\left(F(W_L)+F(W_R)\right) - \frac{1}{2}\mathcal{D}(W_L,W_R) \left( W_R - W_L \right),
\end{equation}
where $\mathcal{D}(W_L,W_R)$ is a viscosity matrix, defined in terms of Roe matrix $\mathcal{A}(W_L,W_R)$, verifying
$$
\mathcal{A}(W_L,W_R) (W_R-W_L)=F(W_R)-F(W_L).
$$
Following \cite{castro2012class}, Polynomial Viscosity Matrix (PVM) methods have been considered, where $\mathcal{D}(W_L,W_R)$ is defined by evaluating a polynomial of degree $l$, $p^l(W_L,W_R)(x)$, in a linearized Roe matrix $\mathcal{A}(W_L,W_R)$: 
\begin{equation}
    \mathcal{D}(W_L,W_R)= p^l(W_L,W_R)(\mathcal{A}(W_L,W_R)),
\end{equation}
with 
\begin{equation*}
    p^l(W_L,W_R)(x)=\sum_{k=0}^l \alpha^k(W_L,W_R) x^k.
\end{equation*}
For simplicity, given the sequence of cell-averages $\{W_i\}$, we will use the notation $p^l(W_i, W_{i+1})=p^l_{i+1/2}$, $\alpha^k(W_i,W_{i+1})=\alpha_{i+1/2}^k$ and $\lambda_{j}(W_i,W_{i+1})=\lambda_{j,i+1/2}$. In particular, let $\lambda_{1,i+1/2} < \dots < \lambda_{N,i+1/2}$ be the eigenvalues
of $\mathcal{A}(W_i, W_{i+1})$. For stability demands one needs
\begin{equation} \label{eq:estb:pvm}
     |\lambda_{j,i+1/2}| \leq p^l_{i+1/2}(\lambda_{j,i+1/2})  \leq \frac{\Delta x}{\Delta t}, \quad \forall j=1,\dots, N,
\end{equation}
with the CFL condition
$$
  \frac{\Delta t}{\Delta x}  \max_j |\lambda_{j}(W_i, W_{i+1})| \leq \gamma,
  $$
being $\gamma \leq 1$.

The simplest choice for a PVM method corresponds to the constant polynomial (PVM-0 method):
\begin{equation*}
p^0_{i+1/2}(x)=\alpha^0_{i+1/2},
\end{equation*}
 For a PVM-0 flux, the stability condition (\ref{eq:estb:pvm}) reads as
\begin{equation}
    \max_j |\lambda_{j,i+1/2}| \leq \alpha_{i+1/2}^0 \leq \frac{\Delta x}{\Delta t}.
\end{equation}
The Lax-Friedrichs flux corresponds $\alpha_{i+1/2}^0=\dfrac{\Delta x}{\Delta t}$, the modified Lax-Friedrichs represents the choice $\alpha_{i+1/2}^0= \gamma \dfrac{\Delta x}{\Delta t}$  and the Rusanov numerical flux corresponds to $\alpha_{i+1/2}^0=\max_j |\lambda_{j,i+1/2}|$.

We also consider the HLL numerical flux introduced in \cite{harten1983upstream}, which corresponds to the first degree polynomial
\begin{equation*}
p_{i+1/2}^1(x)=\alpha_{i+1/2}^0+\alpha_{i+1/2}^1 x,
\end{equation*}
where the coefficients $\alpha_{i+1/2}^0$ and $\alpha_{i+1/2}^1$ are defined as
\begin{equation}
    \alpha_{i+1/2}^0=\frac{S_{i+1/2}^R|S_{i+1/2}^L|-S_{i+1/2}^L|S_{i+1/2}^R|}{S_{i+1/2}^R-S_{i+1/2}^L}, \quad \alpha_{i+1/2}^1=\frac{|S_{i+1/2}^R|-|S_{i+1/2}^L|}{S_{i+1/2}^R-S_{i+1/2}^L}.
\end{equation}
Here, $S_{i+1/2}^L$ and $S_{i+1/2}^R$ are, respectively, approximations of the minimum and maximum wave speed propagation.
}

\subsection{Full-order models: examples}\label{subsec:FOMs}
\irene{Although this work introduces a strategy for developing ROMs for general 1D systems of balance laws, it is important to remember that the derivation of the ROM depends on the chosen numerical flux and the explicit form of the scheme is utilized to construct the reduced model.} Consequently, we consider three different 1D balance laws to evaluate the proposed strategy: the linear transport equation with source term, the Burgers equation with nonlinear source term and SW system with topography and Manning friction.

The numerical schemes used for each system will be explicitly detailed in this subsection in order to subsequently derive the corresponding ROMs.

\subsubsection{Transport equation with a linear source term}

Let us consider the 1D transport equation with a linear source term
\begin{equation}\label{transporte_ecuacion}
    w_t+cw_x=\beta w, \quad c,\beta \in \mathbb{R},
\end{equation}
which corresponds to the choices
\begin{equation}
    W=w, \quad F(w)=cw, \quad S(w)=\beta w,
\end{equation}
for \eqref{PDE}.
Notice that, given a cell-average $w_i^n$, the stationary solution $w_i^{*,n}$ verifying $w_i^{*,n}(x_i)=w_i^n$ is as follows:
\begin{equation}
    w_i^{*,n}(x)=w_i^n e^{(\beta/c) (x-x_i)}.
\end{equation}

\irene{Exactly well-balanced reconstruction operators will be applied as in \cite{Castro2020High} to preserve the stationary solutions. Therefore, if a PVM-0 method is considered, the first-order EWB method reads as follows:}
\begin{equation}
\begin{split}
      w_i^{n+1}&= w_i^n - \frac{\Delta t}{2 \Delta x} c \left( w_{i+1}^n e^{-  \frac{ \beta \Delta x}{2c}} +  w_{i}^n \left(e^{  \frac{ \beta \Delta x}{2c}}-  e^{  \frac{ -\beta \Delta x}{2c}} \right) -  w_{i-1}^n e^{ \frac{ \beta \Delta x}{2c}} \right) \\
       &+\frac{\Delta t }{2 \Delta x}  \left( \alpha_{i+1/2}^{0,n}(w_{i+1}^n e^{-  \frac{ \beta \Delta x}{2c}} -  w_{i}^n e^{  \frac{ \beta \Delta x}{2c}} ) -\alpha_{i-1/2}^{0,n}( w_{i}^n e^{  \frac{ -\beta \Delta x}{2c}}-   w_{i-1}^n e^{ \frac{ \beta \Delta x}{2c}}) \right) \\
      &+ \frac{\Delta t}{ \Delta x} c \left( w_{i}^n e^{  \frac{ \beta \Delta x}{2c}}  -  w_{i}^n e^{ \frac{ -\beta \Delta x}{2c}} \right).
\end{split}
\end{equation}

With the purpose of simplicity only one PVM-0 method has been considered, namely the modified Lax-Friedrichs method, corresponding to 

$$
\alpha_{i+1/2}^{0,n} = \gamma \frac{\Delta x}{\Delta t}.
$$

\subsubsection{Burgers equation with a nonlinear source term}

Let us consider the 1D Burgers equation with a nonlinear source term
\begin{equation}\label{burgers_ecuacion}
    w_t+\left(\frac{w^2}{2}\right)_x=\beta w^2, \quad \beta \in \mathbb{R},
\end{equation}
which corresponds to the choices
\begin{equation}
    W=w, \quad F(w)=\frac{w^2}{2}, \quad S(w)=\beta w^2,
\end{equation}
in \eqref{PDE}. \refuno{Notice that it can be regarded as a convective model taking into account a centrifugal force, defined by $\beta = -C \, \theta_x$, where $\theta$ represents the angle of the domain. Then, by considering a constant value of $\beta$ implies to consider a domain with a constant curvature. Let us remark that $\beta>0$ implies a negative curvature, then the flow accelerates by the centrifugal acceleration. }  In this case, given a cell-average $w_i^n$, the stationary solution verifying $ w_i^{*,n}(x_i)=w_i^n$ is:
\begin{equation}
    w_i^{*,n}(x)=w_i^n e^{\beta (x-x_i)}.
\end{equation}
\irene{Again, we have considered the approach in \cite{Castro2020High} to build exactly well-balanced methods, in which the property is transferred to the reconstruction operators. If a PVM-0 method is considered, the following first-order EWB is obtained:}
\begin{equation}
\begin{split}
      w_i^{n+1}&= w_i^n - \frac{\Delta t}{4 \Delta x}  \left( (w_{i+1}^n)^2 e^{-  \beta \Delta x} + ( w_{i}^n)^2 \left(e^{   \beta \Delta x}-  e^{   -\beta \Delta x} \right) -  (w_{i-1}^n)^2 e^{ \beta \Delta x} \right) \\
       &+\frac{\Delta t }{2 \Delta x}  \left( \alpha_{i+1/2}^{0,n}(w_{i+1}^n e^{-  \frac{ \beta \Delta x}{2}} -  w_{i}^n e^{  \frac{ \beta \Delta x}{2}} ) -\alpha_{i-1/2}^{0,n}( w_{i}^n e^{  \frac{ -\beta \Delta x}{2}}-   w_{i-1}^n e^{ \frac{ \beta \Delta x}{2}}) \right) \\
      &+ \frac{\Delta t}{ 2 \Delta x} (w_{i}^n)^2 \left(  e^{   \beta \Delta x}  -   e^{  -\beta \Delta x}\right).
\end{split}
\end{equation}
As in previous case, we consider the modified Lax-Friedrichs method.

\subsubsection{Shallow water system with non-flat bathymetry and Manning friction}
Let us consider the 1D SW system with non-flat bathymetry and Manning friction. \refuno{The one-dimensional shallow water equations (SWE), also known as the Saint-Venant equations, are a system of hyperbolic partial differential equations that describe free-surface flows in rivers and open channels. Originally derived by Barré de Saint-Venant in the 19th century (\cite{saintvenant1871}), these equations express the fundamental principles of mass conservation and momentum conservation for an incompressible, inviscid fluid under the hydrostatic pressure assumption.
The conservative form of the SWE with bottom slope and Manning friction source terms reads:}
\begin{equation}\label{swf_ecuacion}
    \begin{cases}
    h_t+q_x=0,\\
q_t + \left( \displaystyle \frac{q^2}{h} + \frac{1}{2} g h^2 \right)_x= -ghz_x - g\displaystyle \frac{n_b^2 q |q|}{h^{7/3}},\\
    \end{cases}
\end{equation}
which correspond to the choices $H(x)=-z(x)$ and
\begin{equation}
    \begin{split}
       & W =\begin{pmatrix}
h \\
q \\
\end{pmatrix} , \, F(W) = \begin{pmatrix}
q \\
\displaystyle \frac{q^2}{h}+\displaystyle \frac{g}{2}h^2\\
\end{pmatrix}, \, S(W) =\begin{pmatrix}
0\\
gh\\
\end{pmatrix}, \, R(W)  =\begin{pmatrix}
0 \\
- g\displaystyle \frac{n_b^2 q |q|}{h^{7/3}}
\\
\end{pmatrix}.
    \end{split}
\end{equation}
\refuno{where the first equation corresponds to mass conservation, and the second equation represents momentum conservation, with source terms for the bed slope and Manning's friction law.}
In \eqref{swf_ecuacion}, the unknowns $h$ and $q$ are the water depth of the water layer and discharge, respectively; the function $z(x)$ is the depth function; $g$ is the gravity acceleration and $n_b$ is the Manning friction coefficient. We denote by  $u=q/h$  the depth-averaged velocity and $c=\sqrt{gh}$ the celerity. 

The eigenvalues of the system are $\lambda^\pm=u\pm c$ and the Froude number defined by  
\begin{equation}
Fr(W)= \displaystyle \frac{|u|}{c},
\end{equation}
characterizes the flow regime: subcritical ($Fr<1$), critical ($Fr=1$) or supercritical ($Fr>1$). 

\irene{In this work, we are interested in the preservation of the steady states corresponding to the water at rest, which are of the form
\begin{equation}
    u=0, \quad \eta=constant,
\end{equation}
where $$\eta=h+z,$$ denotes the free surface. Following \cite{castro2017well}, if the path-conservative framework is considered and the family of segments is chosen to approximate the geometrical source term, the method is exactly well-balanced for water at rest. Choosing a PVM-0 method as numerical flux, the first-order EWB method for non-moving steady states reads as follows:}
\begin{equation}\label{FOM_swe_LF_h}
    \begin{split}
        h_i^{n+1}&= h_i^n - \frac{\Delta t}{2 \Delta x } \left( q_{i+1}^n - q_{i-1}^n\right)  + \frac{\Delta t}{2 \Delta x } \left( \alpha_{i+1/2}^{0,n}(\eta_{i+1}^n - \eta_{i}^n) - \alpha_{i-1/2}^{0,n}(\eta_{i}^n - \eta_{i-1}^n)\right),
        \\
    \end{split}
\end{equation}
\begin{equation}\label{FOM_swe_LF_q}
    \begin{split}
         q_i^{n+1}&= q_i^n - \frac{\Delta t}{2 \Delta x } \left( (u_ {i+1}^n)^2 h_{i+1}^n + \frac{1}{2}g (h_ {i+1}^n)^2 - (u_ {i-1}^n)^2 h_{i-1}^n -\frac{1}{2}g (h_ {i-1}^n)^2\right) \\
         & + \frac{\Delta t}{2 \Delta x } \left( \alpha_{i+1/2}^{0,n}(q_{i+1}^n - q_{i}^n) - \alpha_{i-1/2}^{0,n}(q_{i}^n - q_{i-1}^n)\right) \\
         &- \frac{g \Delta t}{4 \Delta x} \left( (h_{i+1}^n + h_i^n) (z_{i+1} - z_i) + (h_{i}^n + h_{i-1}^n) (z_{i} - z_{i-1})\right) - \Delta t g \frac{n_b^2 q_i^n |q_i^n|}{(h_i^n)^{7/3}}.
    \end{split}
\end{equation}

As in previous cases, we consider the case corresponding to the modified Lax-Friedrichs method, defined by 
$$
\alpha_{i+1/2}^{0,n} = \gamma \frac{\Delta x}{\Delta t}.
$$

On the other hand, if the HLL flux is selected, the following first-order EWB method for water-at-rest is considered:
\begin{equation}\label{FOM_swe_HLL_h}
    \begin{split}
        h_i^{n+1}&= h_i^n - \frac{\Delta t}{2 \Delta x } \left( q_{i+1}^n - q_{i-1}^n\right) + \frac{\Delta t}{2 \Delta x } \alpha_{i+1/2}^{0,n} \left(\eta_{i+1}^n - \eta_i^n  \right) \\
        & -  \frac{\Delta t}{2 \Delta x } \alpha_{i-1/2}^{0,n} \left(\eta_{i}^n - \eta_{i-1}^n  \right) + \frac{\Delta t}{2 \Delta x } \alpha_{i+1/2}^{1,n} \left(q_{i+1}^n - q_i^n  \right) \\
        &  -  \frac{\Delta t}{2 \Delta x } \alpha_{i-1/2}^{1,n} \left(q_{i}^n - q_{i-1}^n  \right),
    \end{split}
\end{equation}
\begin{equation}\label{FOM_swe_HLL_q}
    \begin{split}
         q_i^{n+1}&= q_i^n - \frac{\Delta t}{2 \Delta x } \left( (u_ {i+1}^n)^2 h_{i+1}^n + \frac{1}{2}g (h_ {i+1}^n)^2 - (u_ {i-1}^n)^2 h_{i-1}^n -\frac{1}{2}g (h_ {i-1}^n)^2\right) \\
         & + \frac{\Delta t}{2 \Delta x } \left(  \alpha_{i+1/2}^{1,n} \left(- (\widetilde{u}_{i+1/2}^n)^2 + g \widetilde{h}_{i+1/2}^n\right) \left(\eta_{i+1}^n-\eta_{i}^n\right) \right) \\
         & - \frac{\Delta t}{2 \Delta x } \left(  \alpha_{i-1/2}^{1,n} \left(- (\widetilde{u}_{i-1/2}^n)^2 + g \widetilde{h}_{i-1/2}^n\right) \left(\eta_{i}^n-\eta_{i-1}^n\right) \right) \\
          & + \frac{\Delta t}{2 \Delta x } \alpha_{i+1/2}^{0,n}  \left(q_{i+1}^n-q_{i}^n\right)   - \frac{\Delta t}{2 \Delta x } \alpha_{i-1/2}^{0,n} \left(q_{i}^n-q_{i-1}^n\right)  \\
           & + \frac{\Delta t}{ \Delta x } \alpha_{i-1/2}^{1,n} \widetilde{u}_{i+1/2}^n  \left(q_{i+1}^n-q_{i}^n\right)   - \frac{\Delta t}{ \Delta x } \alpha_{i-1/2}^{1,n} \widetilde{u}_{i-1/2}^n  \left(q_{i}^n-q_{i-1}^n\right)  \\
         &- \frac{g \Delta t}{2 \Delta x} \left( \tilde{h}_{i+1/2} (z_{i+1} - z_i) + \tilde{h}_{i-1/2} (z_{i} - z_{i-1})\right) - \Delta t g \frac{n_b^2 q_i^n |q_i^n|}{(h_i^n)^{7/3}},
    \end{split}
\end{equation}
where
\begin{equation} \label{htilde+utilde}
    \widetilde{h}_{i+1/2}^n= \frac{h_{i+1}^n+h_i^n}{2}, \quad \widetilde{u}_{i+1/2}^n = \frac{\sqrt{h_{i+1}^n} u_{i+1}^n + \sqrt{h_{1}^n} u_{1}^n }{\sqrt{h_{i+1}^n} + \sqrt{h_{1}^n} }
\end{equation}
(see \cite{glaister1988approximate}).

\section{POD-based reduced-order models for hyperbolic PDE systems}\label{sec:POD-ROMs}
Let us describe the general strategy to develop a POD-based reduced-order method for a given FOM previously introduced. Given the $j-$th component of $W$, $w_j$, we consider the following snapshot matrix $ M_{w_j} $ whose columns are the set of $N_T$ numerical solutions at each time step $t^n, \, n=1, \dots,N_T$, where $t^{N_T}=T_f$, being $T_f$ the final computational time:
\begin{equation}
 M_{w_j}  =  \begin{pmatrix}
 \vspace{2mm}
w_{j,1}^1 & w_{j,1}^2 & \cdots & w_{j,1}^{N_T}\\
w_{j,2}^1 & w_{j,2}^2 & \cdots & w_{j,2}^{N_T}\\
\vdots & \vdots & \ddots & \vdots\\
w_{j,N_x}^1 & w_{j,N_x}^2 & \cdots & w_{j,N_x}^{N_T}
\end{pmatrix},
\end{equation}
i.e., $w_{j,i}^n$ is the numerical approximation of $w_j$ at the $i-$th cell at time $t^n$.

First, we compute the functions of the POD basis by applying the singular value decomposition to the snapshot matrix for each variable $w_j$ to obtain
\begin{equation}\label{svd}
    M_{w_j}M_{w_j}^T= \Phi_j \Sigma_j \Sigma_j^T \Phi_j^T,
\end{equation}
where 
\begin{equation}
    \Sigma_j=diag(\sigma_1,\cdots, \sigma_r) \in \mathbb{R}^{N_x \times N_T}
\end{equation}
is a diagonal matrix whose entries are the singular values of $M_{w_j}$ and the matrix $\Phi_j=\left(\Phi_{j,1},\cdots, \Phi_{j,N_x}\right) \in \mathbb{R}^{N_x \times N_x}$ consists of the orthogonal eigenvectors of $M_{w_j}M_{w_j}^T$, with $\Phi_{j,k}=\left(\Phi_{j,1,k},\cdots, \Phi_{j,N_x,k}\right)^T $.
Now, let us apply the Galerkin decomposition to each component of $W$:
\begin{equation}\label{Galerkin-decomp}
    w_{j,i}^n \approx wr_{j,i}^n = \sum_{k=1}^M \hat{w}_{j,k}^n \Phi_{j,i,k},
\end{equation}
where $M \ll \min\{N_x,N_T\}$ is the number of POD modes.
In order to develop the ROM, we first introduce the Galerkin decomposition \eqref{Galerkin-decomp} into de FOM, multiply each equation by the $p-$th component of the corresponding basis functions $\Phi_{j,i,p}, \, p=1,\cdots,M$, and sum up over the cells to complete the projection.
Following the previous steps, the system to obtain the coefficients for the ROM reads
\begin{equation}
    \hat{w}_{j,p}^{n+1}=\hat{w}_{j,p}^{n}-\frac{\Delta t}{\Delta x} \left( \sum_{i=1}^{N_x}\mathcal{L}_j \left( Wr_{i-1}^n,Wr_{i}^n,Wr_{i+1}^n\right) \Phi_{j,i,p}\right),
\end{equation}
where $p=1,\cdots, M$, $\mathcal{L}_j$ is the $j-$th component of the operator $\mathcal{L}$, and $Wr_{i}^n$ is the vector whose components are $wr_{j,i}^n$.

The number of POD modes $M$ is chosen following \cite{quarteroni2015reduced}, which suggests a criterion to select the POD dimension $M \leq r$ as the smaller integer such that
\begin{equation}\label{criterio_eleccionModos}
    I(M)=\frac{\sum_{i=1}^M \sigma_i^2}{\sum_{i=1}^r \sigma_i^2} \geq 1 - \varepsilon_{\text{POD}}^2,
\end{equation}
where $\varepsilon_{\text{POD}}$ is a desired tolerance. 

\subsection{Dealing with nonlinearities}
Accurate solutions can often be obtained with a limited number of POD modes when dealing with smooth flow dynamics. Nevertheless, for highly nonlinear and irregular flows, traditional POD-based techniques may encounter limitations. This issue stems from the use of time-invariant basis functions, as discussed in \cite{gerbeau2014approximated} and \cite{zokagoa2018pod}. Various approaches have been introduced to tackle these issues.

In this work, we propose adopting the PID approach. Rather than employing a single set of basis functions throughout the entire simulation duration, this method divides the time domain into multiple subintervals, within which the standard POD method is applied independently. This approach allows the basis functions to vary over time.

The entire simulation interval $[0, T_f]$ is divided into $N_v$ non-overlapping subintervals, referred to as \textit{time windows}, organized as follows:
\begin{equation}
    [0, T_f]=[t^{N_1}, t^{N_2}] \cup [t^{N_3}, t^{N_4}] \cup \cdots \cup [t^{2N_v -1}, t^{2 N_v}],
\end{equation}
where $t^{N_1}=t^1=0$ and $t^{2 N_v}=t^{N_T}=T_f$. Each time window is denoted by 
$$T^v=[t^{N_{2v-1}},t^{N_{2v}}].$$

Thus, each snapshot matrix $ M_{w_j} $ is also divided into $N_v$ submatrices $ M_{w_j}^v \in \mathbb{R}^{{N_x}\times{N_{sv}}},  \, v=1, \cdots, N_v $, to which the POD procedure is performed. Here, $N_{sv}$ is the number of snapshots per time window. As a result, each time window will have its own POD basis composed of the functions $\Phi_{j,k}^v, \, v=1, \cdots, N_v$. \refuno{For simplicity, a uniform partition of the time domain into windows is considered in this work.}

At the interface of each time window, in order to switch from one time interval to the other without requiring an online computation of the full-order model and therefore reducing the computational cost,  we impose continuity in the projection, considering the following jump condition between two times windows $T^{v}$ and $T^{v+1}$:
\begin{equation}
    \hat{w}_{j,k}^{n,v+1}= \sum_{i=1}^{N_x} \sum_ {l=1}^M  \hat{w}_{j,k}^{n,v} \Phi_{j,i,l}^{v} \Phi_{j,i,k}^{v+1}.
\end{equation}

This approach allows to reduce the number of POD modes when dealing with nonlinear terms, so that, considering a much smaller number of POD modes and applying the PID method, a similar accuracy to the results obtained with a standard POD method can be obtained.

In addition, sometimes the POD method cannot be directly applied to some systems. Consider, for example, the Manning friction term in the SW equations \eqref{swf_ecuacion}:
\begin{equation}
     -g\frac{n_b^2 q |q|}{h^{7/3}}=-g\frac{n_b^2 u |u|}{h^{1/3}}.
\end{equation}
This is a highly nonlinear term, in which we find quotients and powers of functions, among others. When solving the ROM, a system to update the POD coefficients $\hat{w}_j^n$ for the conserved variables is obtained. However, sometimes other auxiliary variables are also involved, to which the Galerkin decomposition method is applied. This is the case, for instance, of the velocity $u$ in the SW equations. These additional snapshot matrices, built offline, allows to improve
the efficiency of the standard POD method.  The main problem is the fact that we do not have a direct expression that allows us to update the POD coefficients for these auxiliary variables. In each iteration of the ROM, once the POD coefficients of the conserved variables have been obtained, it would not be efficient to calculate these variables in the physical space in order to be able to carry out the projection in the corresponding POD basis functions and obtain the POD coefficients of the auxiliary variables. To tackle this problem, two different strategies are studied and compared in this paper: the time-averaging approach (see \cite{solan2023development}, \cite{zokagoa2018pod}) and DEIM (see \cite{quarteroni2015reduced}, \cite{hesthaven2016certified}). 

\begin{Remark}
\irene{Note that for the shallow water system, in addition to considering the modified Lax-Friedrichs flux, the HLL method is also used, which allows a higher accuracy to be obtained. However, considering the HLL flux poses a significant challenge in deriving the ROM due to the presence of numerous nonlinearities. This will allow us to verify that the strategies considered in the article to tackle this problem (time-averaging and DEIM) are adequate and compare them to see which one is more suitable.}
\end{Remark}

One the one hand, the time-averaging approach described in \cite{zokagoa2018pod}, \cite{solan2023development} is based on the PID method decribed above. When developing the ROM, instead of considering the reduced variable $wr_j$, the following approximation is considered at each time window:
\begin{equation}\label{timeaveraging_expression}
\overline{w_{j,i}^v}=\frac{1}{N_{sv}} \sum_{s=v_1}^{v_{N_{sv}}} w_{j,i}^s, \quad i=1, \cdots, N_x,
\end{equation}
where $t^{v_1}=t^{N_{2v-1}} $ and $t^{v_{N_{sv}}}=t^{N_{2v}}$ are the initial and final times of the $v-$th time window $T^v, \, v=1,\cdots, N_v$.

On the other hand, DEIM offers a way to update the POD coefficients for auxiliary variables, as discussed in \cite{fu2018pod}, \cite{cstefuanescu2013pod}, \cite{zokagoa2018pod}, among others. The main concept involves substituting the full-space orthogonal projection with a projection onto a reduced-dimensional subspace. This subspace is formed by a basis constructed through a greedy algorithm, which iteratively minimizes the interpolation error across the set of snapshots, evaluated in the maximum norm. As a result, this projection technique requires computing only a few specific components of the nonlinear terms, thereby significantly reducing computational costs. Let us consider, for instance, the velocity $u=q/h$ for the SW equations. This variable can be approximated in the ROM as follows:
\begin{equation}
     u_i^n \approx \sum_{k=1}^M \hat{u}_k^n \Phi_{3,i,k},
\end{equation}
where $\Phi_{3,i,k}$ is the $i-$th component of the basis function
\begin{equation}
\Phi_{3,k}=\left(\Phi_{3,1,k}, \cdots,\Phi_{3,N_x,k} \right)^T,
\end{equation}
which is the $k-$th column of the matrix $\Phi_{3}$ in \eqref{svd}. A pseudo-code to update the POD coefficients $\hat{u}$ at each time iteration is given in Algorithms \ref{deim_offline_u} and \ref{deim_online_u}, where the former is devoted to the offline process and the latter  to the online stage.
\begin{algorithm}[h]
\caption{DEIM: offline stage}
\label{deim_offline_u}
\KwIn{Snapshot matrix $\Phi_3 \in \mathbb{R}^{N_x \times M}$, number of DEIM points $M$}
\KwOut{Interpolation matrix $U$ and index set $\mathcal{I}$}

$i_1 \gets \arg\max_{i=1,\dots,N_x} |\Phi_{3,i,1}|$\;

$U \gets \Phi_{3,1}$\;

$\mathcal{I} \gets \{i_1\}$\;

\For{$m = 2$ \KwTo $M$}{
    $r \gets \Phi_{3,m} - U \left(U_{\mathcal{I}}^{-1} \Phi_{3,\mathcal{I},m}\right)$\;
    
    $i_m \gets \arg\max_{i=1,\dots,N_x} |r_i|$\;
    
    $U \gets [U \quad \Phi_{3,m}]$\;
    
    $\mathcal{I} \gets \mathcal{I} \cup \{i_m\}$\;
}
\Return{$U$, $\mathcal{I}$}
\end{algorithm}
In Algorithm \ref{deim_offline_u}, $U_\mathcal{I}$ and $ \Phi_{3,\mathcal{I},m}$ refers to the vectors formed by the elements at the positions of the indices $\mathcal{I}$ of $U$ and $\Phi_{3,m}$, respectively.

\begin{algorithm}[h]
\caption{DEIM: online stage}
\label{deim_online_u}
\KwIn{$U_{\mathcal{I}} \in \mathbb{R}^{M \times M}$, index set $\mathcal{I} = \{i_1, \dots, i_M\}$}
\KwOut{Approximated vector $\hat{u}$}

Evaluate the nonlinear function at interpolation points: 
$\hat{u}_{\mathcal{I}} \gets \left.\frac{q}{h}\right|_{\mathcal{I}}$\;

Solve the linear system: 
$U_{\mathcal{I}} \, \hat{u} = \hat{u}_{\mathcal{I}}$\;

\Return{$\hat{u}$}
\end{algorithm}

\section{Reduced-order models: some examples}\label{sec:ROM_examples}
Let us illustrate the above procedure by applying it to the transport equation with a linear source term, the Burgers equation with a nonlinear source term, and the SW system with topography and Manning friction introduced in Subsection \ref{subsec:FOMs}. Notice that, since the latter two systems have nonlinear terms, models using the PID strategy are considered. For the sake of simplicity, the superindices corresponding to the time windows have been removed.

\subsection{Transport equation with a linear source term}
The ROM is developed by applying the Galerkin decomposition to $w$, obtaining the following vector formulation:
\begin{equation}
    \hat{w}^{n+1}=\hat{w}^n- \frac{c \Delta t}{2 \Delta x} A \hat{w}^n + \frac{\gamma}{2} B \hat{w}^n + \frac{c \Delta t}{\Delta x} C \hat{w}^n, 
    \end{equation}
where the training matrices $A=(A_{pk})$, $B=(B_{pk})$, $C=(C_{pk})$ are given by
\begin{equation}
\begin{split}
    A_{pk}&=  a_1 \Phi_{1,p} +  \sum_{i=2}^{N_x -1 } \left( \Phi_{i+1,k}e^- + \Phi_{i,k}(e^+ -e^-)-\Phi_{i-1,k}e^+  \right) \Phi_{i,p}  +  a_{N_x} \Phi_{N_x,p},
    \end{split}        
\end{equation}
\begin{equation}
\begin{split}
     B_{pk}&=  b_1 \Phi_{1,p} +  \sum_{i=2}^{N_x -1 } \left( \Phi_{i+1,k}e^- - \Phi_{i,k}(e^+  + e^-)+\Phi_{i-1,k}e^+  \right) \Phi_{i,p}  +  b_{N_x} \Phi_{N_x,p},
    \end{split}        
\end{equation}
\begin{equation}
\begin{split}
\hspace{-6.5cm} C_{pk}&=   \sum_{i=1}^{N_x } \left(  \Phi_{i,k}(e^+  - e^-)\right) \Phi_{i,p},
    \end{split}        
\end{equation}
with $ e^-=e^{- \frac{\beta \Delta x}{2 c}}, \quad  e^+=e^{\frac{\beta \Delta x}{2 c}}$
\irene{and $a_1, b_1, a_{N_x},b_{N_x}$ the boundary terms. For instance, if free boundary conditions are considered, one has:
\begin{equation*}
    \begin{split}
        a_1&=\Phi_{2,k}e^- - \Phi_{1,k}e^-, \quad  a_{N_x}=  \Phi_{N_x,k}e^+ -\Phi_{N_x-1,k}e^+, \\
        b_1&=\Phi_{2,k}e^- - \Phi_{1,k} e^-, \quad b_{N_x}= - \Phi_{N_x,k}e^+ +\Phi_{N_x-1,k}e^+.
    \end{split}
\end{equation*}

}

\subsection{Burgers equation with a nonlinear source term}
The ROM is developed by applying the Galerkin decomposition to $w$, obtaining the following vector formulation:
\begin{equation}
    \hat{w}^{n+1}=\hat{w}^n- \frac{ \Delta t}{4 \Delta x} (\hat{w}^n)^T A \hat{w}^n + \frac{\gamma}{2} B \hat{w}^n + \frac{ \Delta t}{2 \Delta x} (\hat{w}^n)^T C \hat{w}^n, 
    \end{equation}
where the training 3D linear tensors $A=(A_{plk})$ and $C=(C_{plk})$ and the matrix  $B=(B_{pk})$ are given by 
\begin{equation}
\begin{split}
    A_{plk}&=  a_1 \Phi_{1,p} +  \sum_{i=2}^{N_x -1 } \left( \Phi_{i+1,k} \Phi_{i+1,l}E^- + \Phi_{i,k}\Phi_{i,l}(E^+ -E^-)-\Phi_{i-1,k}\Phi_{i-1,l}E^+  \right) \Phi_{i,p},
    \end{split}        
\end{equation}
\begin{equation}
\begin{split}
     B_{pk}&=  b_1 \Phi_{1,p} +  \sum_{i=2}^{N_x -1 } \left( \Phi_{i+1,k}e^- - \Phi_{i,k}(e^+  + e^-)+\Phi_{i-1,k}e^+  \right) \Phi_{i,p}  +  b_{N_x} \Phi_{N_x,p},
    \end{split}        
\end{equation}
\begin{equation}
\begin{split}
  \hspace{-6.5cm}    C_{plk}&=    \sum_{i=1}^{N_x  } \left(  \Phi_{i,k}\Phi_{i,l}(E^+  - E^-)\right) \Phi_{i,p},
    \end{split}        
\end{equation}
with $E^-=e^{- \beta \Delta x}, \,E^+=e^{\beta \Delta x}, \, e^-=e^{- \frac{\beta \Delta x}{2 }}, \, e^+=e^{\frac{\beta \Delta x}{2 }}$
and $a_1, b_1, a_{N_x},b_{N_x}$ the boundary terms.

\subsection{Shallow water system with non-flat bathymetry Manning friction}
Let us consider the SW with bottom and Manning friction \eqref{swf_ecuacion}. Using the general notation in \eqref{PDE}, $h=w_1$ and $q=w_2$. Moreover, the POD method  will be also applied to the velocity $u$, applying DEIM to update it when solving the system at each iteration. Each variable will be approximated as follows:
\begin{equation}
    h_i^n \approx \sum_{k=1}^M \hat{h}_k^n \Phi_{1,i,k}, \quad 
    q_i^n \approx \sum_{k=1}^M \hat{q}_k^n \Phi_{2,i,k}, \quad 
    u_i^n \approx \sum_{k=1}^M \hat{u}_k^n \Phi_{3,i,k}.
\end{equation}
As described in Section \ref{subsec:FOMs}, the modified Lax-Friedrichs and HLL numerical fluxes have been considered, leading to two different ROMs.

\subsubsection{Modified Lax-Friedrichs ROM}\label{sec:SW_LF_ROM}
If the modified Lax-Friedrichs numerical flux \eqref{FOM_swe_LF_h}-\eqref{FOM_swe_LF_q} is considered, the vector formulation of the ROM model reads as follows:
\begin{equation}\label{SWfric}
    \begin{split}
\hat{h}^{n+1}&=\hat{h}^n - \frac{\Delta t}{2 \Delta x} A \hat{q}^n + \frac{\gamma}{2} B \hat{h}^n + \frac{\gamma}{2} C,\\
\hat{q}^{n+1}&=\hat{q}^n - \frac{\Delta t}{2 \Delta x} (\hat{u}^n)^T D \hat{q}^n - g\frac{\Delta t}{4 \Delta x} (\hat{h}^n)^T E \hat{h}^n + \frac{\gamma}{2} F  \hat{q}^n - g\frac{\Delta t}{4 \Delta x}  G \hat{h}^n -  \Delta t \textit{Fric},
    \end{split}
\end{equation}
where the training matrices $A=(A_{pk})$, $B=(B_{pk})$,  $F=(F_{pk})$ and $G=(G_{pk})$, the vector $C=(C_{p})$, and the 3D linear tensors $D=(D_{plk})$ and $E=(E_{plk})$ are given by
\begin{equation}
\begin{split}
\hspace{-2.5cm}
    A_{pk}&=  a_1 \Phi_{1,1,p} +  \sum_{i=2}^{N_x -1 } \left( \Phi_{2,i+1,k} -\Phi_{2,i-1,k} \right) \Phi_{1,i,p} +  a_{N_x} \Phi_{1,N_x,p},
\end{split}        
\end{equation}
\begin{equation}
\begin{split}
     B_{pk}&=  b_1 \Phi_{1,1,p} +  \sum_{i=2}^{N_x -1 } \left( \Phi_{1,i+1,k} - 2 \Phi_{1,i,k} + \Phi_{1,i-1,k} \right) \Phi_{1,i,p} +  b_{N_x} \Phi_{1,N_x,p},
    \end{split}        
\end{equation}
\begin{equation}
\begin{split}
 \hspace{-2.8cm}     C_{p}&= c_1 \Phi_{1,1,p} +  \sum_{i=2}^{N_x -1 } \left( z_{i+1} - 2 z_{i} + z_{i-1} \right) \Phi_{1,i,p} +  c_{N_x} \Phi_{1,N_x,p},
    \end{split}        
\end{equation}
\begin{equation}
\begin{split}
      D_{plk}&=  d_1 \Phi_{2,1,p} +  \sum_{i=2}^{N_x -1 } \left( \Phi_{2,i+1,k} \Phi_{3,i+1,l} -\Phi_{2,i-1,k} \Phi_{3,i-1,l} \right) \Phi_{2,i,p} +  d_{N_x} \Phi_{2,N_x,p},
    \end{split}        
\end{equation}
\begin{equation}
\begin{split}
      E_{plk}&=  e_1 \Phi_{2,1,p} +  \sum_{i=2}^{N_x -1 } \left( \Phi_{1,i+1,k} \Phi_{1,i+1,l} -\Phi_{1,i-1,k} \Phi_{1,i-1,l} \right) \Phi_{2,i,p} +  e_{N_x} \Phi_{2,N_x,p},
    \end{split}        
\end{equation}
\begin{equation}
\begin{split}
       F_{pk}&=  f_1 \Phi_{2,1,p} +  \sum_{i=2}^{N_x -1 } \left( \Phi_{2,i+1,k} - 2 \Phi_{2,i,k} + \Phi_{2,i-1,k} \right) \Phi_{2,i,p} +  f_{N_x} \Phi_{2,N_x,p},
    \end{split}        
\end{equation}
\begin{equation}\label{SW_trainingmatrices_LF}
\begin{split}
 \hspace{0.4cm}       G_{pk}&=  g_1 \Phi_{2,1,p}+  g_{N_x} \Phi_{2,N_x,p}\\
        &+  \sum_{i=2}^{N_x -1 } \left( \left( \Phi_{1,i+1,k} + \Phi_{1,i,k} \right) \left( z_{i+1} + z_{i} \right) + \left( \Phi_{1,i,k} + \Phi_{1,i-1,k} \right) \left( z_{i} + z_{i-1} \right)
        \right) \Phi_{2,i,p},
    \end{split}        
\end{equation}
with \irene{$a_1, b_1, c_1, d_1, e_1, f_1, g_1,a_{N_x}, b_{N_x}, c_{N_x}, d_{N_x}, e_{N_x}, f_{N_x}, g_{N_x}$ the boundary terms.}

In order to build the friction term $\textit{Fric}$ in \eqref{SWfric}, the two strategies described in the previous section (time windows and DEIM) could be considered. This leads to different expressions of the friction term $\textit{Fric}$ in \eqref{SWfric}:
\begin{itemize}
    \item If time-averaging is applied to all the variables, one has:
    \begin{equation}\label{SW_fric_T-A}
        \textit{Fric}=g n_b^2 H,
    \end{equation}
    where $H=(H_p)$ reads as follows:
    \begin{equation}\label{SW_fric_T-A_Hp}
        H_p= \sum_{i=1}^{N_x} \frac{ |\overline{u_i^n}| \overline{u_i^n}}{ \left( \overline{h_i^n}\right)^{1/3}} \Phi_{2,i,p}.
    \end{equation}
    Here, $\overline{u_i^n}$ and $\overline{h_i^n}$ are the time-averaging approximations introduced in \eqref{timeaveraging_expression} applied to $u_i^n$ and $h_i^n$, respectively.
    \item  If time-averaging is only applied to the term
     \begin{equation*}
         \frac{|u|}{h^{4/3}},
     \end{equation*} 
     one has
\begin{equation}\label{SW_fric_T-A+DEIM}
        \textit{Fric}= g n_b^2 H \hat{q}^n,
    \end{equation}
    where $H=(H_{pk})$ reads as follows:
    \begin{equation}\label{SW_fric_T-A+DEIM_H}
        H_{pk}= \sum_{i=1}^{N_x} \frac{ |\overline{u_i^n}| }{ \left( \overline{h_i^n}\right)^{4/3}} \left( \Phi_{2,i,p} \right)^2.
    \end{equation}
    \item If the auxiliary variable
    \begin{equation}
        f=\frac{|q|}{h^{7/3}}
    \end{equation}
    is considered, and the Galerkin decomposition method is applied to it
    \begin{equation}\label{fROM_galerkindecomp}
        f_i^n \approx \sum_{k=1}^M \hat{f}_ k^n \Phi_{4,i,k}, 
    \end{equation}
    using DEIM to update the coefficients $ \hat{f}_ k^n$, then one has
     \begin{equation}\label{SW_fric_DEIM}
        \textit{Fric}= g n_b^2 (\hat{q}^n)^T H \hat{f}^n,
    \end{equation}
    where $H=(H_{plk})$ reads as follows:
    \begin{equation}\label{SW_fric_DEIM_H}
          H_{plk}= \sum_{i=1}^{N_x}   \Phi_{4,i,k} \Phi_{2,i,l} \Phi_{2,i,p} .
    \end{equation}
    
\end{itemize}

\subsubsection{HLL ROM}
Again, the time-averaging approach and DEIM will be applied if the HLL numerical flux \eqref{FOM_swe_HLL_h}-\eqref{FOM_swe_HLL_q} is considered. Here, $\overline{\widetilde{u}_{i+1/2}^n}$ and $\overline{\widetilde{h}_{i+1/2}^n}$ denote the time-averaging approximations of $\widetilde{u}_{i+1/2}^n$ and $\widetilde{h}_{i+1/2}^n$, respectively (see \eqref{htilde+utilde}), introduced in \eqref{timeaveraging_expression}. The vector formulation reads as follows:
\begin{equation}
    \begin{split}
\hat{h}^{n+1}&=\hat{h}^n - \frac{\Delta t}{2 \Delta x} A \hat{q}^n + \frac{\Delta t}{2 \Delta x} M_1 + \frac{\Delta t}{2 \Delta x} M_2 + \frac{\Delta t}{2 \Delta x} M_3,\\
\hat{q}^{n+1}&=\hat{q}^n - \frac{\Delta t}{2 \Delta x} (\hat{u}^n)^T D \hat{q}^n - g\frac{\Delta t}{4 \Delta x} (\hat{h}^n)^T E \hat{h}^n + \frac{\Delta t}{2 \Delta x} M_4 + \frac{\Delta t}{2 \Delta x} M_5 + \frac{\Delta t}{2 \Delta x} M_6 \\
&+\frac{\Delta t}{2 \Delta x} M_7 - g\frac{\Delta t}{4 \Delta x}  G \hat{h}^n -  \Delta t \textit{Fric},
    \end{split}
\end{equation}
where the training matrices $A$ and $G$ and the training tensors $D$ and $E$ are the same as in \eqref{SW_trainingmatrices_LF}, and the three previous possibilities \eqref{SW_fric_T-A}-\eqref{SW_fric_T-A+DEIM}-\eqref{SW_fric_DEIM} for the friction term $\textit{Fric}$ can be considered.
Two different approaches are considered to define the terms $M_m, \, m=1, \cdots, 6$, coming from the diffusion part of the HLL numerical flux.
On the one hand, if the time-averaging technique is applied to approximate the values $\alpha_{i+1/2}^{s,n}, \, s=1,2$, denoting these approximations by $\overline{\alpha_{i+1/2}^{s,n}}, \, s=1,2$, one has:
\begin{equation}
\begin{split}
& M_1= U_1 \hat{h}^n, \quad M_2= U_2 \hat{q}^n, \quad M_3= U_3, \\
 M_4= & U_4  \hat{h}^n, \quad M_5= U_5 \hat{q}^n, \quad M_6= U_6 \hat{q}^n,  \quad M_7= U_7,
\end{split}
\end{equation}
where the training matrices $U_1=((U_1)_{pk})$, $U_2=((U_2)_{pk})$, $U_4=((U_4)_{pk})$, $U_5=((U_5)_{pk})$, $U_6=((U_6)_{pk})$ and $U_7=((U_7)_{p})$, and the training vector $U_3=((U_3)_{p})$ are given by
\begin{equation}
\begin{split}
(U_1)_{pk}&=  (u_1)_1 \Phi_{1,1,p} + + (u_1)_{N_x} \Phi_{1,N_x,p} \\
&+ \sum_{i=2}^{N_x -1 } \left( \overline{\alpha_{i+1/2}^{0,n}} \left( \Phi_{1,i+1,k}- \Phi_{1,i,k} \right) - \overline{\alpha_{i-1/2}^{0,n}} \left( \Phi_{1,i,k}- \Phi_{1,i-1,k} \right) \right) \Phi_{1,i,p}, 
\end{split}
\end{equation}
\begin{equation}
\begin{split}
(U_2)_{pk}&=  (u_2)_{1} \Phi_{1,1,p}+ (u_2)_{N_x} \Phi_{1,N_x,p} \\
&+ \sum_{i=2}^{N_x -1 } \left( \overline{\alpha_{i+1/2}^{1,n}} \left( \Phi_{2,i+1,k}- \Phi_{2,i,k} \right) - \overline{\alpha_{i-1/2}^{1,n}} \left( \Phi_{2,i,k}- \Phi_{2,i-1,k} \right) \right) \Phi_{1,i,p},
\end{split}
\end{equation}
\begin{equation}
\begin{split}
(U_3)_{p}&=  (u_3)_1 \Phi_{1,1,p} + (u_3)_{N_x} \Phi_{1,N_x,p} \\
&+ \sum_{i=2}^{N_x -1 } \left( \overline{\alpha_{i+1/2}^{0,n}} \left( z_{i+1}- z_{i} \right) - \overline{\alpha_{i-1/2}^{0,n}} \left( z_{i}- z_{i-1} \right) \right) \Phi_{1,i,p},
\end{split}
\end{equation}
\begin{equation}
\begin{split}
(U_4)_{pk}&=  (u_4)_1 \Phi_{2,1,p}+ (u_4)_{N_x} \Phi_{2,N_x,p} \\
&+ \sum_{i=2}^{N_x -1 } \left( \overline{\alpha_{i+1/2}^{1,n}} \left( -(\overline{\widetilde{u}_{i+1/2}^n})^2 + g \overline{\widetilde{h}_{i+1/2}^n} \right) \left( \Phi_{1,i+1,k}- \Phi_{1,i,k} \right) \right) \Phi_{2,i,p}\\
&- \sum_{i=2}^{N_x -1 } \left( \overline{\alpha_{i-1/2}^{1,n}} \left( -(\overline{\widetilde{u}_{i-1/2}^n})^2 + g \overline{\widetilde{h}_{i-1/2}^n} \right) \left( \Phi_{1,i,k}- \Phi_{1,i-1,k} \right) \right) \Phi_{2,i,p},
\end{split}
\end{equation}
\begin{equation}
\begin{split}
(U_5)_{pk}&=  (u_5)_{1} \Phi_{2,1,p} + (u_5)_{N_x} \Phi_{2,N_x,p},\\
&+ \sum_{i=2}^{N_x -1 } \left( \overline{\alpha_{i+1/2}^{0,n}} \left( \Phi_{2,i+1,k}- \Phi_{2,i,k} \right) - \overline{\alpha_{i-1/2}^{0,n}} \left( \Phi_{2,i,k}- \Phi_{2,i-1,k} \right) \right) \Phi_{2,i,p},
\end{split}
\end{equation}

\begin{equation}
\begin{split}
(U_6)_{pk}&=  (u_6)_{1} \Phi_{2,1,p} + (u_6)_{N_x} \Phi_{2,N_x,p}\\
&+ \sum_{i=2}^{N_x -1 } 2 \left(  \overline{\alpha_{i+1/2}^{1,n}} \overline{\widetilde{u}_{i+1/2}^n} \left( \Phi_{2,i+1,k}- \Phi_{2,i,k} \right) -  \overline{\alpha_{i-1/2}^{1,n}} \overline{\widetilde{u}_{i-1/2}^n} \left( \Phi_{2,i,k}- \Phi_{2,i-1,k} \right) \right) \Phi_{2,i,p}, 
\end{split}
\end{equation}

\begin{equation}
\begin{split}
(U_7)_{p}&=  (u_7)_1 \Phi_{2,1,p}+ (u_7)_{N_x} \Phi_{2,N_x,p} \\
&+ \sum_{i=2}^{N_x -1 } \left( \overline{\alpha_{i+1/2}^{1,n}} \left( -(\overline{\widetilde{u}_{i+1/2}^n})^2 + g \overline{\widetilde{h}_{i+1/2}^n} \right) \left( z_{i+1}- z_{i} \right) \right) \Phi_{2,i,p}\\
&- \sum_{i=2}^{N_x -1 } \left( \overline{\alpha_{i-1/2}^{1,n}} \left( -(\overline{\widetilde{u}_{i-1/2}^n})^2 + g \overline{\widetilde{h}_{i-1/2}^n} \right) \left( z_{i}- z_{i-1} \right)\right) \Phi_{2,i,p},
\end{split}
\end{equation}
and $(u_m)_r, \, m=1, \cdots, 7, \, r=1, N_x,$ are the boundary terms. This choice will be denoted in the numerical experiments as HLL - TAv.

On the other hand, if the Galerkin decomposition method is applied to  $\alpha_{i+1/2}^{s,n}, \, s=1,2$
\begin{equation}
    \alpha_{i+1/2}^{s,n} \approx \sum_{k=1}^M \hat{\alpha}_k^{s,n} \mathcal{A}_{i+1/2,k}^s, \, s=1,2,
\end{equation}
and DEIM is used to update the coefficients $\hat{\alpha}_k^{s,n}, \, s=1,2$, one has:
\begin{equation}
\begin{split}
& M_1= (\hat{\alpha}^{0,n})^T U_1 \hat{h}^n, \quad M_2= (\hat{\alpha}^{1,n})^T U_2 \hat{q}^n, \quad M_3= U_3 \hat{\alpha}^{0,n}, \\
 M_4= (\hat{\alpha}^{1,n})^T & U_4 \hat{h}^n, \quad M_5= (\hat{\alpha}^{0,n})^T U_5 \hat{q}^n, \quad M_6= (\hat{\alpha}^{1,n})^T U_6 \hat{q}^n, \quad M_7= U_7 \hat{\alpha}^{1,n},
\end{split}
\end{equation}
where the training 3D linear tensors  $U_1=((U_1)_{plk})$, $U_2=((U_2)_{plk})$, $U_4=((U_4)_{plk})$, $U_5=((U_5)_{plk})$ and $U_6=((U_6)_{plk})$, and the training matrices $U_3=((U_3)_{pk})$ and  $U_7=((U_7)_{pk})$ are given by

\begin{equation}
\begin{split}
(U_1)_{plk}&=  (u_1)_1 \Phi_{1,1,p}+ (u_1)_{N_x} \Phi_{1,N_x,p} \\
&+ \sum_{i=2}^{N_x -1 } \left(  \mathcal{A}_{i+1/2,l}^0 \left( \Phi_{1,i+1,k}-  \Phi_{1,i,k} \right)  -  \mathcal{A}_{i-1/2,l}^0 \left( \Phi_{1,i,k} -   \Phi_{1,i-1,k}  \right)  \right) \Phi_{1,i,p},
\end{split}
\end{equation}
\begin{equation}
\begin{split}
(U_2)_{plk}&=  (u_2)_{1} \Phi_{1,1,p} + (u_2)_{N_x} \Phi_{1,N_x,p} \\
&+ \sum_{i=2}^{N_x -1 } \left( \mathcal{A}_{i+1/2,l}^1 \left( \Phi_{2,i+1,k}- \Phi_{2,i,k} \right) - \mathcal{A}_{i-1/2,l}^1 \left( \Phi_{2,i,k}- \Phi_{2,i-1,k} \right) \right) \Phi_{1,i,p}, 
\end{split}
\end{equation}
\begin{equation}
\begin{split}
(U_3)_{pk}&=  (u_3)_1 \Phi_{1,1,p} + (u_3)_{N_x} \Phi_{1,N_x,p} \\
&+ \sum_{i=2}^{N_x -1 } \left( \mathcal{A}_{i+1/2,k}^0 \left( z_{i+1}- z_{i} \right) - \mathcal{A}_{i-1/2,k}^0 \left( z_{i}- z_{i-1} \right) \right) \Phi_{1,i,p},
\end{split}
\end{equation}
\begin{equation}
\begin{split}
(U_4)_{plk}&=  (u_4)_1 \Phi_{2,1,p} + (u_4)_{N_x} \Phi_{2,N_x,p}\\
&+ \sum_{i=2}^{N_x -1 } \left( \mathcal{A}_{i+1/2,l}^1 \left( -(\overline{\widetilde{u}_{i+1/2}^n})^2 + g \overline{\widetilde{h}_{i+1/2}^n} \right) \left( \Phi_{1,i+1,k}- \Phi_{1,i,k} \right) \right) \Phi_{2,i,p}\\
&- \sum_{i=2}^{N_x -1 } \left( \mathcal{A}_{i-1/2,l}^1 \left( -(\overline{\widetilde{u}_{i-1/2}^n})^2 + g \overline{\widetilde{h}_{i-1/2}^n} \right) \left( \Phi_{1,i,k}- \Phi_{1,i-1,k} \right) \right) \Phi_{2,i,p},
\end{split}
\end{equation}
\begin{equation}
\begin{split}
(U_5)_{plk}&=  (u_5)_{1} \Phi_{2,1,p}+ (u_5)_{N_x} \Phi_{2,N_x,p} \\
&+ \sum_{i=2}^{N_x -1 } \left( \mathcal{A}_{i+1/2,l}^0 \left( \Phi_{2,i+1,k}- \Phi_{2,i,k} \right) - \mathcal{A}_{i-1/2,l}^0 \left( \Phi_{2,i,k}- \Phi_{2,i-1,k} \right) \right) \Phi_{2,i,p}, 
\end{split}
\end{equation}
\begin{equation}
\begin{split}
(U_6)&_{plk}=  (u_6)_{1} \Phi_{2,1,p}+ (u_6)_{N_x} \Phi_{2,N_x,p} \\
&+ \sum_{i=2}^{N_x -1 } 2 \left(  \mathcal{A}_{i+1/2,l}^1 \overline{\widetilde{u}_{i+1/2}^n} \left( \Phi_{2,i+1,k}- \Phi_{2,i,k} \right) -  \mathcal{A}_{i-1/2,l}^1 \overline{\widetilde{u}_{i-1/2}^n} \left( \Phi_{2,i,k}- \Phi_{2,i-1,k} \right) \right) \Phi_{2,i,p}, 
\end{split}
\end{equation}
\begin{equation}
\begin{split}
(U_7)_{pk}&=  (u_7)_1 \Phi_{2,1,p} + (u_7)_{N_x} \Phi_{2,N_x,p}\\
&+ \sum_{i=2}^{N_x -1 } \left( \mathcal{A}_{i+1/2,k}^1 \left( -(\overline{\widetilde{u}_{i+1/2}^n})^2 + g \overline{\widetilde{h}_{i+1/2}^n} \right) \left( z_{i+1}- z_{i} \right) \right) \Phi_{2,i,p}\\
&- \sum_{i=2}^{N_x -1 } \left( \mathcal{A}_{i-1/2,k}^1 \left( -(\overline{\widetilde{u}_{i-1/2}^n})^2 + g \overline{\widetilde{h}_{i-1/2}^n} \right) \left( z_{i}- z_{i-1} \right) \right) \Phi_{2,i,p},
\end{split}
\end{equation}
and \irene{$(u_m)_s, \, m=1, \cdots, 7, \, s\in \{1, N_x\},$ are the boundary terms.} This choice will be denoted in the numerical experiments as HLL - DEIM.

\section{Well-balanced property}\label{sec:WB_theorem}
The choice of well-balanced methods to build the FOMs is justified by the following result:
\begin{Theorem}\label{theo:wb}
\irene{Let us suppose that the numerical scheme \eqref{met_num_FOM_operador} for the FOM is well-balanced for a stationary solution $W^*$, i.e., $W_i^n=W_i^1, \, \forall i, \, \forall n$, where $W_i^1$ are the cell-averages of $W^*$. Then, the associated ROM is also well-balanced for $W^*$.}
\end{Theorem}
\begin{proof}
Let us first show that, if and the initial cell-averages $W_i^1$ belong to a stationary solution $W^*$ for which the FOM is well-balanced, the matrices $M_{w_j}M_{w_j}^T$ have only one non-zero eigenvalue and, consequently, the ROM has only one mode, i.e., $M=1$.
Indeed, all the columns of $M_{w_j}$ are identical:
\begin{equation}
\begin{split}
 M_{w_j}  &=  \begin{pmatrix}
 \vspace{2mm}
w_{j,1}^1 & w_{j,1}^2 & \cdots & w_{j,1}^{N_T}\\
w_{j,2}^1 & w_{j,2}^2 & \cdots & w_{j,1}^{N_T}\\
\vdots & \vdots & \ddots & \vdots\\
w_{j,N_x}^1 & w_{j,N_x}^2 & \cdots & w_{j,N_x}^{N_T}
\end{pmatrix} = \begin{pmatrix}
 \vspace{2mm}
w_{j,1}^1 & w_{j,1}^1 & \cdots & w_{j,1}^{1}\\
w_{j,2}^1 & w_{j,2}^1 & \cdots & w_{j,2}^{1}\\
\vdots & \vdots & \ddots & \vdots\\
w_{j,N_x}^1 & w_{j,N_x}^1 & \cdots & w_{j,N_x}^{1}
\end{pmatrix}\\
&= \left( w_j^1 | \cdots | w_j^1 \right).
\end{split}
\end{equation}
Let us consider a vector $v \in \mathbb{R}^{N_x}$. Since
\begin{equation}
     M_{w_j}^T v =  (w_j^1)^T  v  \begin{pmatrix}
 \vspace{2mm}
1\\
\vdots \\
1
\end{pmatrix},
\end{equation}
one has
\begin{equation}
    \ker(M_{w_j}^T)=(w_j^1)^{\perp}.
\end{equation}
Let us compute the eigenvalues of $M_{w_j}M_{w_j}^T$:
\begin{itemize}
    \item If $v \in (w_j^1)^{\perp} $:
    \begin{equation}
        M_{w_j}M_{w_j}^T v = (w_j^1)^T  v M_{w_j} \begin{pmatrix}
 \vspace{2mm}
1\\
\vdots \\
1
\end{pmatrix} = 0.
    \end{equation}
 \item If $v = w_j^1$:
   \begin{equation}
        M_{w_j}M_{w_j}^T  w_j^1 =  \left\|(w_j^1)\right\|^2 M_{w_j} \begin{pmatrix}
 \vspace{2mm}
1\\
\vdots \\
1
\end{pmatrix} = N_x  \left\|(w_j^1)\right\|^2 w_j^1.
    \end{equation}
\end{itemize}
Thus, the matrices $M_{w_j}M_{w_j}^T, \, j=1, \cdots, N$ have only one non-zero eigenvalue $\sigma= N_x  \left\|(w_j^1)\right\|^2$ with  eigenvector $w_j^1$.

Consequently, only one POD mode is selected to build the associated ROM. If we denote by $wP_j^n$ the projection of $w_j^n$ into the functions of the POD basis, one has:
\begin{equation}
    \frac{1}{N_T} \left\|w_j^n-wP_j^n \right\|^2 = \sum_{m=M+1}^{r} \sigma_m.
\end{equation}
Notice that, since $\sigma_m=0, \, \forall m>1$, one has the equality $w_j^1=wP_j^1$. Due to the fact that the initial condition for the ROM is computed from the FOM by projecting in the POD basis, at the initial time $t^1$, one has 
\begin{equation}
wr_{j,i}^1=\hat{w}_{j,1}^1 \Phi_{j,i,1}=wP_{j,i}^1=w_{j,i}^1.
\end{equation}
This implies that
\begin{equation}
    \mathcal{L}_j \left( Wr_{i-1}^1,Wr_{i}^1,Wr_{i+1}^1\right)=  \mathcal{L}_j \left( W_{i-1}^1,W_{i}^1,W_{i+1}^1\right),
\end{equation}
and thus
\begin{equation}
    \hat{w}_{j,1}^2=\hat{w}_{j,1}^1 \Rightarrow wr_{j,i}^2=\hat{w}_{j,1}^2 \Phi_{j,i,1}=wr_{j,i}^1, \, \forall i.
\end{equation}
Reasoning in this way, one has:
\begin{equation}
    wr_{j,i}^{n}=wr_{j,i}^1, \, \forall i, \, \forall n, \, \forall j.
\end{equation}

\end{proof}

\begin{Corollary}
\irene{The ROMs considered in  Section \ref{sec:ROM_examples} are fully exactly well-balanced, i.e., they preserve all stationary solutions for the transport equation with a linear source term and the Burgers equation with a nonlinear source term, while in the shallow water model case the ROM is exactly well-balanced for the family of stationary solutions corresponding to water at rest.}
\end{Corollary}

\section{Simulation in parameter-dependent systems}\label{sec:prediction}
Let us suppose that the source term in \eqref{PDE} depends on a physical parameter $\mu$, i.e., the PDE can be written as follows:
\begin{equation}
    W_t+F(W)_x= G(W, \mu).
\end{equation}
Notice that the three systems of balance laws introduced in Section \ref{subsec:FOMs} are of this type. Indeed, in the transport equation with linear source term $S(w)=\beta w$ and the Burgers equation with nonlinear source term $S(w)=\beta w^2$, the parameter is $\mu= \beta$, whereas in the SW system with Manning friction \eqref{swf_ecuacion},
    \begin{equation*}
        R(W)= \left( 0, -g\displaystyle \frac{n_b^2 q |q|}{h^{7/3}} \right)^T,
    \end{equation*}
the parameter $\mu$ is the friction coefficient ($\mu= n_b$).

Sometimes, our interest is to quickly obtain good approximations of the solutions when considering variations of the parameter $\mu$. Then, our next goal is to efficiently provide accurate approximations of the solutions for different values of this parameter, enhancing the flexibility and speed of the model in capturing variations in physical scenarios. The POD-based strategy described in the previous sections will be used to build the predictive ROMs. The key to these prediction models is to include a training step, that is, numerical solutions computed with the FOM for a set of values of the parameter $\mu_l, \, l=1, \cdots, N_{\mu}$, called \textit{training set}, will be considered in the offline process. In particular, the POD basis functions must contain information of these solutions obtained for the different values of the training set. Thus, if a new value of $\mu$ which does not belong to the training set is considered, the predictive ROM will give acceptable approximations. 

Given the variable $w_j, \, j=1, \cdots N$, its snapshot matrix $ M_{w_j}$ reads now as follows:
\begin{equation}
     M_{w_j}= \left(  M_{w_j}^{\mu_1} | \cdots | M_{w_j}^{\mu_{N_{\mu}}}\right) \in \mathbb{R}^{N_x \times (N_t \cdot N_{\mu})}, 
\end{equation}
where $M_{w_j}^{\mu_l}$ is the snapshot matrix obtained with the FOM with $\mu=\mu_l, \, l=1, \cdots, N_{\mu}$. A previous work \cite{crisovan2019model} studies reduced order modeling techniques for hyperbolic conservation laws with applications in uncertainty quantification. In that study, a POD-Greedy approach combined with an empirical interpolation method (EIM) is employed to create a reduced basis space. It should be noted that this is a different approach from that presented in this section, where the technique described in Section \ref{sec:POD-ROMs} is used to build ROMs that can be used to obtain good approximations of the solutions for parameterized systems trained with a set of different parameter values when considering other parameter values outside the training set. Unlike in the paper \cite{crisovan2019model}, here we consider hyperbolic PDE systems with source terms (in fact, the parameters considered in this paper belong to this term). 


\section{Numerical experiments}\label{sec:numericalexperiments}

In all the experiments, free boundary conditions are considered and the CFL parameter is set to $0.9$.

\subsection{Validation of the well-balanced property}
This first experiment is devoted to the validation of the well-balanced property proved in Theorem \ref{theo:wb}. Stationary solutions for the three systems described in Section \ref{subsec:FOMs} have been considered, checking that the corresponding discrete stationary solution defined by its projection into the POD basis function is preserved in the ROM. Since a stationary solution is chosen as initial condition, the number of POD modes is $M=1$. Thus, only one time window $N_v=1$ is considered.

\subsubsection{Transport equation with a linear source term} \label{test_1_1}

 \begin{figure}[h]
 	\begin{center}
 		\subfloat[][Initial condition]{\includegraphics[width=0.49\textwidth]{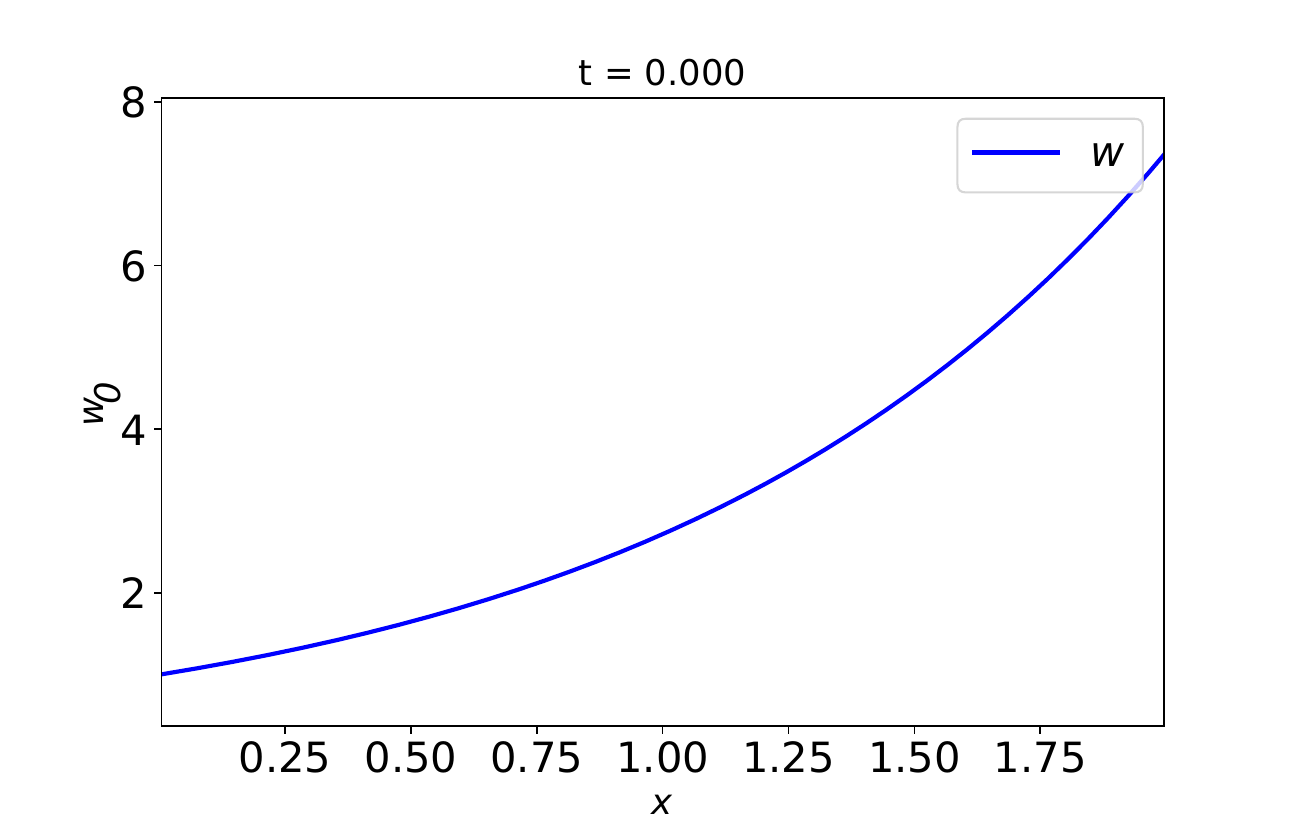}}
 		\subfloat[][Differences]{\includegraphics[width=0.49\textwidth]{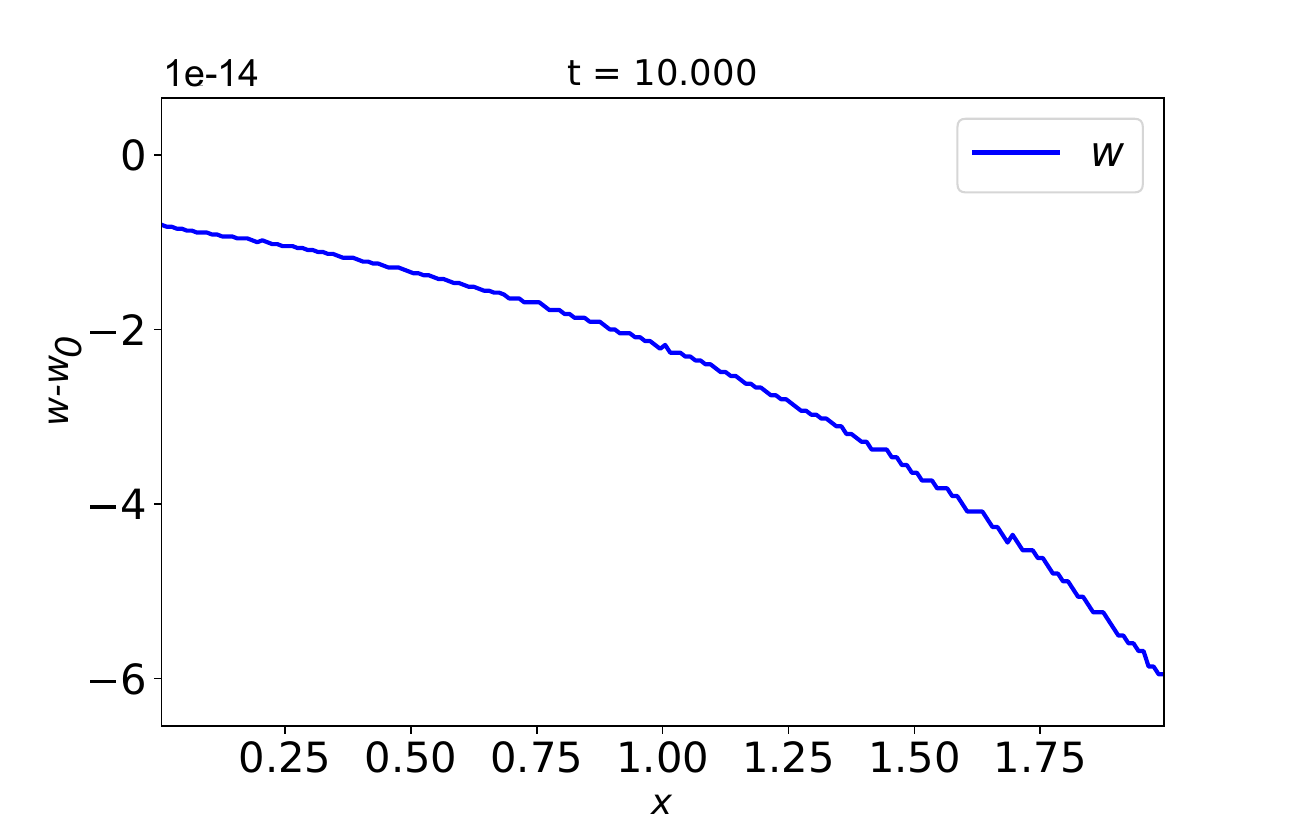}}
 		\caption{Test \ref{test_1_1}. Well-balanced property check: transport equation. Initial condition (left) and differences between the initial and final numerical solutions for a 200-cell mesh at $T_f=10$ s with the ROM (right). }
 		\label{transporte_wbcheck_cini+dif}
 	\end{center}
 \end{figure}
Let us consider the transport equation with linear source term \eqref{transporte_ecuacion}, with $c=1$ and $\beta=1$. Let us consider the space interval  $[0, 2]$ and the time interval  $[0,10]$. As initial condition, we consider discrete stationary solution given by the projection into the POD basis of the stationary solution
$$
w_0(x)=e^{x}$$ 
for the FOM (see Figure \ref{transporte_wbcheck_cini+dif}, left). $L^1$-errors between the initial and final cell-averages of the ROM have been computed for different meshes (see Table \ref{transporte_wbcheck_erroresCPU}). In particular, Figure \ref{transporte_wbcheck_cini+dif} right shows the differences bewteen the initial and final solutions for a 200-cell mesh at $T_f=10$ s with the ROM. CPU times in seconds required by the FOM and ROM to obtain the solution at final time are also shown in Table \ref{transporte_wbcheck_erroresCPU}, \refuno{with the corresponding speed-ups also reported}. As expected, the ROM is faster than the FOM, obtaining a significant reduction in computational time if the number of cells is high.

 \begin{table}[h]
\centering
\begin{tabular}{|c|c|cc|c|} \hline
Cells & Errors with ROM & \multicolumn{2}{c|}{CPU time} & \refuno{Speed-up}\\
 & & FOM & ROM & \\\hline
200&   5.17e-14 & 1.24 & 1.13 & 1.10 \\
400&   1.26e-15 & 6.01 & 4.35 & 1.38 \\
800&  6.40e-14 & 34.23& 16.65 & 2.06\\
1600& 0.0  & 201.24 & 68.02 & 2.95\\
 \hline
\end{tabular}
\caption{Test \ref{test_1_1}. Well-balanced property check: transport equation. $L^1$-errors between the initial and final cell-averages of the ROM, CPU times in seconds required by the FOM and ROM to obtain the solution at $T_f=10$ s \refuno{and speed-ups}.} 
\label{transporte_wbcheck_erroresCPU}
\end{table}

\subsubsection{Burgers equation with a nonlinear source term} \label{test_1_2}

 
Let us consider the Burgers equation with a nonlinear source term \eqref{burgers_ecuacion}, with $\beta=1$. Let us consider the space interval  $[0, 2]$ and the time interval  $[0,10]$. As initial condition, we consider discrete stationary solution given by the projection into the POD basis of the stationary solution
$$
w_0(x)=0.1e^{x}$$ 
for the FOM. $L^1$-errors between the initial and final cell-averages of the ROM have been computed for different meshes (see Table \ref{burgers_wbcheck_erroresCPU}).  CPU times in seconds required by the FOM and ROM to obtain the solution at final time \refuno{and speed-ups} are also shown in Table \ref{burgers_wbcheck_erroresCPU}. Again, the ROM is faster than the FOM, specially if fine meshes are considered.

 \begin{table}[h]
\centering
\begin{tabular}{|c|c|cc|c|} \hline
Cells & Errors with ROM & \multicolumn{2}{c|}{CPU time} & \refuno{Speed-up}\\
 & & FOM & ROM & \\\hline
200&   0.0 & 0.85 & 0.8 &1.06\\
400&   2.33e-15 & 4.38 & 3.21 & 1.36\\
800&  1.09e-16 & 22.88 & 11.74 & 1.95 \\
1600& 8.21e-14  & 130.46 & 50.12 &2.60\\
 \hline
\end{tabular}
\caption{Test \ref{test_1_2}. Well-balanced property check: Burgers equation. $L^1$-errors between the initial and final cell-averages of the ROM and CPU times in seconds required by the FOM and ROM to obtain the solution at $T_f=10$ s \refuno{and speed-ups}.} 
\label{burgers_wbcheck_erroresCPU}
\end{table}
 
\subsubsection{Frictionless shallow water equations with non-flat bathymetry} \label{test_1_3}

\begin{figure}[h]
 \begin{center}
   \subfloat[][Free surface]{\includegraphics[width=0.49\textwidth]{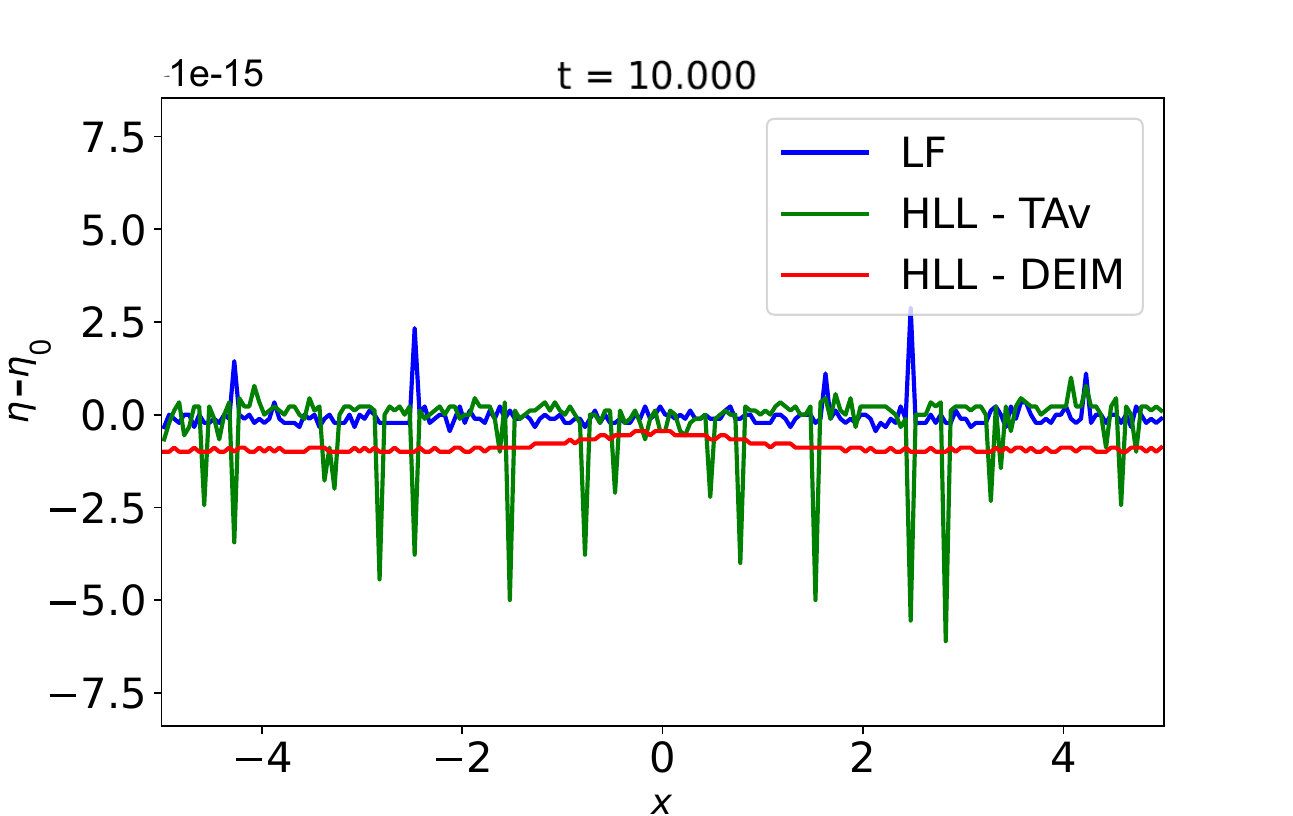}}\vspace{2mm}
   \subfloat[][Velocity]{\includegraphics[width=0.49\textwidth]{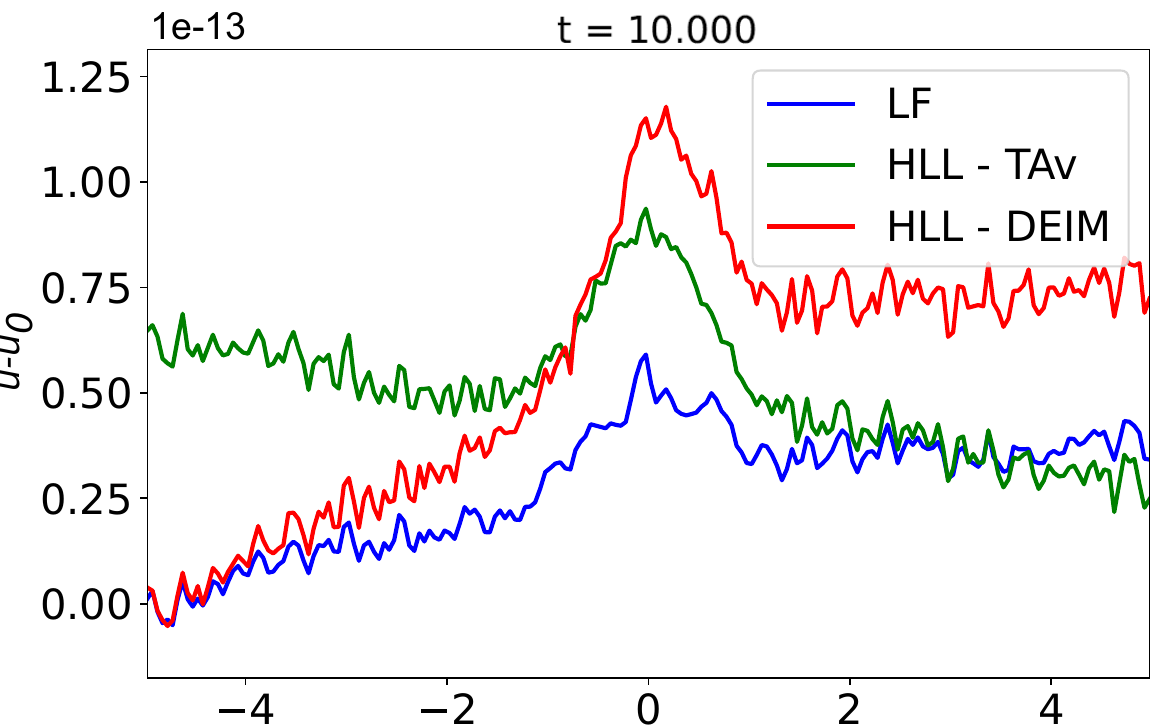}}
     \caption{Test \ref{test_1_3}. Well-balanced property check: shallow water equations. Differences between the initial and final numerical solutions for a 200-cell mesh at $T_f=10$ s with the ROM: free surface $\eta$ (left) and velocity $u$ (right). } 
     \label{sw_wbcheck_dif}
  \end{center}
 \end{figure} 
Let us consider the SW system \eqref{swf_ecuacion} with $n_b=0$. Let us consider the space interval  $[-5,5]$ and the time interval  $[0,10]$. As initial condition, we consider discrete stationary solution given by the projection into the POD basis of the stationary solution
$$
h_0(x)=-z(x), \quad u_0(x)=0, \quad z(x)=-1+0.5e^{-x^2},
$$ 
for the FOM. The gravity is $g=9.81$. The modified Lax-Friedrichs flux and the HLL scheme (considering the two versions HLL - TAv and HLL - DEIM) have been applied. $L^1$-errors between the initial and final cell-averages of the ROM have been computed for different meshes (see Tables \ref{sw_wbcheck_erroresCPU_LF} and \ref{sw_wbcheck_erroresCPU_HLL} for modified Lax-Friedrichs and HLL, respectively). In particular, Figure \ref{sw_wbcheck_dif} shows the differences bewteen the initial and final solutions for a 200-cell mesh at $T_f=10$ s with the ROM for the different schemes. CPU times in seconds required by the FOM and ROM to obtain the solution at final time are also shown in Tables \ref{sw_wbcheck_erroresCPU_LF} and \ref{sw_wbcheck_erroresCPU_HLL}, \refuno{which also include the speed-up}. For this test case, the ROM is much faster than the FOM: particularly remarkable is the reduction obtained for the fine mesh with 1600 cells, where the CPU time is reduced by almost 6 times for the modified Lax-Friedrichs scheme and more than 26 times for the HLL.

 \begin{table}[h]
\centering
\begin{tabular}{|c|cc|cc|c|} \hline
Cells & \multicolumn{2}{c|}{Errors with ROM} & \multicolumn{2}{c|}{CPU time} & \refuno{Speed-up}\\
 & h & q & FOM & ROM & \\\hline
200&   3.00e-14 & 2.41e-13 & 2.72 & 1.31 & 2.07\\
400&   2.94e-13 & 5.91e-13 & 7.31 & 4.01 & 1.82\\
800&   1.59e-13 & 5.10e-13 & 24.38 & 8.62 & 2.82\\
1600&   4.32e-15 & 7.35e-14 & 96.79 & 16.70 & 5.80 \\
 \hline
\end{tabular}
\caption{Test \ref{test_1_3}. Well-balanced property check: shallow water equations. Modified Lax-Friedrichs  scheme. $L^1$-errors between the initial and final cell-averages of the ROM and CPU times in seconds required by the FOM and ROM to obtain the solution at $T_f=10$ s, \refuno{and the speed-ups.}}
\label{sw_wbcheck_erroresCPU_LF}
\end{table}

\begin{table}[h]
\centering
\resizebox{\textwidth}{!}{%
\begin{tabular}{|c|cc|cc|ccc|cc|} \hline
Cells & \multicolumn{2}{c|}{Errors (TAv)} & \multicolumn{2}{c|}{Errors (DEIM)} & \multicolumn{3}{c|}{CPU time} & \multicolumn{2}{c|}{\refuno{Speed-up}} \\
 & $h$ & $q$ & $h$ & $q$ & FOM & TAv & DEIM & TAv & DEIM \\\hline
200  & 2.92e-14 & 4.72e-14 & 3.02e-14 & 4.78e-13 & 8.16 & 2.11 & 2.13 & 3.87 & 3.83 \\
400  & 3.50e-13 & 3.99e-13 & 2.77e-13 & 2.14e-13 & 23.92 & 4.08 & 4.18 & 5.86 & 5.72 \\
800  & 2.74e-14 & 2.87e-13 & 6.29e-14 & 3.03e-13 & 90.70 & 7.76 & 7.95 & 11.69 & 11.41 \\
1600 & 4.19e-15 & 1.18e-13 & 7.58e-14 & 4.22e-14 & 438.20 & 15.34 & 16.50 & 28.57 & 26.56 \\
\hline
\end{tabular}
}
\caption{Test \ref{test_1_3}. Well-balanced property check for the shallow water equations using the HLL scheme. The table reports the $L^1$-errors between the initial and final cell-averages for both ROM-TAv and ROM-DEIM, the CPU times (in seconds) required by the FOM and each ROM variant to compute the solution at $T_f = 10$ s, \refuno{and the corresponding speed-ups obtained by each reduced model.}}
\label{sw_wbcheck_erroresCPU_HLL}
\end{table}

\subsection{Relationship between the number of POD modes, time windows and \newline errors} 
\subsubsection{Transport equation with a linear source term} \label{test_2_1}
Let us first consider the transport equation with linear source term \eqref{transporte_ecuacion}, with $c=1$ and $\beta=1$. The space interval is $[0, 2]$ and the time interval  $[0,0.8]$. As initial condition, we consider discrete stationary solution given by the projection into the POD basis of a small perturbation of a stationary solution
$$
w_0(x)=e^{x}+0.1e^{-100(x-0.3)^2}$$ 
(see Figure \ref{transporte_pert_cini+ErvsNv} (left)). We consider a mesh with 200 cells. 

\begin{figure}[h]
 \begin{center} 
   \subfloat[][Initial condition ]
{\includegraphics[width=0.49\textwidth]{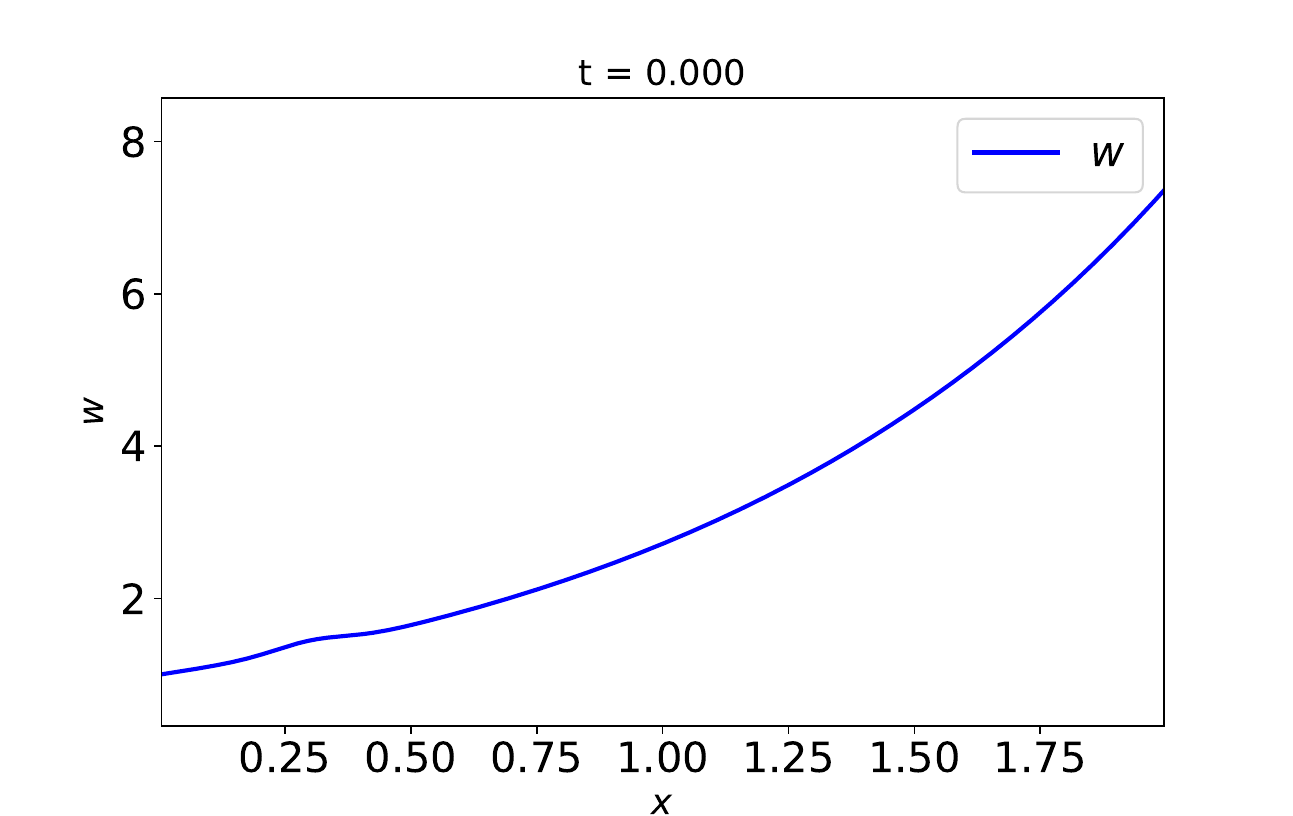}}
 \subfloat[][POD modes vs $L^1$-errors ]
{ \includegraphics[width=0.49\textwidth]{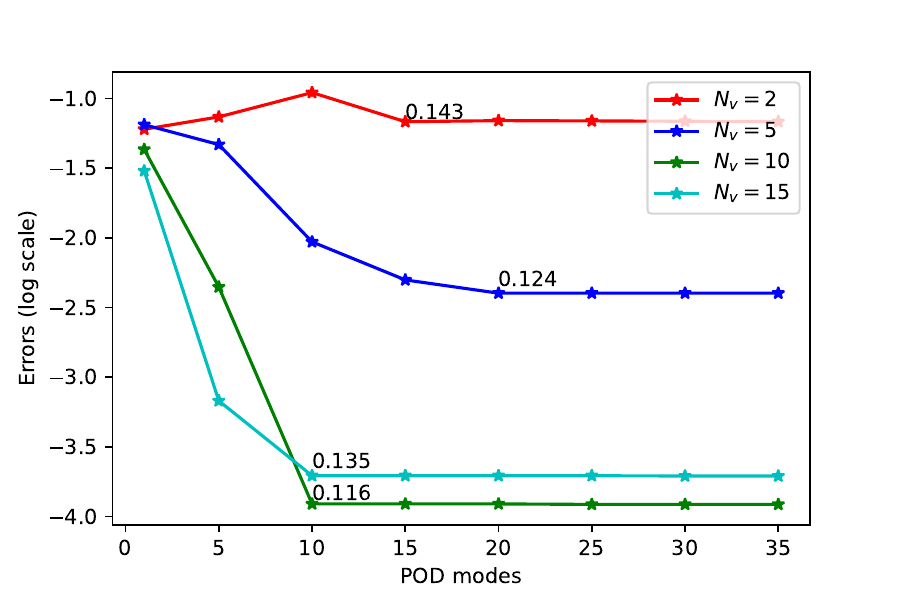}}
     \caption{Test \ref{test_2_1}. Transport equation:  small perturbation of a stationary solution. Initial condition (left) and number of POD modes vs the $L^1$-errors in logarithmic scale for different number of time windows (right).} 
     \label{transporte_pert_cini+ErvsNv}
  \end{center}
 \end{figure}

In order to study the relationship between the number of POD modes $M$ and the time windows $N_v$,  $L^1$-errors between the numerical solutions of the FOM and ROM have been computed in the final time $T_f=0.8$ s considering different numbers of modes for several values of $N_v$. Figure \ref{transporte_pert_cini+ErvsNv} (right) shows the results obtained: the number of POD modes vs the $L^1$-errors in logarithmic scale are plotted. As expected, it is observed that for each value of $N_v$, the errors start to decay with increasing $M$ and end up stabilising. In addition, the numbers appearing at each line for the value $M$ at which the errors stagnate are the CPU times obtained. Note that a large number of modes and time windows are not necessary: in this case, taking $M=10$ with $N_v=10$, we obtain errors of the order of 1e-4. \refuno{The CPU time required by the FOM to reach the final solution at time $t = 0.8$ s was 0.43 s. In Table~\ref{transporte_tab:speedup_vs_Nv}, the speed-ups between the FOM and the ROM are reported for the different time windows, considering the number of modes at which the errors have stabilized, as observed in Figure~\ref{transporte_pert_cini+ErvsNv} (right).}

\refuno{
\begin{table}[h]
\centering
\begin{tabular}{|c|c|c|c|}
\hline
$N_v$ & $M$ & ROM CPU time & Speed-up \\
\hline
2  & 15 & 0.143 & 3.01 \\
5  & 20 & 0.124 & 3.47 \\
10 & 10 & 0.116 & 3.71 \\
15 & 10 & 0.135 & 3.19 \\
\hline
\end{tabular}
\caption{\refuno{Test \ref{test_2_1}. Transport equation: speed-ups obtained by the ROM for different numbers of time windows $N_v$, computed as the ratio between the FOM time (0.43 s) and the ROM CPU times. The number of modes corresponds to the value at which the error was observed to stabilize in Figure \ref{transporte_pert_cini+ErvsNv} (right).}}
\label{transporte_tab:speedup_vs_Nv}
\end{table}
}


 \begin{figure}[h]
 \begin{center}
   \subfloat[][$t=0.8$ s.]{\includegraphics[width=0.49\textwidth]{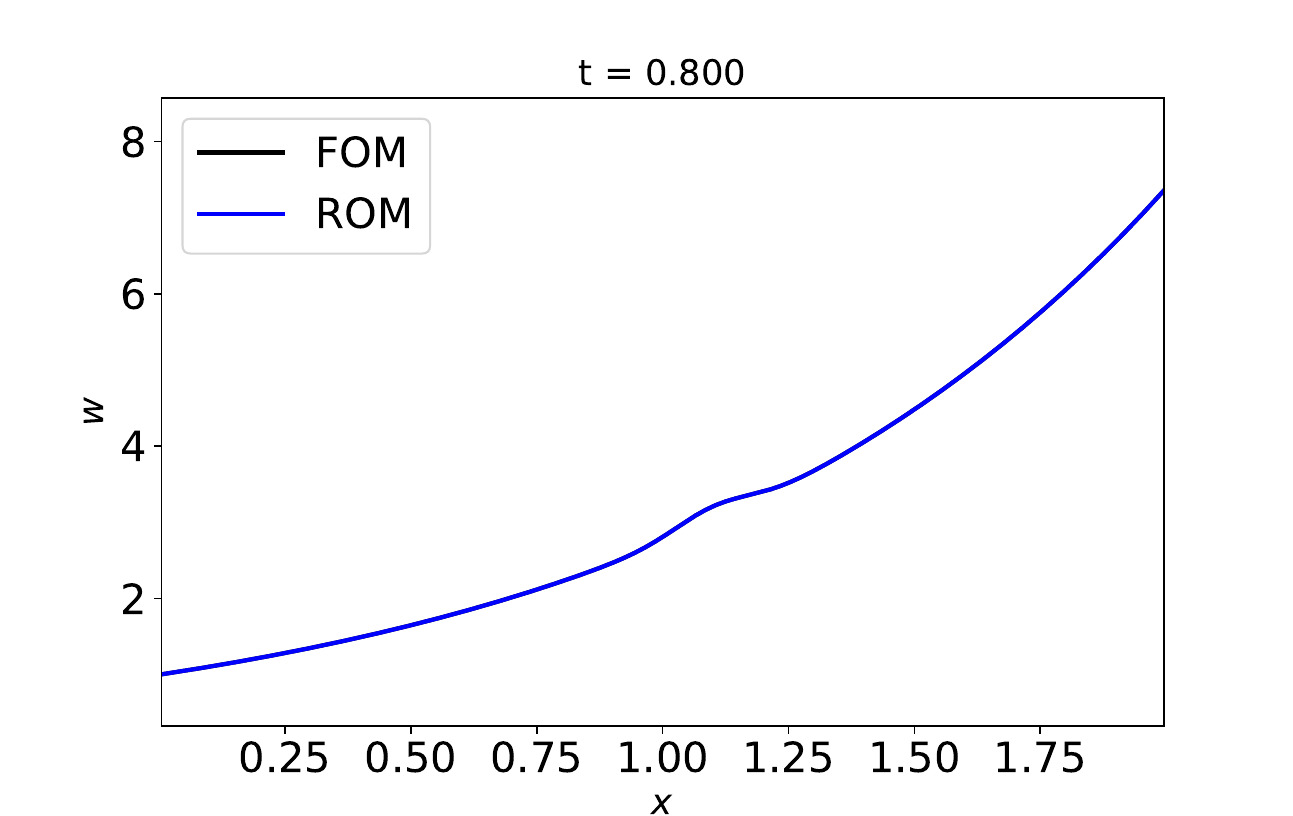}}\vspace{2mm}
    \subfloat[][$t=10$ s.]{\includegraphics[width=0.49\textwidth]{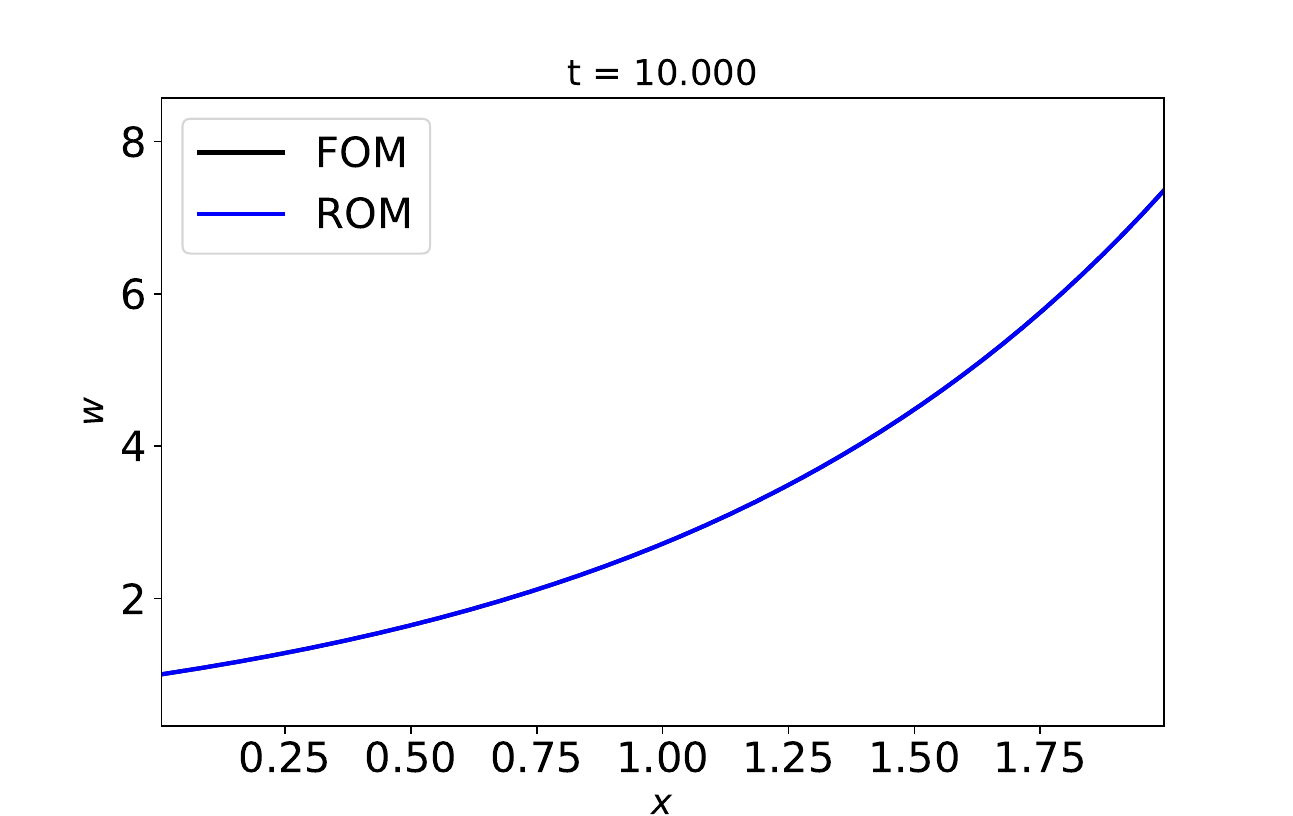}}
     \caption{Test \ref{test_2_1}. Transport equation: small perturbation of a stationary solution. Numerical solutions of FOM and ROM  at $t=0.8$ s. and $t=10$ s.} 
     \label{transporte_pert_t08+10}  
  \end{center}
 \end{figure} 
 
 Moreover, since a perturbation of a steady state is selected as initial condition, let us check what happens when the perturbation leaves the domain. For this purpose, we now consider a total simulation time $T_f=10$ s. and $N_v=100$. Figure \ref{transporte_pert_t08+10} shows the numerical solutions obtained for FOM and ROM at times $t=0.8$ s. and $t=10$ s. The difference obtained in  $L^1$-norm between both solutions is equal to 7.97e-8. In addition, we have measured the error between the last two iterations of the ROM to check that a steady state has been reached: indeed, the error obtained is 5.62e-16.

\subsubsection{Burgers equation with a nonlinear source term} \label{test_2_2}
Let us consider the Burgers equation with a nonlinear source term \eqref{burgers_ecuacion}, with $\beta=1$. The space interval is $[0, 2]$ and the time interval  $[0,3]$. As initial condition, we consider discrete stationary solution given by the projection into the POD basis of a  perturbation of a stationary solution
$$
u_0(x)=0.1e^{x}+0.1e^{-100(x-0.3)^2}.
$$ 
We consider a mesh with 200 cells. 


Let us study the relationship between the number of time windows $N_v$ and the number of POD modes $M$ given a tolerance $\varepsilon_{\text{POD}}$ for the criterion to choice $M$ in \eqref{criterio_eleccionModos}. Given different tolerances $\varepsilon_{\text{POD}}$, Figure \ref{burgers_pert_Nv-vs-POD} (left) shows the numer of POD modes obtained for some values of $N_v$. As expected, the number of POD modes is considerably reduced as $N_v$ increases.


 \begin{figure}[h]
 \begin{center}
   \subfloat[][Time windows vs POD modes]{\includegraphics[width=0.49\textwidth]{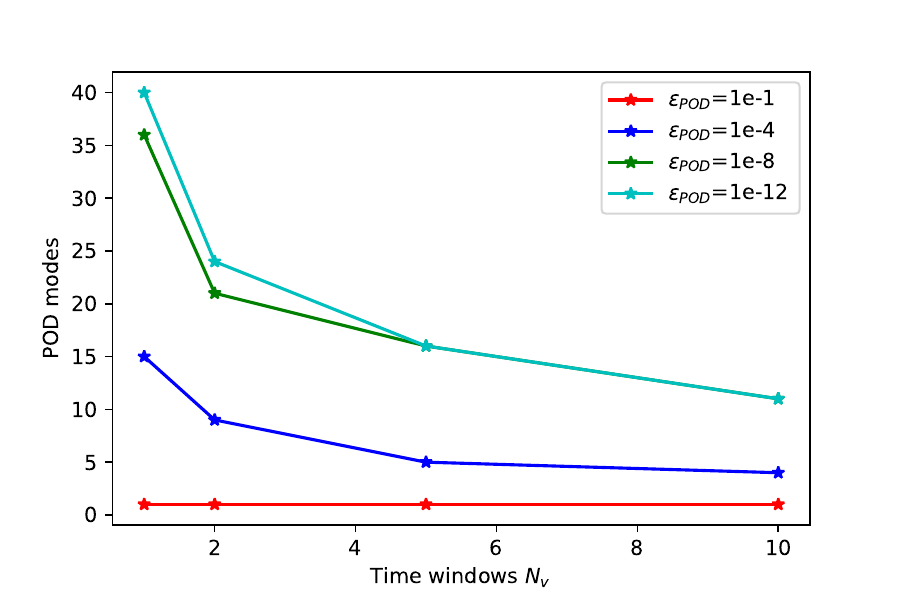}}
    \subfloat[][$N_v$ vs $L^1$-errors (log scale) for $M=5$]{\includegraphics[width=0.49\textwidth]{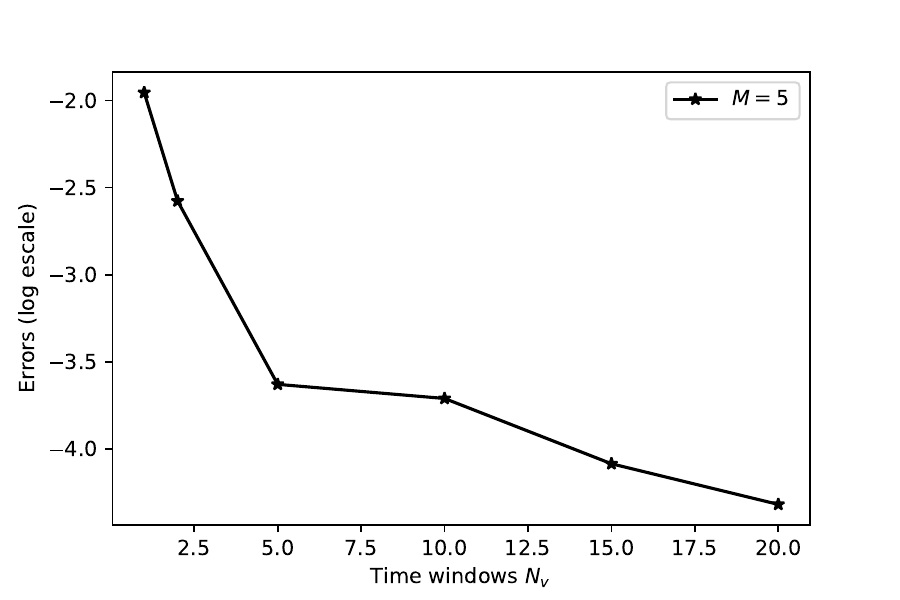}}
     \caption{Test \ref{test_2_2}. Burgers equation: perturbation of a stationary solution. Time windows vs POD modes (left) and Time windows vs $L^1$-errors in logarithmic scale choosing 5 POD modes (right).} 
\label{burgers_pert_Nv-vs-POD}  
  \end{center}
 \end{figure} 
 Since this is a nonlinear system, the use of time windows plays a fundamental role. To reinforce this statement, we have also checked that, if we fix the number of POD modes and increase $N_v$, we obtain better approximations of the solution. Figure \ref{burgers_pert_Nv-vs-POD} (right)  shows the $L^1$-errors in logarithmic scale between the numerical solutions of the FOM and ROM for 6 different values of $N_v$ taking 5 POD modes. \refuno{The CPU time required by the FOM to reach the final solution at $T_f = 3$ s is 0.277 s. Table~\ref{burgers_tab:speedup_fixed_modes} reports the speed-ups obtained with 5 modes and different numbers of time windows $N_v$. These results show that the speed-up remains relatively stable across different values of $N_v$ when the number of modes is fixed.}
 In Figure \ref{burgers_pert_soluciones_5modos} the solutions obtained with the ROM for these values of $N_v$ are plotted and compared with the solution of the FOM (in black).


\refuno{
\begin{table}[h]
\centering
\begin{tabular}{|c|c|c|}
\hline
$N_v$ & ROM Time & Speed-up \\
\hline
1  & 0.158 & 1.75 \\
2  & 0.162 & 1.71 \\
5  & 0.167 & 1.66 \\
10 & 0.164 & 1.69 \\
15 & 0.166 & 1.67 \\
20 & 0.167 & 1.66 \\
\hline
\end{tabular}
\caption{\refuno{Test \ref{test_2_2}. Burgers equation: perturbation of a stationary solution. Speed-ups obtained by the ROM using 5 modes for different numbers of time windows $N_v$.}}
\label{burgers_tab:speedup_fixed_modes}
\end{table}
}

 \begin{figure}[h]
 \begin{center}
   \subfloat[][Solutions with different $N_v$]{\includegraphics[width=0.49\textwidth]{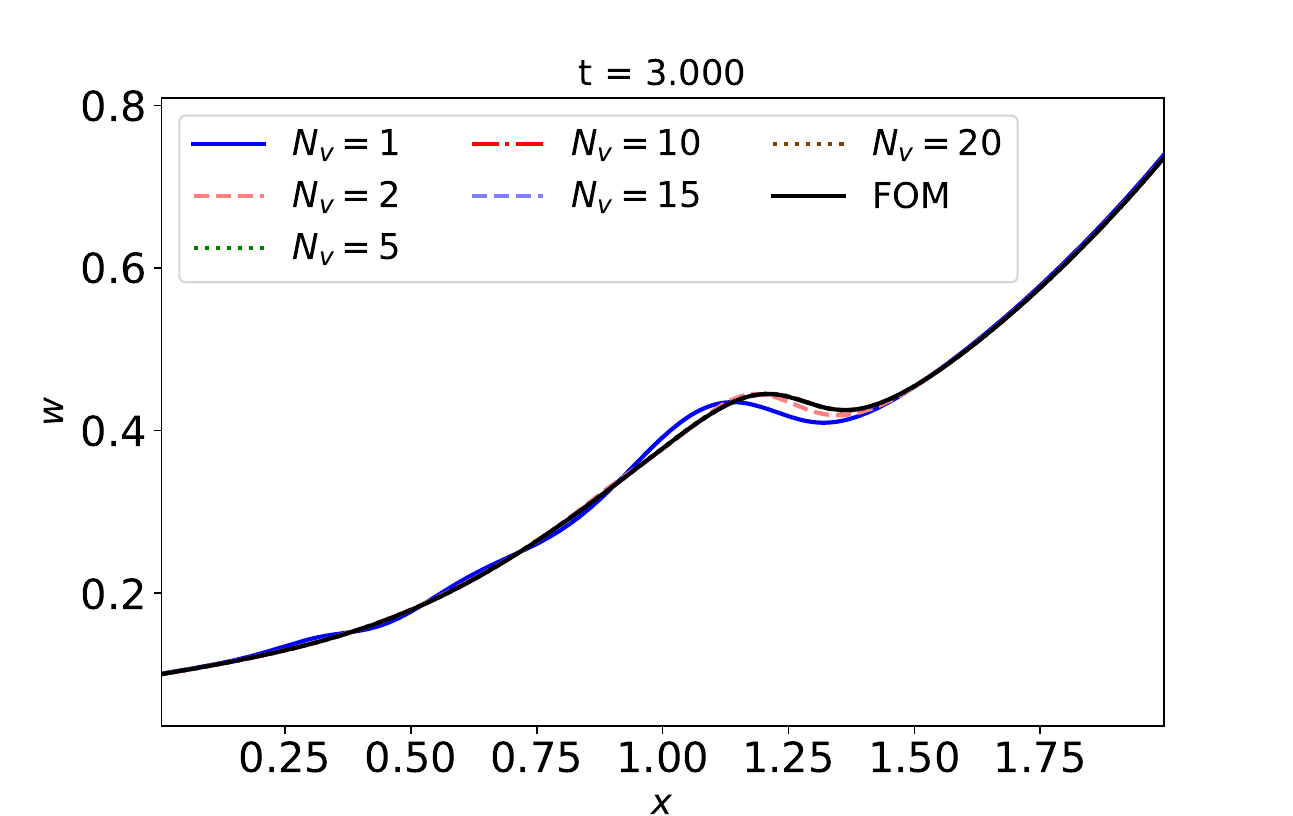}}
    \subfloat[][Solutions with different $N_v$: zoom]{\includegraphics[width=0.49\textwidth]{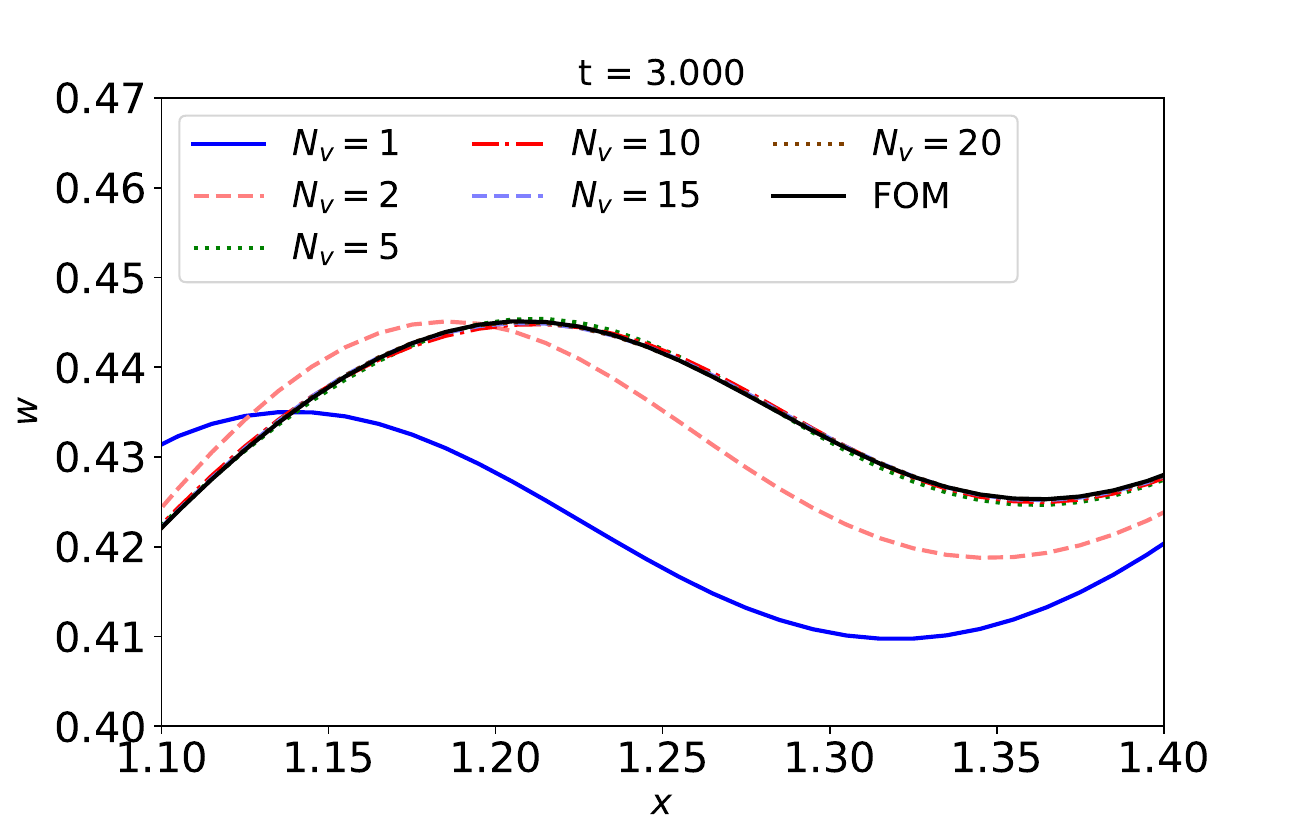}}
     \caption{Test \ref{test_2_2}. Burgers equation: perturbation of a stationary solution. Numerical solutions at $T_f=3$ s. for different time windows choosing 5 POD modes (left) and zoom (right).} 
\label{burgers_pert_soluciones_5modos} 
  \end{center}
 \end{figure} 

\subsection{Shallow water equations: dam-break problem} \label{test_3}
 We consider the SW system with $n_b=0.1$. The space interval  is $[0,12]$ and the time interval is $[0,1]$.  A mesh with 200 cells is considered. The initial condition is a dam-break with a mild slope:
    \begin{equation}\label{dambreak_cini}
    h_0(x)= \left\{ \begin{array}{lcc} 2-z(x) & if & 0 \leq x \leq 6 \\ \\ 1-z(x) & if & 6 < x \leq 12 \end{array} \right., \quad u_0(x)=0,
     \end{equation}
    where
    \begin{equation} \label{fondo_dambreak}
        z(x)=0.2 \left(1-x/12 \right).
    \end{equation}
    
The number of time windows is $N_v=5$ and $\varepsilon_{\text{POD}}=$1e-10. Our goal is to compare the time-averaging and the DEIM approaches. Let us first apply the modified Lax-Friedrichs numerical flux. Following Section \ref{sec:SW_LF_ROM}, we can consider different choices to linearise the friction term in order to build the ROM.

   Firstly, we consider time-averaging for $u$ and $f=|q|/h^{7/3}$, i.e., the expression of the friction term is given by \eqref{SW_fric_T-A}-\eqref{SW_fric_T-A_Hp}. With the previous conditions, the selected number of POD modes is $M=21$. Figure \ref{sw_dambreak_TW5_Fric0c1_LF_AllTAv} shows, from left to right, the water height $h$ and the discharge $q$.  Notice that significant spurious oscillations appear near the front shock.

  \begin{figure}[h]
 \begin{center}
   \subfloat[][Water depth]{\includegraphics[width=0.49\textwidth]{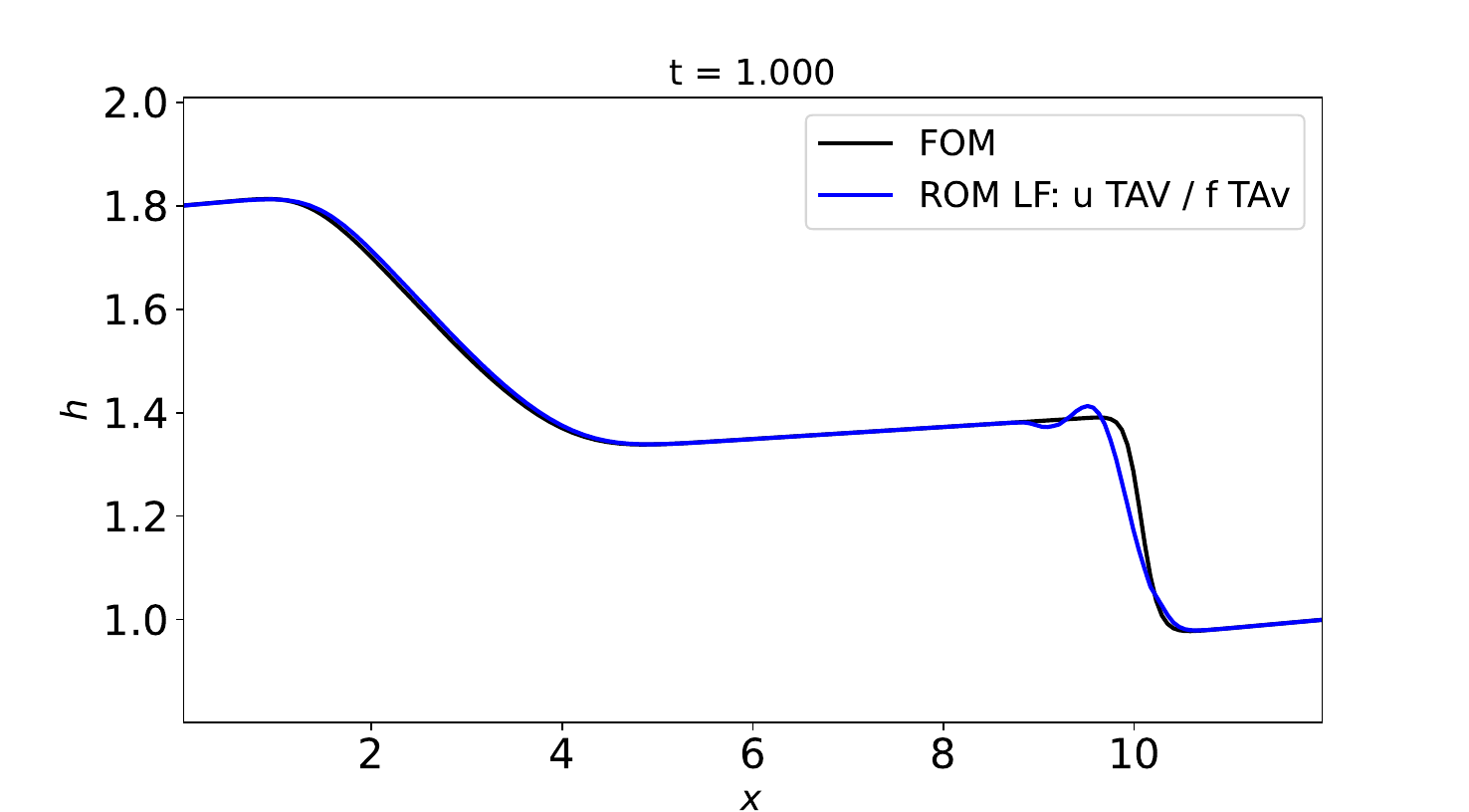}} \vspace{2mm}
   \subfloat[][Discharge]{\includegraphics[width=0.49\textwidth]{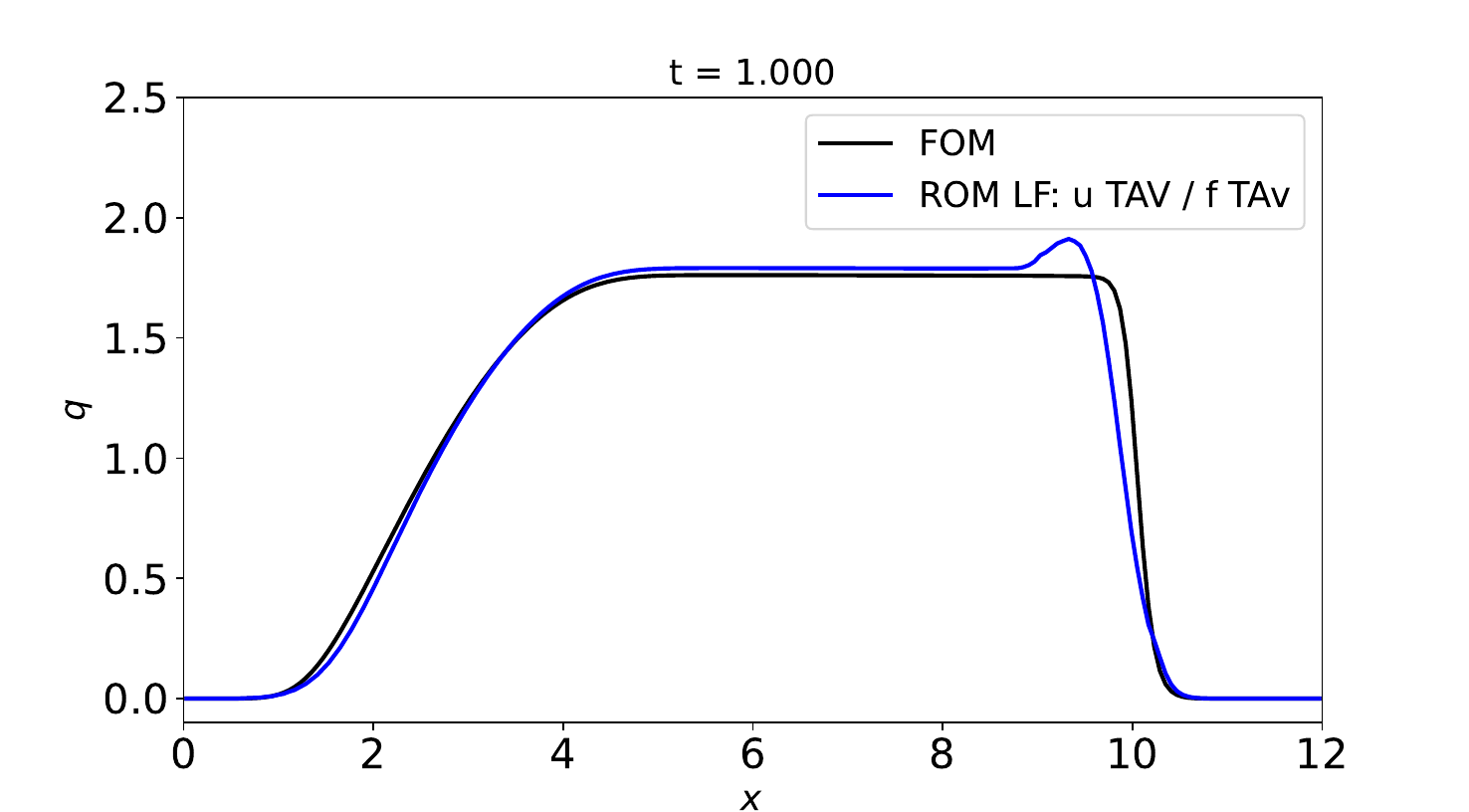}}
    \caption{Test \ref{test_3}. SW equations: dam-break. Lax-Friedrichs numerical flux and time-averaging for $u$ and $f=|q|/h^{7/3}$. Water depth $h$ (left) and discharge $q$ (right) at $T_f=1$.} 
\label{sw_dambreak_TW5_Fric0c1_LF_AllTAv}
  \end{center}
 \end{figure} 

 Secondly, we consider the DEIM approach for $u$ and time-averaging for $f=|q|/h^{7/3}$, leading to the expression \eqref{SW_fric_T-A+DEIM}-\eqref{SW_fric_T-A+DEIM_H} for the friction term in the ROM. Given the previous conditions, the selected number of POD modes is $M=17$. Figure \ref{sw_dambreak_TW5_Fric0c1_LF_uDEIM_fTAv} (up) shows, from left to right, the water depth $h$ of the water and the discharge $q$.  Observe that, for the variable $q$, some spurious oscillations appear near the front shock, although they are significantly smaller than in the previous case: Figure \ref{sw_dambreak_TW5_Fric0c1_LF_uDEIM_fTAv} (down) zooms in near the shock for $q$.
    \begin{figure}[h]
 \begin{center}
   \subfloat[][Water depth]{\includegraphics[width=0.49\textwidth]{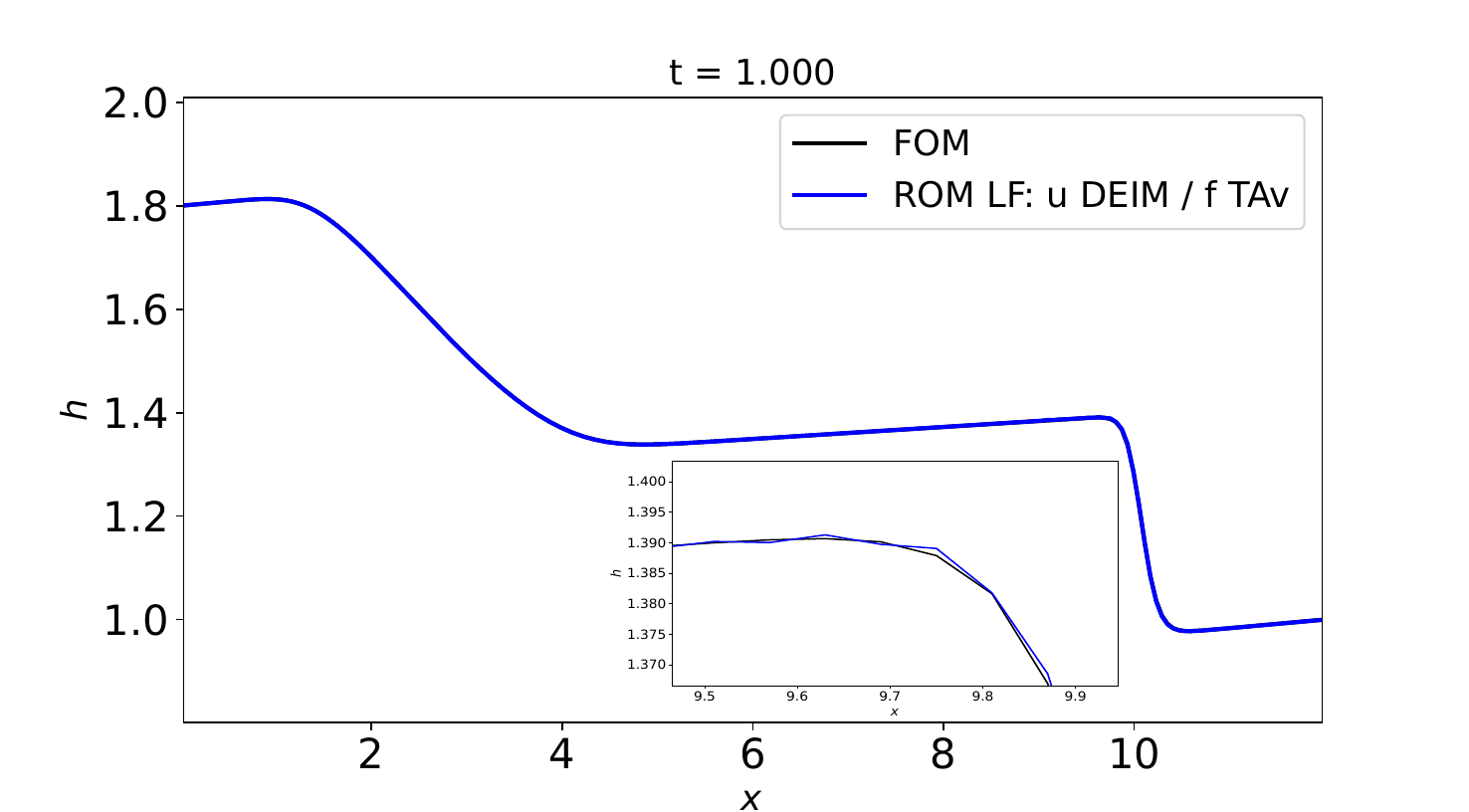}} 
   \subfloat[][Discharge ]{\includegraphics[width=0.49\textwidth]{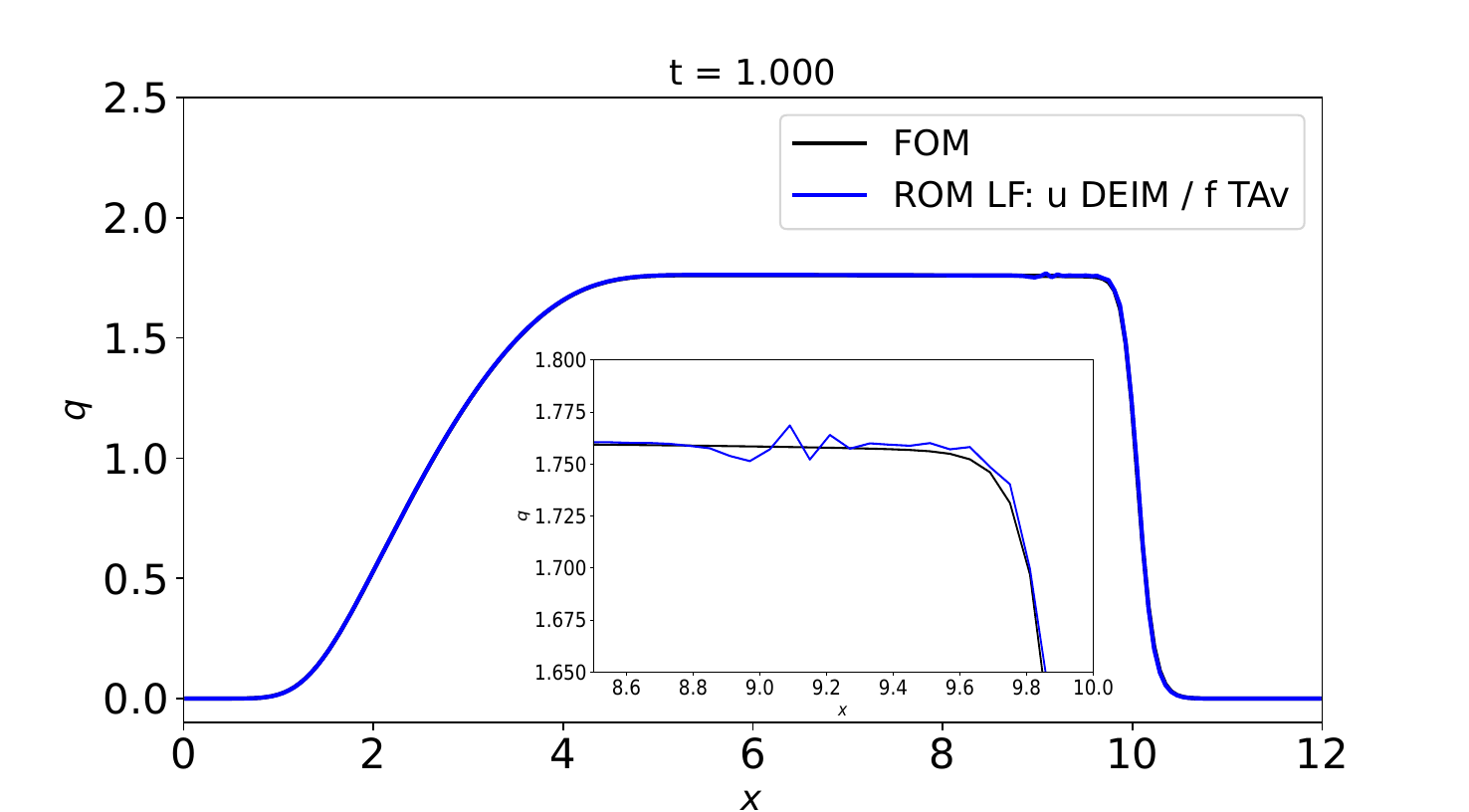}}
    \caption{Test \ref{test_3}. SW equations: dam-break. Lax-Friedrichs numerical flux, DEIM approach for $u$ and time-averaging for $f=|q|/h^{7/3}$.  Water depth $h$ (up) and discharge $q$ (right) at $T_f=1$ s. Inner figures: zooms.} 
\label{sw_dambreak_TW5_Fric0c1_LF_uDEIM_fTAv}
  \end{center}
 \end{figure}

 

Finally, let us consider DEIM for $u$ and $f=|q|/h^{7/3}$, which gives the expression \eqref{fROM_galerkindecomp}-\eqref{SW_fric_DEIM}-\eqref{SW_fric_DEIM_H} of the friction term in the ROM.  Given the previous choices, the selected number of POD modes is $M=20$. Figure \ref{SW_dambreak_TW5_Fric0c1_LF_uDEIM_fDEIM} (up) shows, from left to right, the variables $h$ and $q$.  Note that the spurious oscillations which appear near the front shock in the previous case are no longer in this plot, unlike the previous cases. Figure \ref{SW_dambreak_TW5_Fric0c1_LF_uDEIM_fDEIM} (down) zooms in near the shock for $q$.

   Moreover, errors in $L^1-$norm between the solution of the ROM and a numerical solution computed with the FOM with the parameter $n_b=0.1$ are shown in Table \ref{sw_dambreak_TW5_Fric0c1_LF_tabla} for these three possible choice for the linearization of $u$ and $f$. We can observe that errors are reduced when  using the DEIM approach. 
    \begin{figure}[h]
 \begin{center}
   \subfloat[][Water depth]{\includegraphics[width=0.49\textwidth]{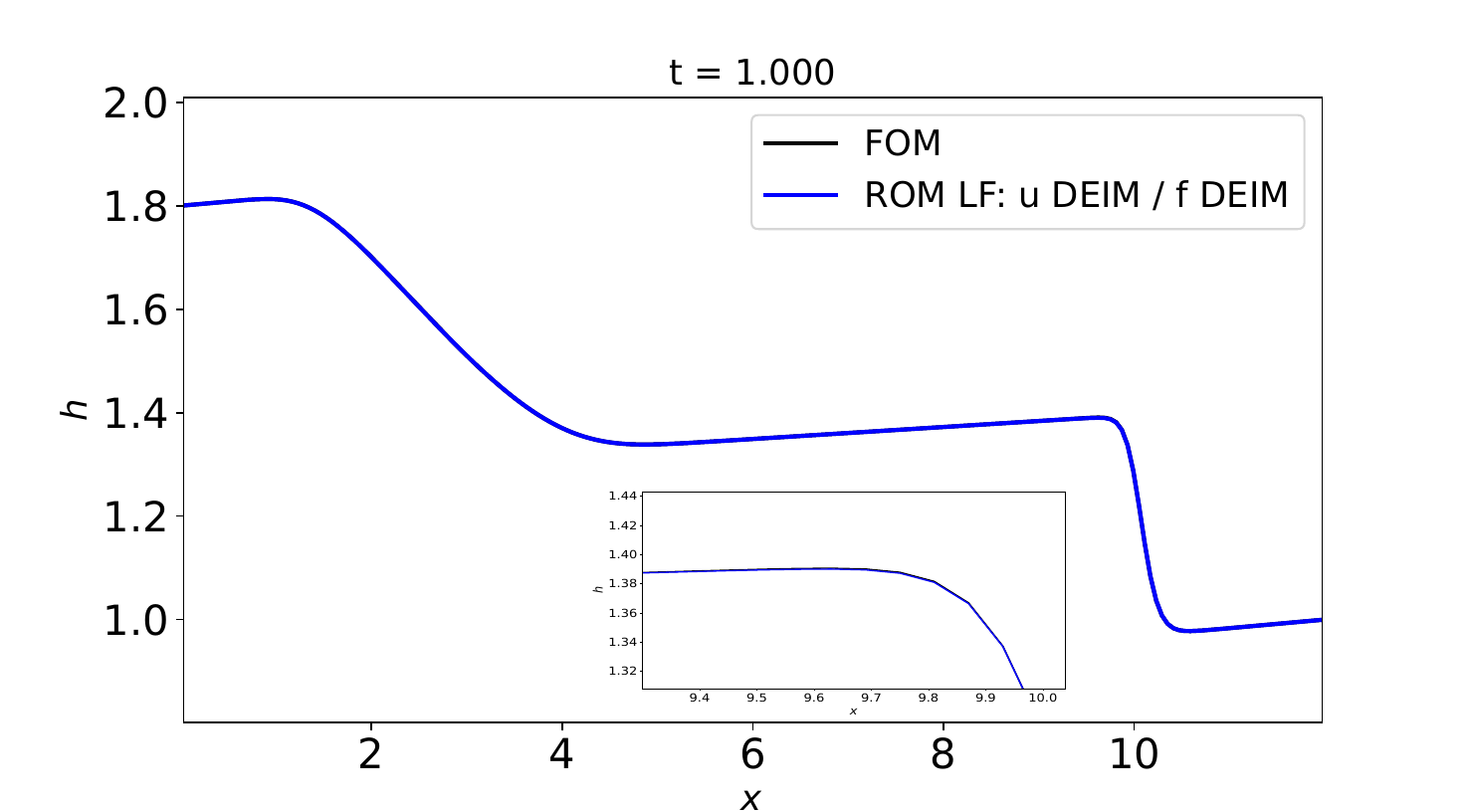}} 
   \subfloat[][Discharge]{\includegraphics[width=0.49\textwidth]{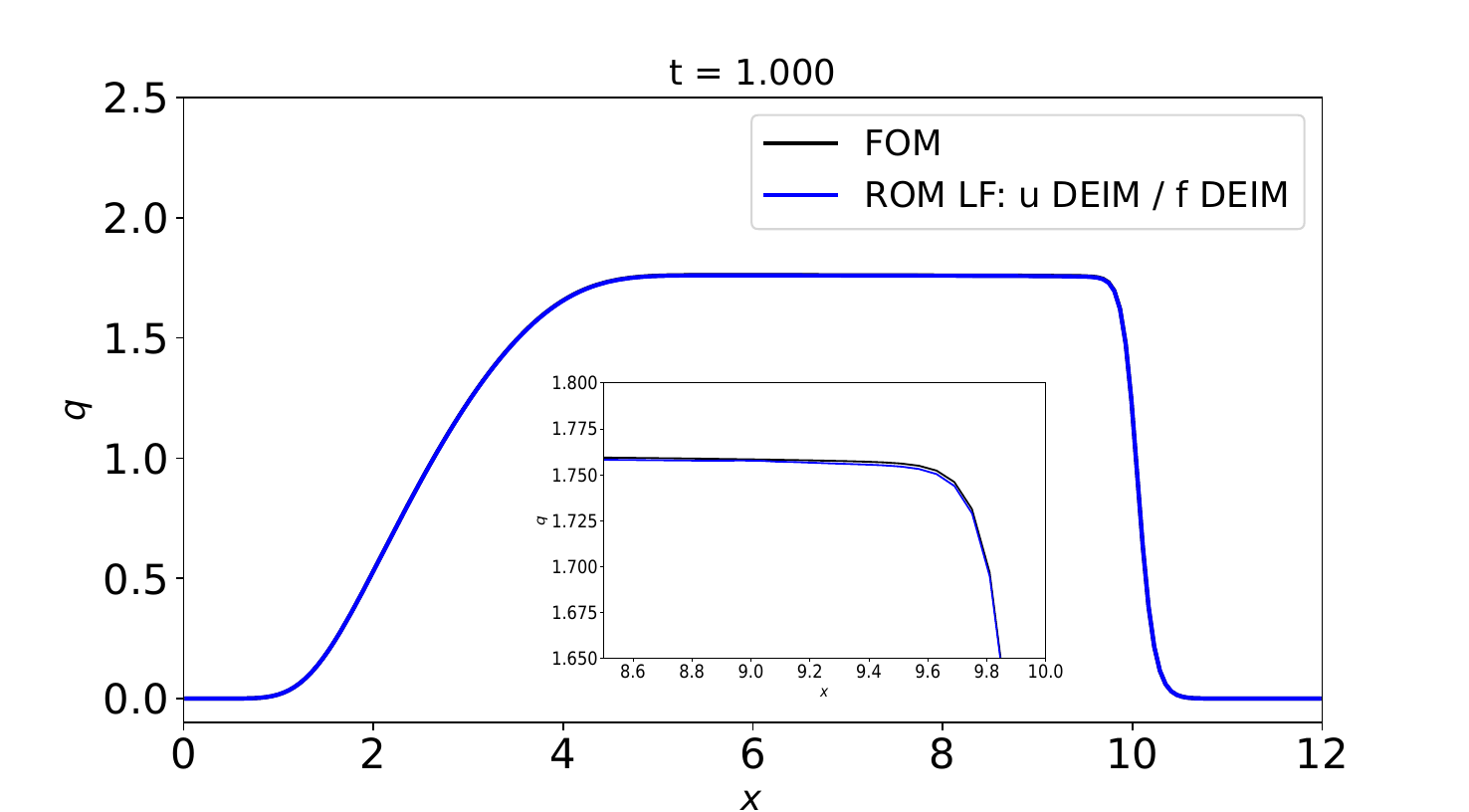}} \\ 
  \caption{Test \ref{test_3}. SW equations: dam-break. Lax-Friedrichs numerical flux, DEIM approach for $u$ and $f=|q|/h^{7/3}$.  Water depth $h$ (left), discharge $q$ (right) at $T_f=1$ s. Inner figures: zooms.} \label{SW_dambreak_TW5_Fric0c1_LF_uDEIM_fDEIM}
  \end{center}
 \end{figure}

 

\medskip 

Let us now consider the HLL numerical flux. As a result of the above findings, the friction term is linearised using the DEIM approach for $u$ and $f=|q|/h^{7/3}$ as in the last of the previous cases. In order to verify that the DEIM strategy does indeed give better results than the time-averaging technique, these two approaches have been applied now to update the coefficients of the diffusive term of the numerical flux $\alpha_{i+1/2}^{0,n}$ and $\alpha_{i+1/2}^{1,n}$, (see \eqref{FOM_swe_LF_h}-\eqref{FOM_swe_LF_q}).

 First, let us consider DEIM for $u$ and $f=|q|/h^{7/3}$ and the time-averaging approach for the coefficients $\alpha_{i+1/2}^{0,n}$ and $\alpha_{i+1/2}^{1,n}$. Considering the previous choices, the selected number of POD modes is $M=20$. Figure \ref{SW_dambreak_TW5_Fric0c1_HLL_uDEIM_fDEIM_coefTAv} (up) shows, from left to right, the height of the water $h$ and the discharge $q$. Again, spurious oscillations appear near the front shock. Figure \ref{SW_dambreak_TW5_Fric0c1_HLL_uDEIM_fDEIM_coefTAv} (down) zooms in near the shock for $q$. 

     \begin{figure}[h]
 \begin{center}
   \subfloat[][ Water depth]{\includegraphics[width=0.49\textwidth]{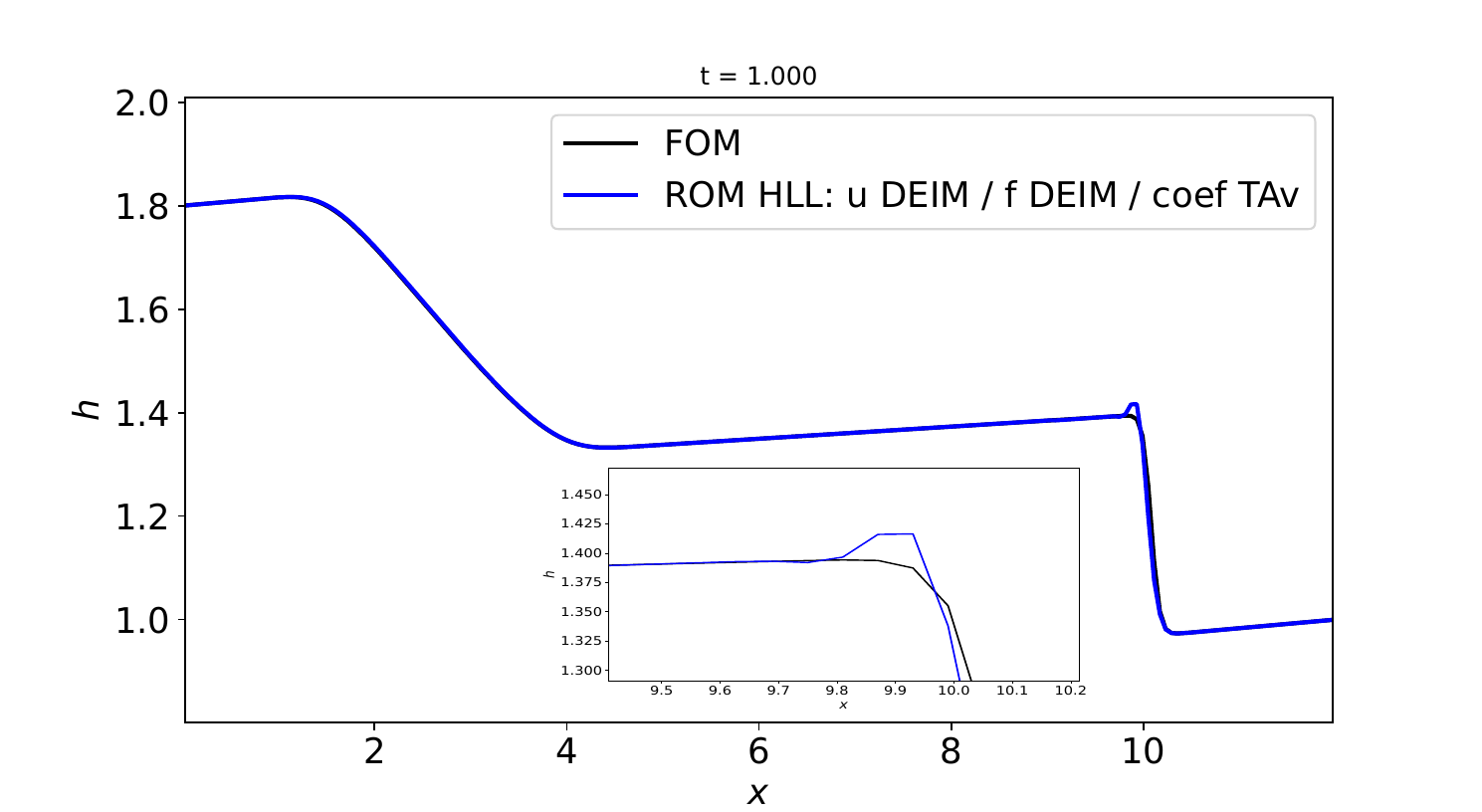}} 
   \subfloat[][ Discharge]{\includegraphics[width=0.49\textwidth]{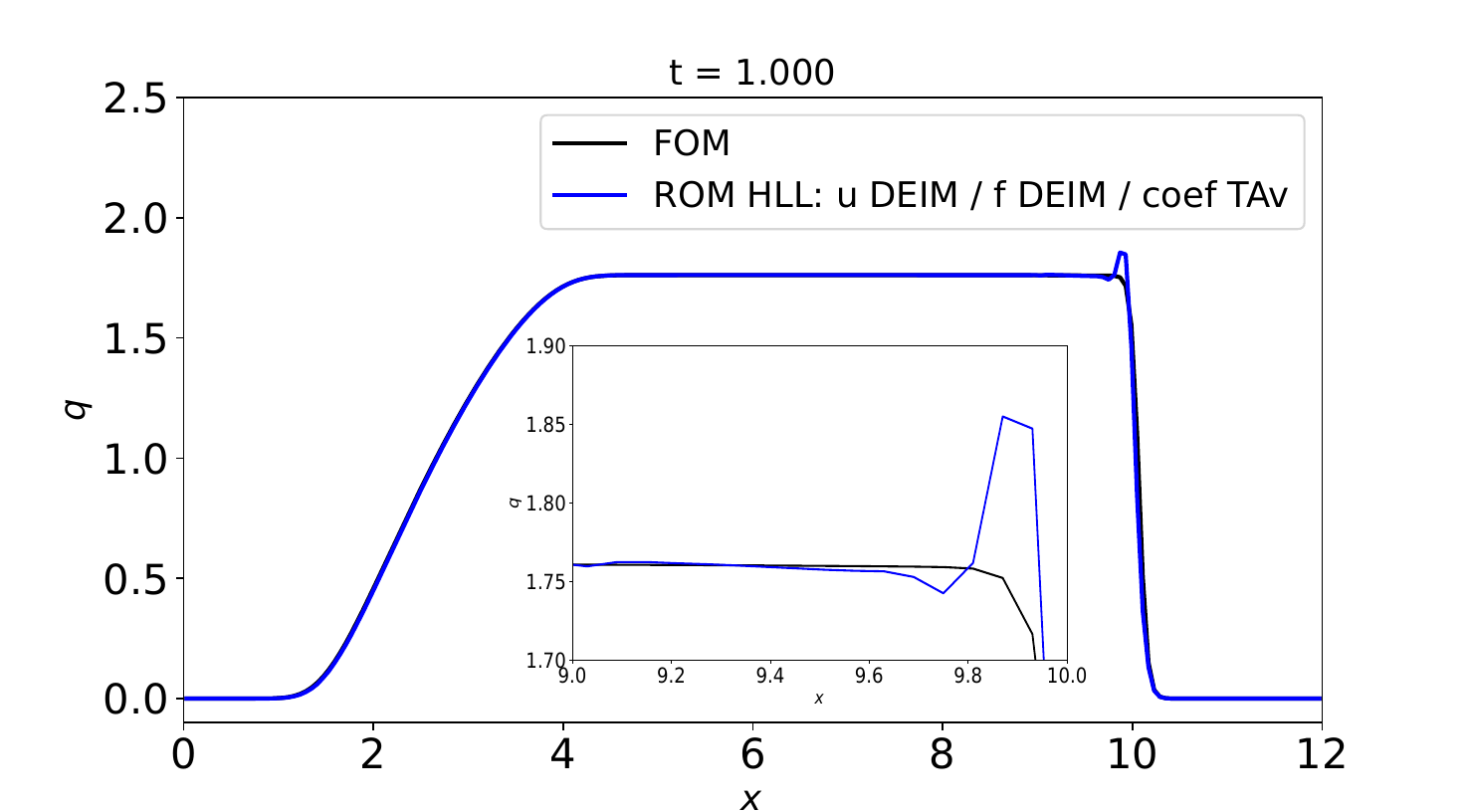}} 
      \caption{Test \ref{test_3}. SW equations: dam-break. HLL numerical flux, DEIM approach for $u$ and $f=|q|/h^{7/3}$ and the time-averaging approach for the coefficients $\alpha_{i+1/2}^{s,n}, \, s=1,2$.  Water depth $h$ (left), discharge $q$ (right) at $T_f=1$ s. Inner figures: zooms.}\label{SW_dambreak_TW5_Fric0c1_HLL_uDEIM_fDEIM_coefTAv}
  \end{center}
 \end{figure}



  Lastly,  the DEIM approach is considered for $u$,  $f=|q|/h^{7/3}$ and the coefficients $\alpha_{i+1/2}^{0,n}$ and $\alpha_{i+1/2}^{1,n}$. Considering the previous choices, the selected number of POD modes is again $M=20$. Figure \ref{SW_dambreak_TW5_Fric0c1_HLL_uDEIM_fDEIM_coefDEIM} (up) shows, from left to right, the height of the water $h$ and the discharge $q$. Spurious oscillations are reduced using DEIM near the front shock. Figure \ref{SW_dambreak_TW5_Fric0c1_HLL_uDEIM_fDEIM_coefDEIM} (down) zooms in near the shock for $q$. 
 \begin{figure}[h]
 \begin{center}
   \subfloat[][ Water depth]{\includegraphics[width=0.49\textwidth]{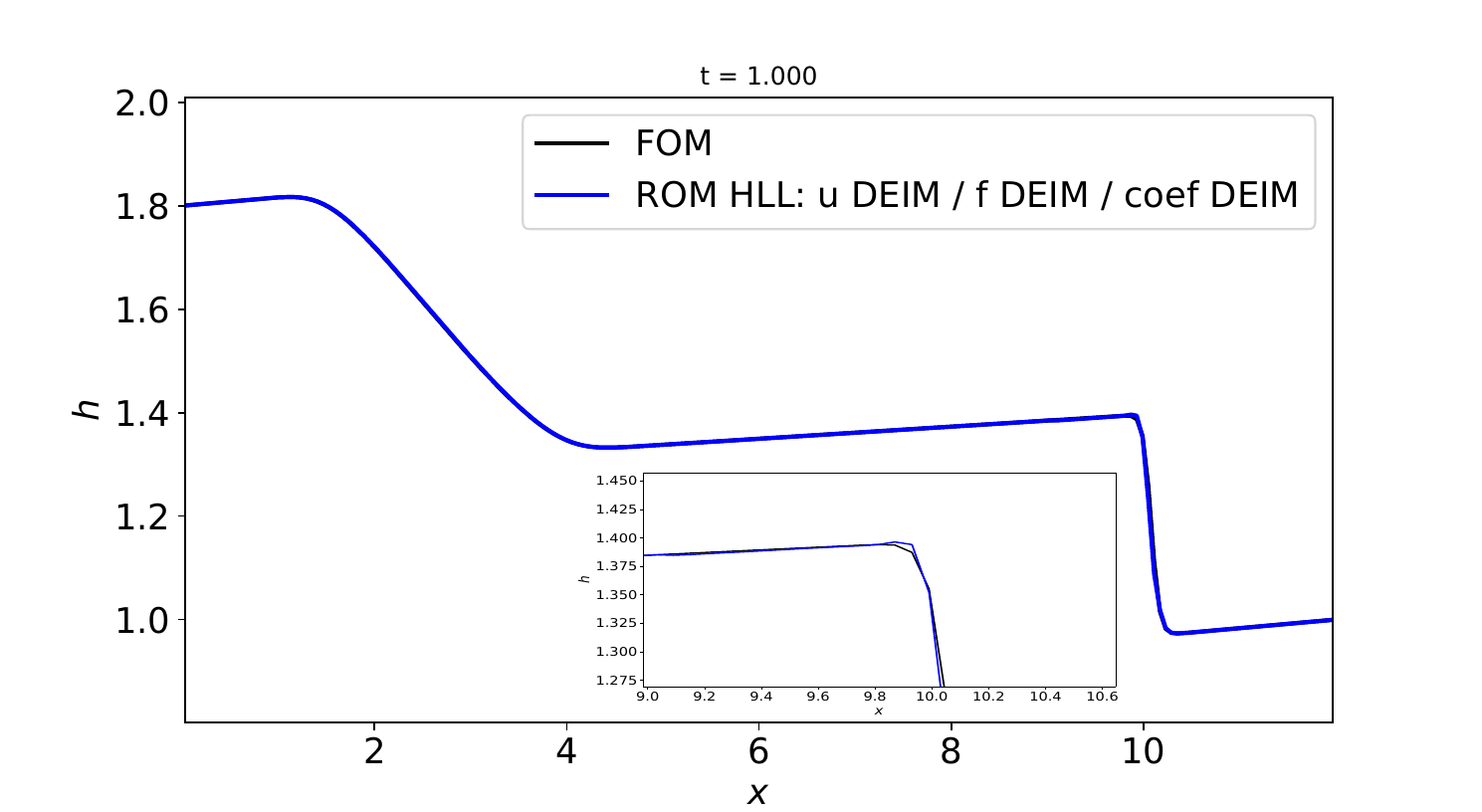}}
   \subfloat[][Discharge]{\includegraphics[width=0.49\textwidth]{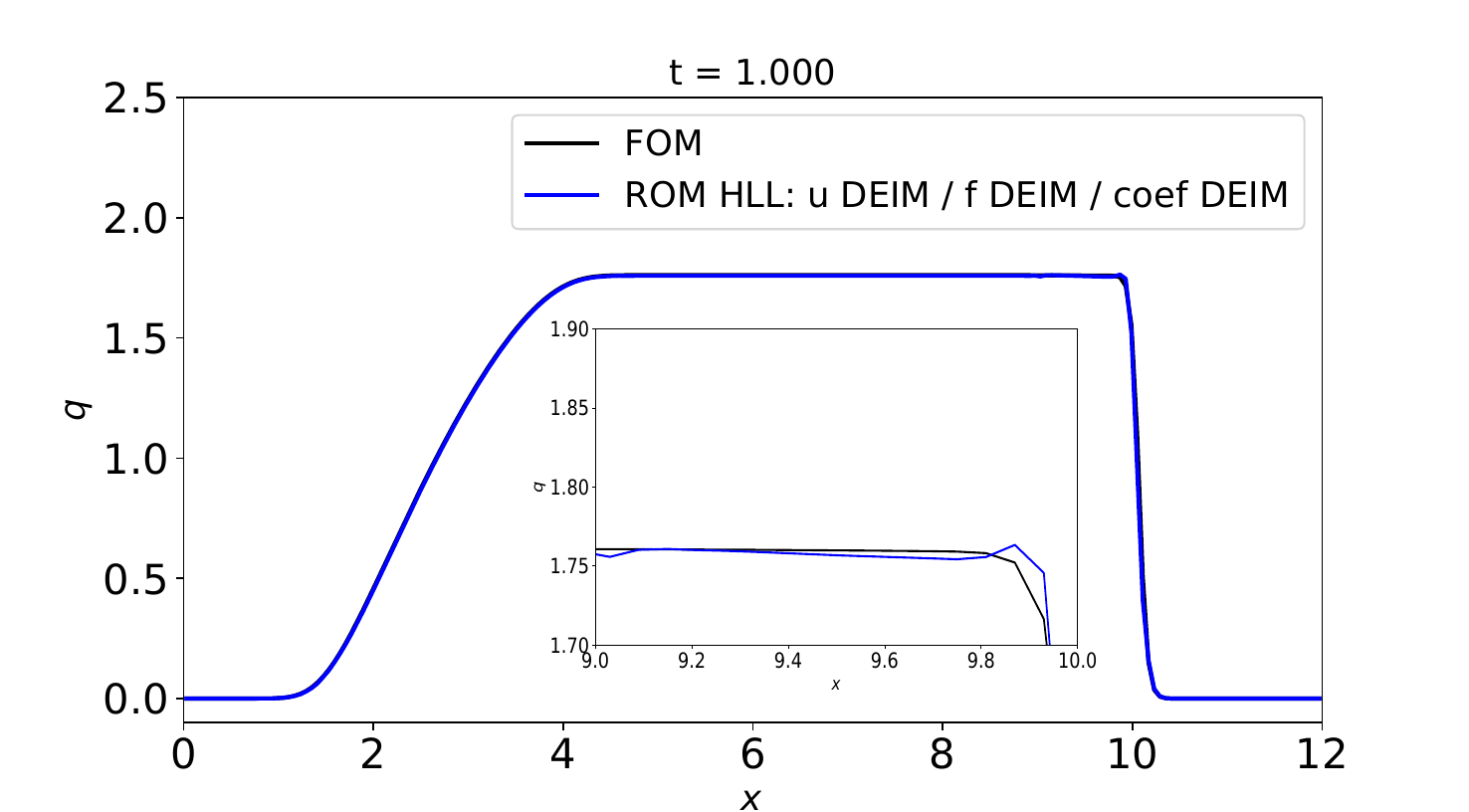}} 
    \caption{Test \ref{test_3}. SW equations: dam-break. HLL numerical flux, DEIM approach for $u$, $f=|q|/h^{7/3}$ and the coefficients $\alpha_{i+1/2}^{s,n}, \, s=1,2$.   Water depth $h$ (left), discharge $q$ (right) at $T_f=1$ s. Inner figures: zooms.}\label{SW_dambreak_TW5_Fric0c1_HLL_uDEIM_fDEIM_coefDEIM}
  \end{center}
 \end{figure}



Errors in $L^1-$norm between the solution of the ROM and a numerical solution computed with the FOM with the parameter $n_b=0.1$ are shown in Table \ref{sw_dambreak_TW5_Fric0c1_HLL_tabla}. It can be observed that by considering DEIM linearization errors are reduced. 
\begin{table}[h]
	\centering
	\begin{tabular}{|c|c|c|c|}
		\hline 
		Lax-Friedrichs & \multicolumn{3}{c|}{Linearization} 
		\\		\hline 
		L$^1$ error   &   
		\begin{tabular}{l}
			$u$ time-averaging \\ $f$ time-averaging 
		\end{tabular} 
		&
		\begin{tabular}{l}
			$u$ DEIM \\ $f$ time-averaging 
		\end{tabular}   
		&
		\begin{tabular}{l}
			$u$ DEIM \\ $f$ DEIM 
		\end{tabular}    
		\\ \hline 
		h &   8.13e-02  &   1.85e-03 & 9.48e-04  \\ \hline 
		q &   5.25e-01  &   1.81e-02 & 9.47e-03
		\\ \hline
	\end{tabular}
	\caption{Test \ref{test_3}. SW equations: dam-break, Lax-Friedrichs numerical flux.  Errors in $L^1-$norm between the solution of the ROM and a numerical solution computed with the FOM with the parameter $n_b=0.1$ for $h$ and $q$ at $T_f=1$ s. by considering time-averaging or DEIM linearization for $u$ or $f=|q|/h^{7/3}$.}
\label{sw_dambreak_TW5_Fric0c1_LF_tabla}
\end{table}

\begin{table}[h]
	\centering
	\begin{tabular}{|c|c|c|}
		\hline 
		HLL & \multicolumn{2}{c|}{Linearization} 
		\\		\hline 
		L$^1$ error  &   
			\begin{tabular}{lc} 
				$u$ DEIM, & ($\alpha_{i+1/2}^{0,n}$,  \, $\alpha_{i+1/2}^{1,n}$) \\
				$f$ DEIM, &   time averaging
			\end{tabular} 
		&
			\begin{tabular}{lc} 
	$u$ DEIM, & ($\alpha_{i+1/2}^{0,n}$,  \, $\alpha_{i+1/2}^{1,n}$) \\
	$f$ DEIM, &   DEIM
\end{tabular}    
		\\ \hline 
		h &   1.52e-02  &  8.33e-03    \\ \hline 
		q &   6.55e-02  & 4.80e-02
		\\ \hline
	\end{tabular}
	\caption{Test \ref{test_3}. SW equations: dam-break, HLL numerical flux.  Errors in $L^1-$norm between the solution of the ROM and a numerical solution computed with the FOM with the parameter $n_b=0.1$ for $h$ and $q$ at $T_f=1$ s. by considering time-averaging or DEIM linearization for $\alpha_{i+1/2}^{0,n}$ and $\alpha_{i+1/2}^{1,n}$ and DEIM for $u$ and $f=|q|/h^{7/3}$.}
	\label{sw_dambreak_TW5_Fric0c1_HLL_tabla}
\end{table}


This experiments shows that the DEIM approach introduced in this work produces better results in the presence of discontinuities than other linearization techniques previously described in the literature such as time-averaging.

\refuno{Table~\ref{sw_dambreak_tab:rom_speedups_lf_hll} summarizes the CPU times and speed-ups for different ROM configurations. For Lax-Friedrichs, full DEIM yields the best accuracy, while, as expected, the full time-averaging configuration is fastest. The mixed approach balances both. A similar pattern appears with HLL: applying DEIM to $u$ and $f$ and time-averaging the coefficients $\left(\alpha_{i+1/2}^{0,n}, \alpha_{i+1/2}^{1,n}\right)$ achieves the highest speed-up, whereas full DEIM improves accuracy slightly. Overall, the results suggest that combining DEIM and time-averaging offers the best trade-off between precision and performance.
}

\begin{table}[h]
\centering
\resizebox{\textwidth}{!}{%
\begin{tabular}{|c|c|c|c|c|c|}
\hline
Scheme & $u$ & $f$ & $\left(\alpha_{i+1/2}^{0,n}, \, \alpha_{i+1/2}^{1,n}\right)$ & \multicolumn{1}{c|}{\begin{tabular}{@{}c@{}}ROM\\CPU time \end{tabular}} & Speed-up \\
\hline
\multirow{3}{*}{Lax-Friedrichs} 
& TAv & TAv & -- & 0.212 & 1.73 \\
& DEIM & TAv & -- & 0.215 & 1.71 \\
& DEIM & DEIM & -- & 0.224 & 1.64 \\
\hline
\multirow{2}{*}{HLL} 
& DEIM & DEIM & TAv & 0.251 & 2.88 \\
& DEIM & DEIM & DEIM & 0.377 & 1.92 \\
\hline
\end{tabular}
}
\caption{\refuno{Test \ref{test_3}. SW equations: dam-break. CPU times and speed-ups for different ROM configurations using the Lax-Friedrichs and HLL schemes. TAv stands for time-averaging. The FOM CPU times are 0.367 s for Lax-Friedrichs and 0.723 s for HLL.}}
\label{sw_dambreak_tab:rom_speedups_lf_hll}
\end{table}

\subsection{Simulation in parameter-dependent systems} \label{test_4}
We consider again the SW system with friction \eqref{swf_ecuacion}. The space interval is $[0,12]$ and the time interval is $[0,1]$. A mesh with 200 cells is taken. Again, the initial condition is the dam-break given by \eqref{dambreak_cini}-\eqref{fondo_dambreak}. Our goal is to build a predictive ROM following Section \ref{sec:prediction}, where the parameter to be predicted is the friction coefficient $n_b$. We apply the modified Lax-Friedrichs  numerical flux with the linearization \eqref{SW_fric_DEIM}-\eqref{SW_fric_DEIM_H} of the friction term. \refuno{It is worth noting that, in this experiment, the cost of constructing the reduced bases from the snapshot matrices (formed by assembling submatrices corresponding to different values of the parameter in the training set) is performed during the offline phase and therefore does not contribute to the computational cost. As a result, the speed-up achieved in this case is expected to be of the same order as that obtained in the previous test, considering the Lax-Friedrichs scheme and the same linearization of $u$ and the friction term.
}
Let's consider the Manning's friction coefficient as $n_b=0.035$. Before building predictive ROMs, we have first evaluated the errors obtained by using the FOM for this value and then deriving the associated ROM. We consider $N_v=25$ time windows and $\varepsilon_{\text{POD}}=$1e-10. Table \ref{swparametrofric_nb0c035_tabla_minimoerror} presents the $L^1-$norm errors when projecting onto the entire space (maximum number of modes) and when using only 5 modes. In both cases, the errors are of the order of $1e-4$. 

\begin{table}[h]
     \centering
     \begin{tabular}{|c|c|c|}
     \hline
     & Projection & 5 POD modes \\ \hline
         $h$ & 1.27e-04 & 1.70e-04 \\
          $q$ & 3.86e-04 & 7.14e-04 \\ \hline
     \end{tabular}
     \caption{Test \ref{test_4}. SW equations with friction. Errors in $L^1-$norm between numerical solution computed with the FOM with the parameter $n_b=0.035$ and the solution of the ROM with the same parameter when projecting onto the entire space (maximum number of modes) and when using only 5 modes at $T_f=1$ s.}
\label{swparametrofric_nb0c035_tabla_minimoerror}
 \end{table}

Now, we have applied the procedure introduced in Section \ref{sec:prediction} to construct different ROMs using bases derived from 4 distinct training sets, ensuring that none of the sets includes parameter $n_b=0.035$. The training sets, labeled as $C_1$, $C_2$, $C_3$ and $C_4$, are the following:
$$
 C_1  = \{0, 1 \}, \,\,
C_2  = \{0.01, 0.05, 0.09 \}, \,\,
C_3  = \{0.03, 0.04 \}, \,\,
C_4  = \{0.07, 0.09 \}. 
$$

Figures \ref{swparametrofric_nb0c035_h} , \ref{swparametrofric_nb0c035_q} and \ref{swparametrofric_nb0c035_dif} display the solutions obtained from the 4 predictive ROMs and the absolute value of the differences between the FOM numerical solution for $n_b=0.035$ and the ROMs for the training sets $C_k, \, k=1, 2, 3, 4$. The $L^1-$errors between the FOM solution and the ones computed with the four ROMs are summarized in Table \ref{swparametrofric_nb0c035_tabla_Ck}. 

Notice that when using $C_1$, which contains extreme values of the friction parameter, we obtain a good approximation of the solution. Furthermore, as expected, the error decreases when the training sets are based on more information around the parameter being predicted $n_b=0.035$, as seen with sets $C_2$ and $C_3$. In fact, with the latter, the error decreases by an entire order of magnitude. Finally, when considering training set $C_4$, which is distant from the parameter being predicted, we observe that although the error increases, the solutions still provide reasonably accurate approximations: the predicted ROM matches well with the FOM considering the exact Manning friction parameter $n_b=0.035$, obtaining small errors. Therefore, in view of the results, the technique for designing ROMs that allows us to make parameter predictions described in Section \ref{sec:prediction} seems to provide accurate approximations in the considered framework.

\begin{figure}[h]
 \begin{center}
   \subfloat[][Water depth]{\includegraphics[width=0.45\textwidth]{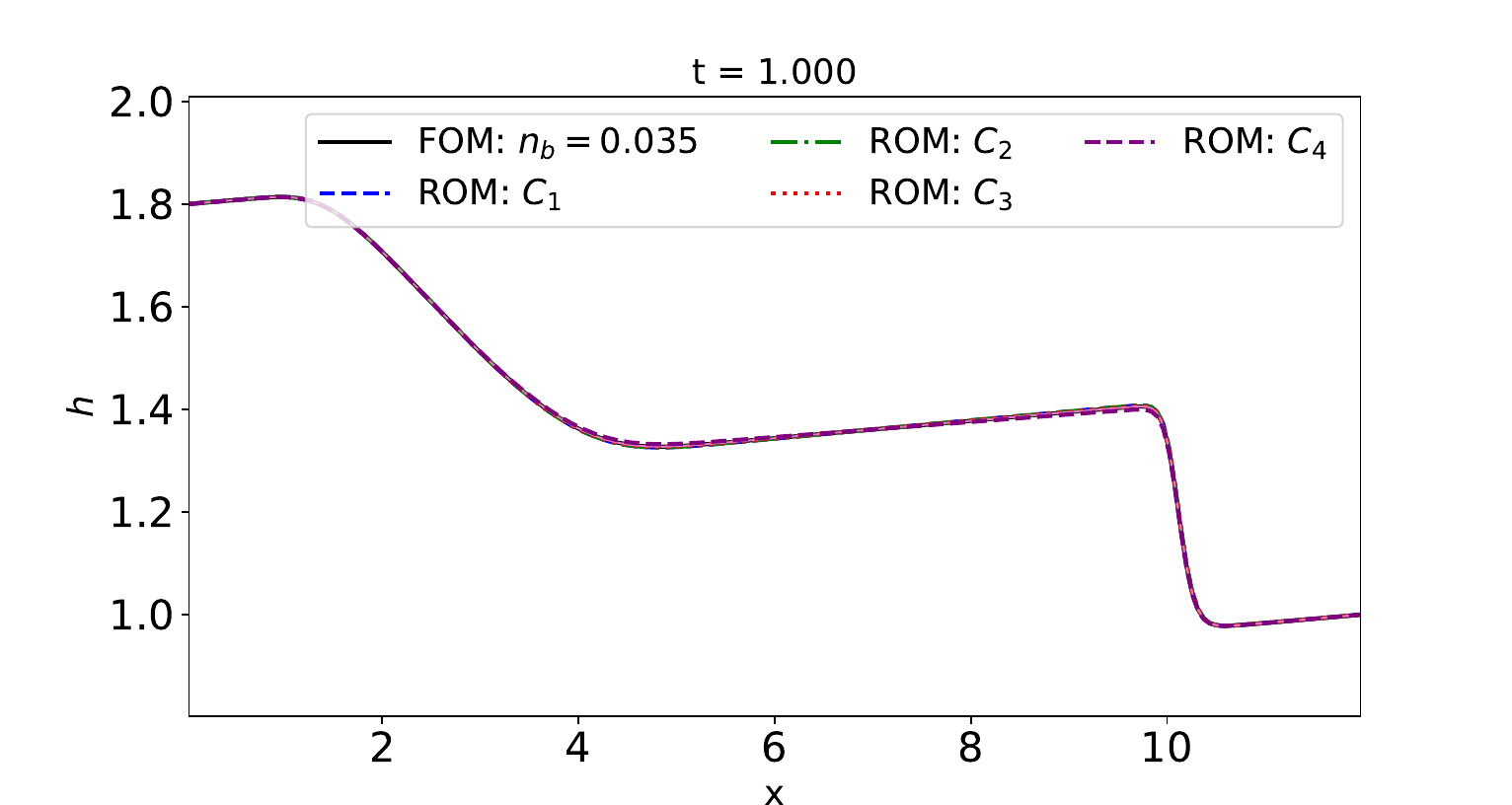}} \vspace{2mm}
   \subfloat[][Water depth: zoom ]{\includegraphics[width=0.45\textwidth]{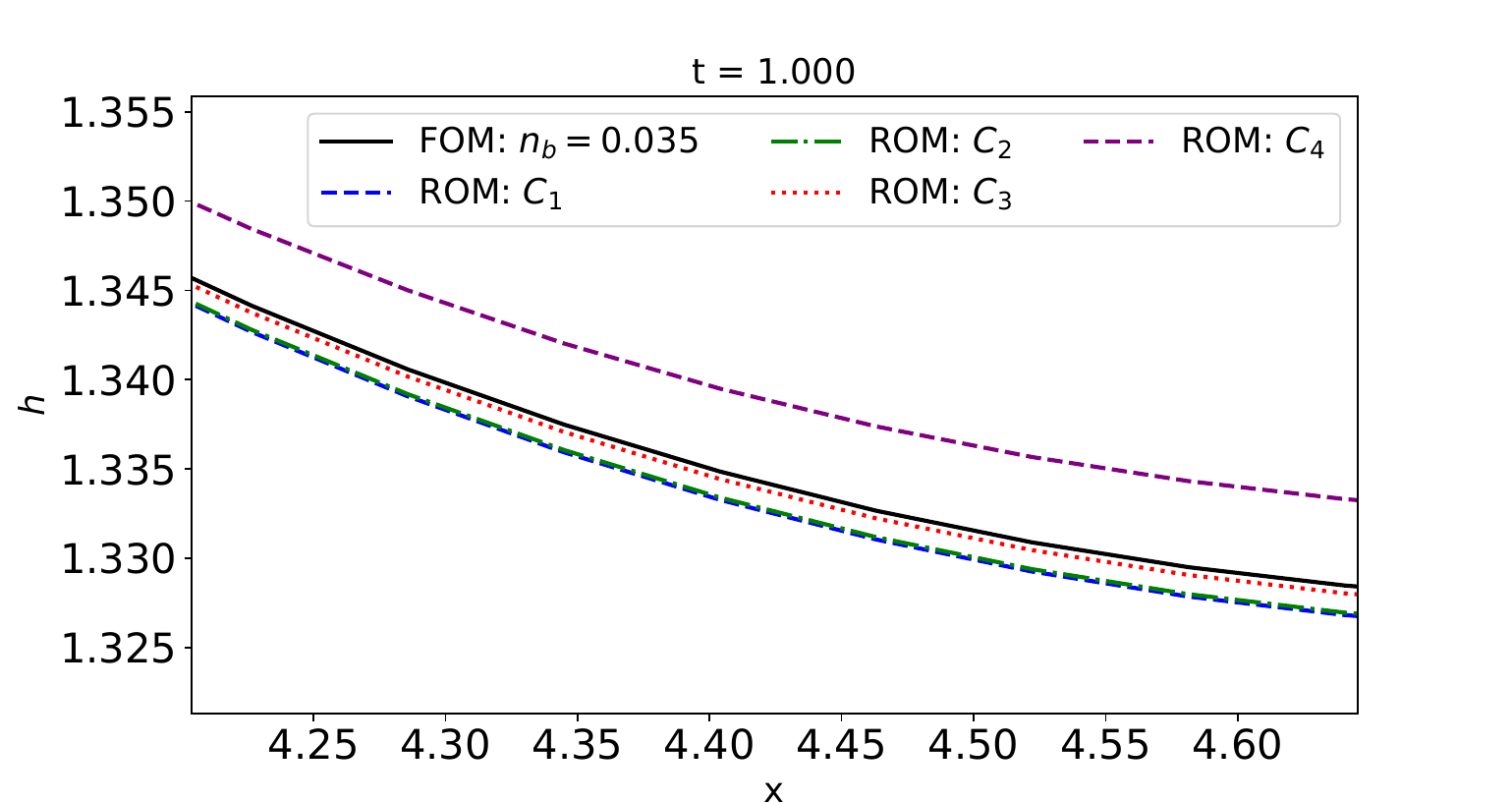}}
     \caption{Test \ref{test_4}. Predictive ROMs for the SW equations with friction. Prediction for $n_b=0.035$ using the training sets $C_k, \, k=1,2,3,4$. Water depth $h$ (left) and zoom in (right) at $T_f=1$ s.} 
\label{swparametrofric_nb0c035_h}
  \end{center}
 \end{figure}

 \begin{figure}[h]
 \begin{center}
   \subfloat[][Water depth]{\includegraphics[width=0.45\textwidth]{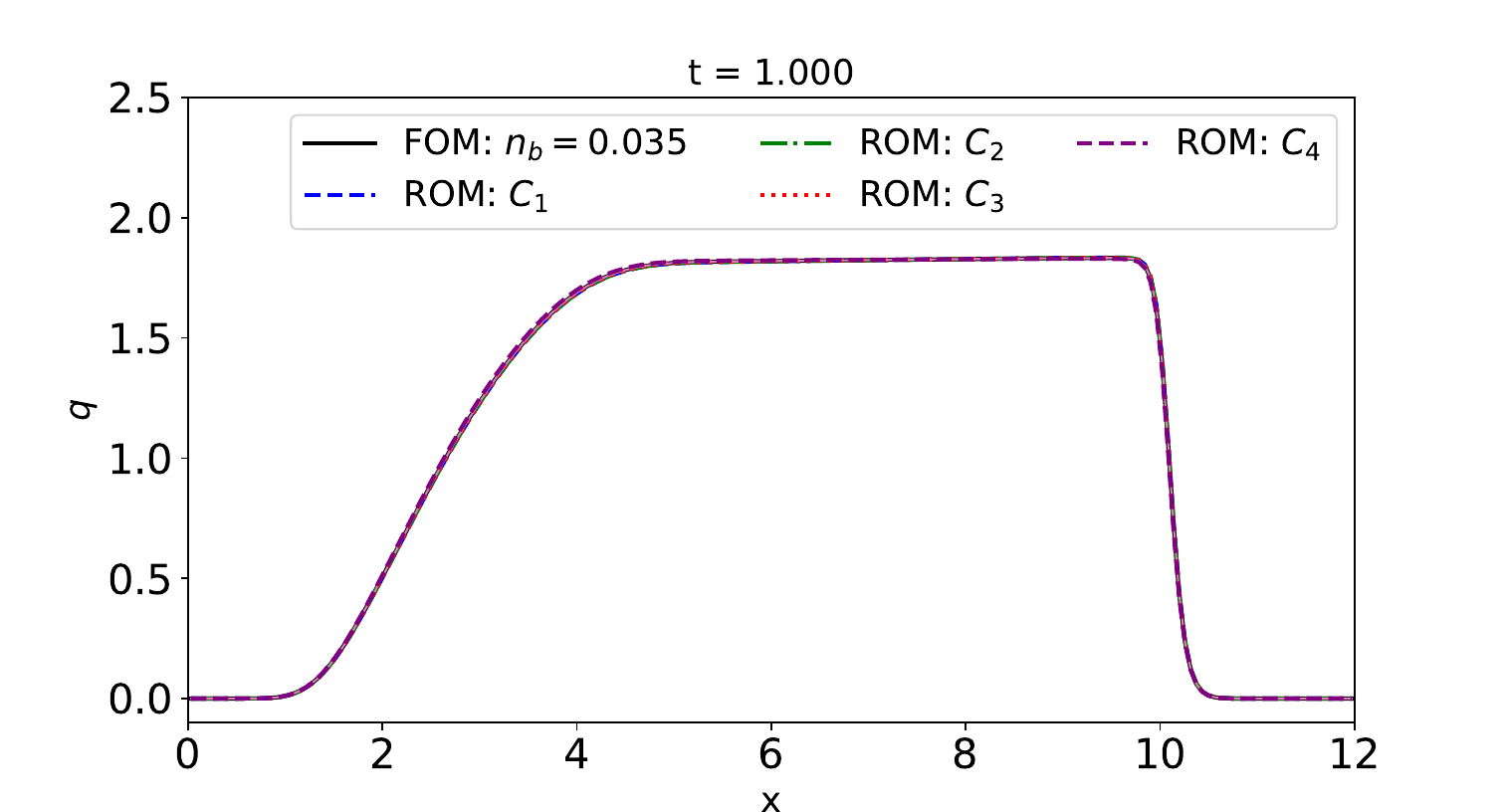}} \vspace{2mm}
   \subfloat[][Water depth: zoom]{\includegraphics[width=0.45\textwidth]{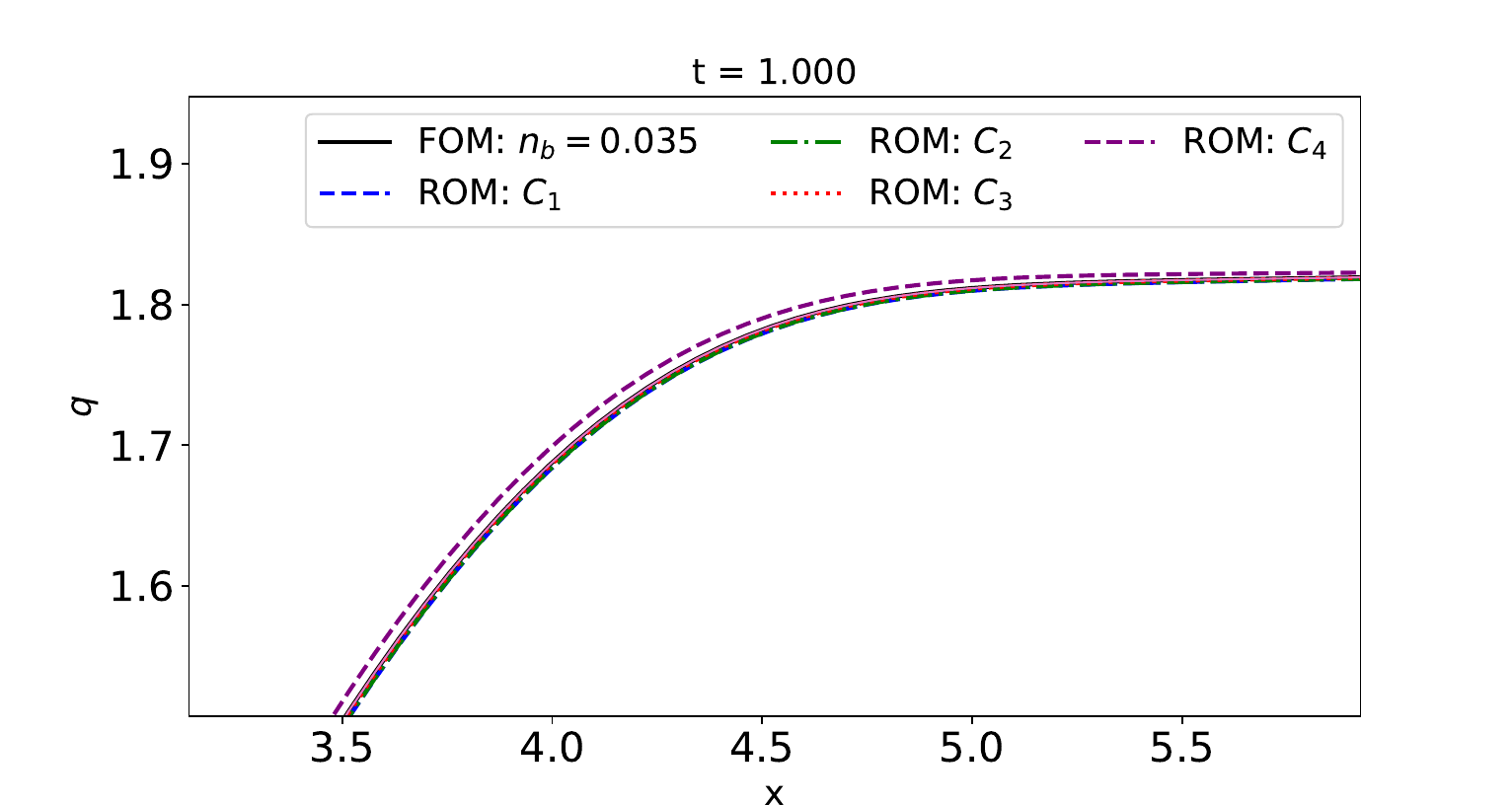}}
     \caption{Test \ref{test_4}. Predictive ROMs for the SW equations with friction. Prediction for $n_b=0.035$ using the training sets $C_k, \, k=1,2,3,4$. Discharge $q$ (left) and zoom in (right) at $T_f=1$ s.} 
\label{swparametrofric_nb0c035_q}
  \end{center}
 \end{figure}
   
 \begin{figure}[h]
 \begin{center}
   \subfloat[][Water depth]{\includegraphics[width=0.45\textwidth]{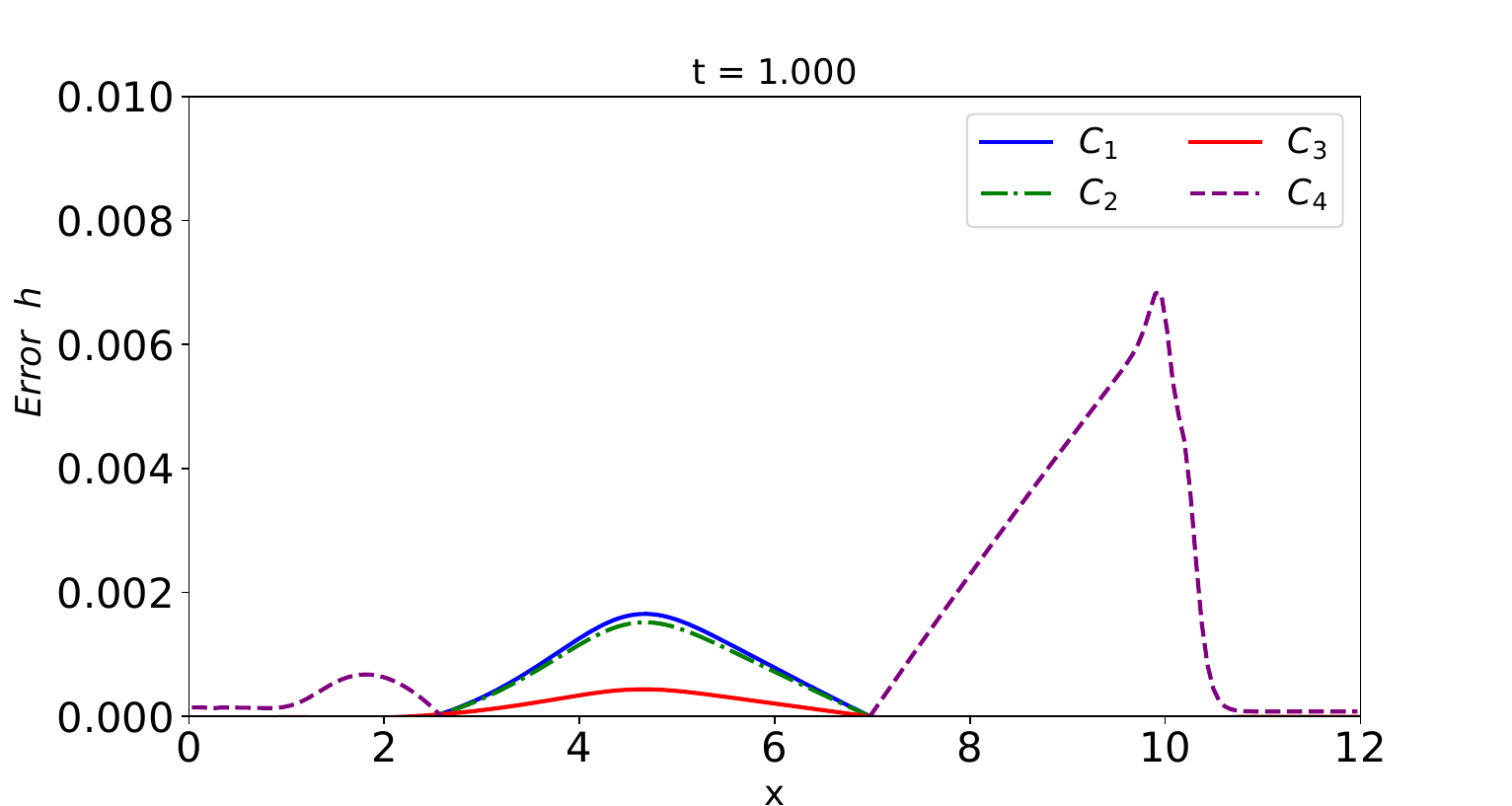}} \vspace{2mm}
   \subfloat[][Discharge]{\includegraphics[width=0.45\textwidth]{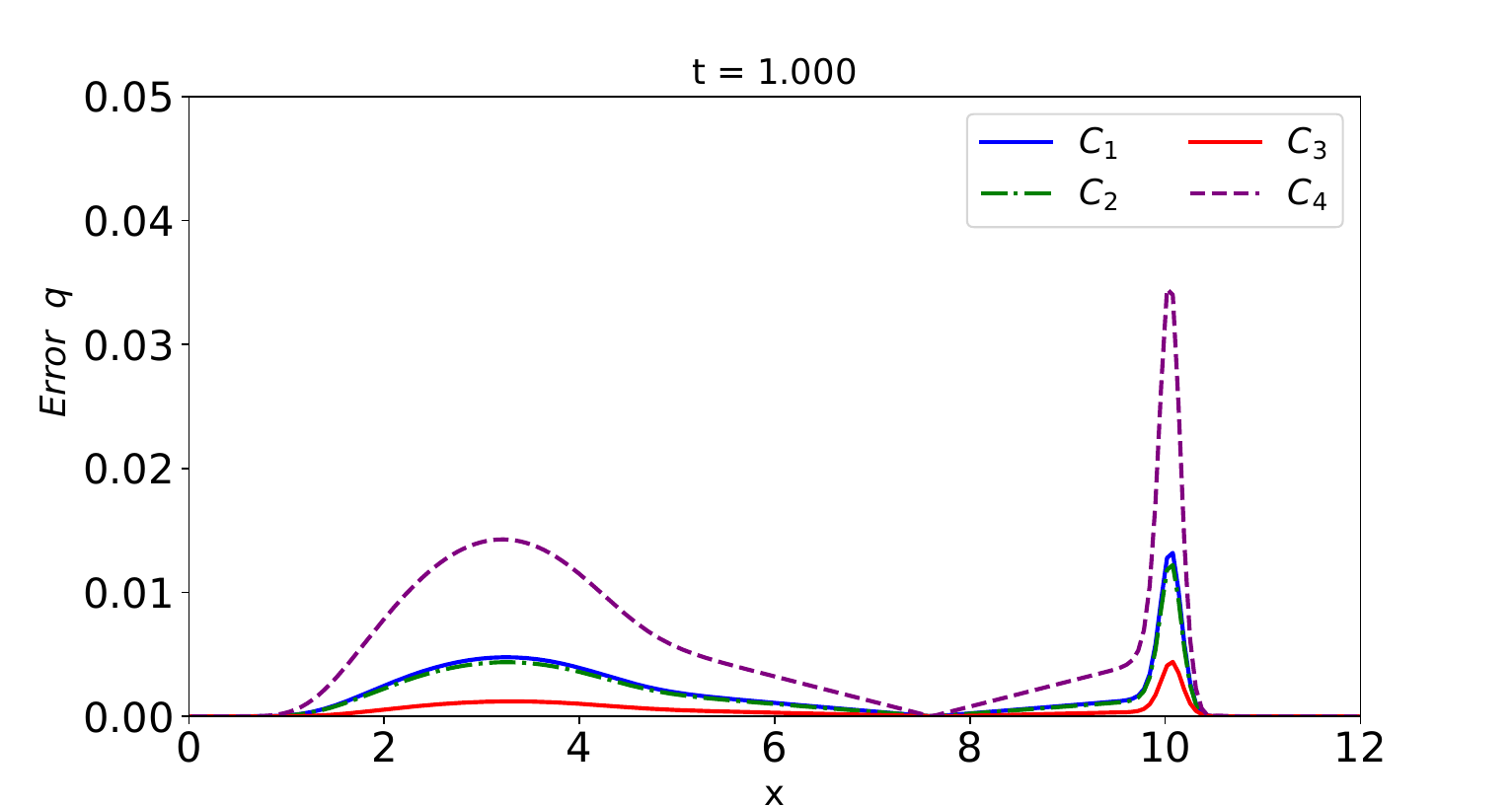}}
     \caption{Test \ref{test_4}. Predictive ROMs for the SW equations with friction. Prediction for $n_b=0.035$ using the training sets $C_k, \, k=1,2,3,4$. Differences in absolute value between the FOM computed with  $n_b=0.035$ and the predictive ROMs. Water depth $h$ (left) and discharge $q$ (right) at $T_f=1$ s.} 
\label{swparametrofric_nb0c035_dif}
  \end{center}
 \end{figure}

  \begin{table}[h]
     \centering
     \begin{tabular}{|c|c|c|c|c|}
     \hline
     & $C_1$ & $C_2$ & $C_3$ & $C_4$ \\ \hline
         $h$ & 8.56e-03 & 7.87e-03 & 2.33e-03 & 2.49e-02 \\
          $q$ & 2.09e-02 &  1.92e-03 & 5.63e-03& 6.17e-02\\ \hline
     \end{tabular}
     \caption{Test \ref{test_4}. Predictive ROMs for the SW equations with friction. Prediction for $n_b=0.035$ using the training sets $C_k, \, k=1,2,3,4$. Errors in $L^1-$norm between the numerical solution computed with the FOM with the parameter $n_b=0.035$ and the solution of the predictive ROMs at $T_f=1$ s.}
     \label{swparametrofric_nb0c035_tabla_Ck}
 \end{table}

 \section{Conclusions} \label{sec:conclusiones}
This work addresses the resolution of hyperbolic systems of balance laws by implementing POD-ROMs. The use of these methods has proven to be computationally efficient, offering significantly less expensive solutions compared to traditional FOMs based on FV schemes. However, since these systems of partial differential equations often include nonlinear terms, dealing with them using ROMs poses a challenge. Appropriate techniques are required to handle these nonlinearities while maintaining accuracy.

Previous works, such as \cite{solan2023development}, have applied the PID approach in combination with time averaging techniques to address the nonlinearities when deriving ROMs. In this work, we introduce an alternative approach based on DEIM to overcome this challenge, once again in combination with the PID approach. The influence of the number of modes and time windows on the results in terms of error and computational time has been investigated through several numerical experiments involving the transport equations and Burgers' equation, both with source terms. Additionally, we compared the time-averaging strategy with the DEIM-based linearization strategy for the SW system with Manning friction. The results demonstrate that the novel DEIM-based technique presented in this article is more accurate and reduces spurious oscillations near discontinuities.

A key finding of this study is the demonstration of an important result concerning the well-balanced property of ROMs. A theorem is introduced, showing that if a ROM is derived from a exactly well-balanced FOM, the ROM inherits this property. To the best of our knowledge, it is the first time that a result on the well-balancedness of the reduced order models has been proved. The key demonstration is that if the original scheme is well-balanced, and the initial conditions are cell-averages of a stationary solution, the snapshot matrix contains only one nonzero eigenvector, meaning the ROM is fully represented by a single POD mode. Numerically, the well-balanced property has been confirmed in various hyperbolic problems, consistently yielding machine precision errors in the preservation of steady-state solutions.

Finally, given the computational advantages of ROMs and their application to hyperbolic systems dependent on certain physical parameters, this work presents a methodology for constructing predictive ROMs. These models are designed using a training set that incorporates different values of the parameter, aiming to deliver fast and accurate results for new parameter values not included in the training set. A sensitivity analysis of this methodology was conducted for the shallow water system, considering the Manning’s friction as the parameter to be estimated, and the results obtained are highly promising.

As possible future work, it would be interesting to apply these approach to other geophysical fluids and also
extend it to the two-dimensional case.

\section*{Acknowledgments}
This work is partially supported by the Grant QUALIFICA (programa: ayudas a acciones complementarias de I+D+i) by Junta de Andalucía grant number QUAL21 005 USE.  This research has also been  partially supported by Junta de Andalucía research project ProyExcel\_00525,  by the European Union - NextGenerationEU program and by grants PID2022-137637NB-C21 and PID2022-137637NB-C22 funded by \\ MCIN/AEI/10.13039/501100011033 and ``ERDF A way of making Europe'' and by the Spanish Government Project PID2021-123153OB-C21, funded by MCIN/AEI/10.13039/501100011033, UE..

\FloatBarrier

\end{document}